%% file: main.tex
\documentclass[reqno, 11pt]{amsart}
\usepackage[utf8]{inputenc}
\usepackage{amsmath}
\usepackage{amsfonts, amsthm, mathtools, commath, amssymb, mathrsfs, appendix, enumerate, setspace}
\usepackage{pgf, tikz-cd, stackrel, graphicx, enumerate, verbatim, stmaryrd}
\usepackage[hidelinks]{hyperref}
\usepackage[style=alphabetic]{biblatex}
\usepackage[left=1.15in, right=1.15in, top=1.3in, bottom=1.3in]{geometry}
\let\amp=&
\theoremstyle{plain}
\newtheorem*{theorem*}{Theorem}
\newtheorem{innercustomthm}{Theorem}
\newenvironment{customthm}[1]
{\renewcommand\theinnercustomthm{#1}\innercustomthm}
{\endinnercustomthm}
\newtheorem{theorem}{Theorem}[section]
\newtheorem{lemma}[theorem]{Lemma}
\newtheorem{prop}[theorem]{Proposition}
\newtheorem{cor}[theorem]{Corollary}
\theoremstyle{definition}

\newtheorem{defn}[theorem]{Definition}

\newtheorem{ass}[theorem]{Assumption}
\theoremstyle{remark}

\newtheorem{rem}[theorem]{Remark}

\newcommand{\A}[2]{\mathbb{A}_{#1}^{#2}}
\newcommand{\Q}{\mathbb{Q}}
\newcommand{\Z}{\mathbb{Z}}
\newcommand{\N}{\mathbb{N}}
\newcommand{\C}[1]{\mathbb{C}^{#1}}
\newcommand{\K}[1]{\mathbb{P}^{#1}}
\newcommand{\R}{\mathbb{R}}
\newcommand{\Hs}[1]{\mathrm{H}^{\sharp}_{\mathcal{#1}}}
\newcommand{\Nm}{\mathrm{N}}
\newcommand{\F}[1]{\mathbb{F}_{#1}}

\newcommand{\V}[1]{\mathbb{V}^{#1}_{0}}
\newcommand{\catname}[1]{\normalfont\mathbf{#1}}

\newcommand{\GL}[1]{\catname{GL}_{#1}}
\newcommand{\SL}[1]{\catname{SL}_{#1}}

\newcommand{\Gm}{\mathbb{G}_m}
\newcommand{\Ga}{\mathbb{G}_a}
\newcommand{\w}[1]{\mathfrak{w}^{#1}_{k, \alpha, I}}
\newcommand{\wc}[1]{\mathfrak{w}_{#1}}
\newcommand{\W}[1]{\mathbb{W}^{#1}_{k,\alpha,I}}
\newcommand{\Wc}[1]{\mathbb{W}_{#1}}
\newcommand{\U}[1]{\mathbf{U}_{\mathfrak{#1}}}
\newcommand{\Vp}[1]{\mathbf{V}_{\mathfrak{#1}}}
\newcommand{\st}[2]{\breve{#1}^{#2}}
\newcommand{\Mbar}[2]{\bar{\mathfrak{M}}^{#2}_{#1,\alpha,I}}
\newcommand{\Man}[1]{\bar{\Sh{M}}_{#1,\alpha,I}}
\newcommand{\M}[2]{\mathfrak{M}^{#2}_{#1,\alpha,I}}
\newcommand{\X}{\mathfrak{X}_{\alpha, I}}

\newcommand{\Xr}[1]{\mathfrak{X}_{#1,I}}
\newcommand{\Xra}[1]{\mathfrak{X}_{#1, \alpha, I}}

\newcommand{\xra}[1]{\mathcal{X}_{#1,\alpha,I}}

\newcommand{\Ig}[2]{\mathfrak{IG}_{#1, #2, I}}

\newcommand{\Igord}[1]{\mathfrak{IG}^{\textrm{ord}}_{#1, I}}

\newcommand{\Ha}{\mathrm{Ha}}
\newcommand{\Hdg}[1]{\mathrm{Hdg}^{#1}}
\newcommand{\Gal}[1]{\mathrm{Gal}(\bar{#1}/#1)}
\newcommand{\sh}[1]{\mathscr{#1}}

\newcommand{\Frob}{\textrm{Frob}}
\newcommand{\Sh}[1]{\mathcal{#1}}
\newcommand{\Tate}[2]{\textrm{Tate}_{\mathfrak{#1},\mathfrak{#2}}(q)}

\newcommand{\Sp}{\textrm{Spec}\hspace{0.1em}}

\newcommand{\Spf}{\textrm{Spf}\hspace{0.2em}}

\newcommand{\D}{\mathrm{d}}

\newcommand{\Fil}[1]{\mathrm{Fil}_{#1}}
\newcommand{\Gr}[1]{\mathrm{Gr}_{#1}}
\mathchardef\mhyphen="2D
\DeclareMathOperator{\coker}{coker}

\DeclareMathOperator{\im}{im}
\DeclareMathOperator{\Hom}{Hom}
\DeclareMathOperator{\End}{End}

\DeclareMathOperator{\Spa}{Spa}
\DeclareMathOperator{\dlog}{dlog}
\DeclareMathOperator{\Res}{Res}
\DeclareMathOperator{\Sym}{Sym}
\DeclareMathOperator{\Der}{Der}

\title[Twisted Triple Product $p$-adic $L$-function for Finite Slope Families]{Twisted Triple Product \texorpdfstring{$p$}{p}-adic \texorpdfstring{$L$}{L}-function for Finite Slope Families of Hilbert Modular Forms}
\author{Ananyo Kazi}
\date{January 2024}
\address{Department of Mathematics and Statistics, Concordia University, Montreal, Quebec H3G 1M8, Canada}

\begin{document}

\begin{abstract}
    Let $L$ be a totally real field, and $p$ be a rational prime that is unramified in $L$. We construct overconvergent families of  classes of relative de Rham cohomology of the universal abelian scheme over Hilbert modular varieties associated to $L$. We show that these classes come equipped with Gauss--Manin connection. We prove convergence for $p$-adic iteration of this connection, improving upon a technique due to Andreatta--Iovita. We use this to construct a $p$-adic twisted triple product $L$-function associated to finite slope families of Hilbert modular forms, extending work of Blanco-Chacon--Fornea for Hida families.
\end{abstract}

\maketitle

\tableofcontents
\input{Introduction}
\input{Hilbert1}
\input{Hilbert2}
\input{Hilbert3}

\input{Hilbert4}
\input{Hecke}
\input{L-function}

\printbibliography

\end{document}

%% file: Introduction.tex
\section{Introduction}
The theory of $p$-adic $L$-functions has been an important object of study due to its numerous arithmetic applications, most notably towards the Birch and Swinnerton--Dyer conjecture and its generalizations. Katz's $p$-adic Kronecker limit formula \cite{KatzEis} relates special values of the two variable $p$-adic $L$-function associated to a quadratic imaginary field at finite order Hecke characters to $p$-adic logarithms of elliptic units. The work of Bertolini--Darmon--Prasanna \cite{BDP} relates the central critical values of a certain $p$-adic Rankin $L$-function associated to a cusp form $f$ and a quadratic imaginary field $K$ to $p$-adic Abel--Jacobi images of {generalized Heegner cycles}. The article of Darmon and Rotger \cite{DarmonRotger} proves a $p$-adic Gross--Zagier formula relating the special values of the $p$-adic Garrett--Rankin triple product $L$-function attached to a triple of Hida families of modular forms to $p$-adic Abel--Jacobi images of certain {generalized Gross--Kudla--Schoen cycles} in the product of three Kuga--Sato varieties. 

The general technique of constructing these $p$-adic $L$-functions consists of two ingredients: firstly, a theory of $p$-adic modular forms and secondly, a way to $p$-adically iterate differential operators on the space of modular forms. For example, in the case of triple product $L$-functions  due to the work of Harris--Kudla \cite{HarrisKudla}, Ichino \cite{Ichino} and others, one expects that the special values of classical $L$-functions that one wishes to $p$-adically interpolate is a meaningful algebraic number upto multiplication by a transcendental period. More precisely, one expects the algebraic number to be expressed as the square of a Petersson inner product of nearly holomorphic modular forms, possibly arising as the image of a holomorphic modular form under the Shimura--Maass operator. Working with Hida families, as in the case of \cite{DarmonRotger}, the $p$-adic analogue of the Shimura--Maass operator is the $\theta$-operator of Serre, which acts as $\theta = q\frac{\D}{\D q}$ on $q$-expansions. The analogy between the Shimura--Maass operator $\delta$ and Serre's $\theta$ operator rests on the fact that they both can be viewed as the application of the Gauss--Manin connection followed by a projection to the space of modular forms. In the case of $\delta$ the projection comes from the splitting induced by the Hodge decomposition which is a special property of Kahler manifolds and in the case of $\theta$ this is given by the unit root splitting, which is a uniquely $p$-adic phenomenon.

It is then natural to adapt these techniques for finite slope families. However, for finite slope families, the finite slope projector does not converge on $p$-adic modular forms, since the $U$ operator is not compact. One can remedy this by working with overconvergent families, following the work of Coleman \cite{Coleman1}, \cite{Coleman2}. However, a separate problem of working with overconvergent families is the unavailability of the unit root splitting. One approach to solve this problem is to instead $p$-adically iterate the Gauss--Manin connection itself. This is the technique employed in the recent work of Andreatta--Iovita \cite{andreatta2021triple}, where they construct triple product $p$-adic $L$-function attached to finite slope families. Moreover, as a consequence of working beyond the ordinary locus, they manage to construct Katz--BDP type $p$-adic $L$-function in the case $p$ is non-split in the quadratic imaginary field \cite{andreattaKatz}. 

In this paper, we construct a twisted triple product $p$-adic $L$-function associated to finite slope families of Hilbert modular forms, in the case when $p$ is unramified in the totally real fields. This generalizes earlier work of Blanco-Chacon and Fornea \cite{michele} for the case of Hida families of Hilbert modular forms. Let us now introduce the relevant actors to state the theorem.

Fix a totally real field $L$ of degree $g$ over $\Q$. Let $\mathfrak{d}$ be the different ideal of $L/\Q$. Fix an integer $N \geq 4$ and a prime $p \nmid N$ that is unramified in $L$. Fix a finite unramified extension $K$ of $\Q_p$ that splits $L$. Assume for simplicity $p > 2$. Let $\Sigma$ be the set of embeddings of $L$ in $K$.

The notion of Hilbert modular forms has a slight ambiguity in the literature. Namely, one can talk either about geometric Hilbert modular forms, or about arithmetic Hilbert modular forms. This discrepancy arises from the fact that the Shimura variety associated to the group $G := \Res_{\Sh{O}_L/\Z}\GL{2}$ at the usual principal $N$ level or $\mu_N$ level doesn't have an interpretation as a fine moduli of abelian varieties. Instead its connected components, which are parametrized by elements in the strict class group are finite \'{e}tale quotients of fine moduli spaces of abelian varieties with polarization data and level structure. The geometric Hilbert modular forms are the sections of modular sheaves on the moduli scheme of abelian varieties, whereas the arithmetic Hilbert modular forms are the automorphic forms on the Shimura variety associated to $G$. The arithmetic Hilbert modular forms are necessary for a good Hecke theory. The geometric modular forms, as the name suggests are much more suitable for use in geometric constructions. For the purpose of this introduction we will ignore this discrepancy.

Let $\mathfrak{W} = \Spf \Sh{O}_K\llbracket (\Sh{O}_L \otimes \Z_p)^{\times} \rrbracket = \Spf \Lambda$ be the formal weight space and let $\mathfrak{W}^0 = \Spf \Lambda^0$ be the connected component of the trivial character. $\Lambda^0$ is a local ring and let $\mathfrak{m}$ be its maximal ideal. Let $\Sh{W}^0$ be the adic analytic fibre of $\mathfrak{W}^0$, and let $\Sh{W}^0_p$ be the affinoid open where $|t| \leq |p| \neq 0$ for all $t \in \mathfrak{m}$. Let $\mathfrak{W}^0_p = \Spf \Sh{O}^+_{\Sh{W}^0_p}$. Let $k^0$ be the universal weight on $\mathfrak{W}^0_p$, which is known to be analytic on $1+p(\Sh{O}_L \otimes \Z_p)$ (Lemma \ref{L21}). Due to analyticity, $k^0$ extends to a character on $1 + p^n\Res_{\Sh{O}_L/\Z}\Ga \simeq \prod_{\sigma \in \Sigma} 1 + p^n\Ga$ for any $n\geq 1$. Fix one such $n$. Thus we may view $k^0$ as a product of characters $k^0 = \prod_{\sigma \in \Sigma} k^0_{\sigma}$ on $1+p^n\Res_{\Sh{O}_L/\Z}\Ga$.

Consider the $p$-adic completion of $M(\mu_N, \mathfrak{c})$ and let $\mathfrak{X}$ be its base change to $\mathfrak{W}^0_p$. Consider the blow-up spaces $\mathfrak{X}_r$ obtained by blowing up the $\Hdg{}$ ideal and taking the open where the inverse image ideal is generated by $\Hdg{p^{r+1}}$. Here $\Hdg{}$ is the ideal generated locally by lifts of the Hasse invariant and $p$. These blow-up spaces are formal models for the rigid analytic overconvergent locus where $|\Hdg{p^{r+1}}| \geq |p|$. The main results of this work are the following.

\begin{customthm}{A}\label{1}
For suitable choice of $r$, there are interpolation sheaves $\mathfrak{w}^0_k$ and $\mathbb{W}^0_k$ on $\mathfrak{X}_r$, that interpolates modular forms and symmetric powers of de Rham classes for weight $k^0$ respectively. The sheaf $\mathbb{W}_{k}^0$ on $\mathfrak{X}_{r}$ is equipped with an increasing filtration $\{\Fil{i}\}_{i\geq 0}$. The filtered pieces are locally free $\Sh{O}_{\mathfrak{X}_{r}}$-modules and $\mathbb{W}_{k}^0$ is the completed colimit of $\Fil{i}\mathbb{W}_{k}^0$. The zeroth filtered piece $\Fil{0}\mathbb{W}_{k}^0 = \mathfrak{w}_{k}^0$ coincides with the modular sheaf of weight $k^0$. \emph{[\S\ref{S3}]}
\end{customthm}

\begin{customthm}{B}\label{2}
The Gauss--Manin connection on $H^1_{\textrm{dR}}(\Sh{A})$ induces a connection $\nabla$ on $\mathbb{W}^0_k$ over the generic fibre that satisfies Griffiths' transversality with respect to the filtration mentioned above. Moreover, let $k = \prod k_{\sigma} = \prod \exp(u_{\sigma}\log(\cdot))$ and $s = \prod s_{\sigma} = \prod \exp(v_{\sigma}\log(\cdot))$ be two analytic weights such that $u_{\sigma}, v_{\sigma} \in \Sh{O}^+_{\Sh{W}^0_p}$. Then for any $g$ of weight $k$, with $U_{\mathfrak{P}}(g) = 0$ for every prime $\mathfrak{P}$ in $\Sh{O}_L$ lying over $p$, $\nabla^s(g)$ makes sense as an overconvergent de Rham class of weight $k+2s$. \emph{[\S\ref{2S4}]}
\end{customthm}

To talk about $p$-adic iteration of the Gauss--Manin connection, one needs to consider families of modular forms for $p$-adically varying weights, as well as families of de Rham classes for varying weights. The first object, i.e. the sheaf of overconvergent modular forms has been geometrically constructed and studied for quite some time \cite{andreattaSiegel}, \cite{Andreatta2016Hilbert}, \cite{Andreatta2018leHS}. The novelty of \cite{andreatta2021triple} has been the construction of a sheaf $\mathbb{W}_k$ that interpolates symmetric powers of $H^1_{\textrm{dR}}$ of the universal elliptic curve for analytic weights $k$, and to iterate the Gauss--Manin connection for $p$-adic weights. This construction is based on the theory of \emph{vector bundles with marked sections}. This approach has been used fruitfully by Graziani \cite{graziani} to define interpolation sheaves of de Rham classes. This was used by Aycock \cite{Aycock} to interpolate the Gauss--Manin connection $\nabla$ in the setting of Hilbert modular forms, following \cite{harron-xiao} and \cite{liu} in the elliptic and Siegel setting respectively. However he didn't address the problem of iteration of the connection.

The article \cite{Molina} of Molina marks a significant improvement on this technique. One crucial step in the construction of the $p$-adic iteration of $\nabla$ is the proof of its convergence. In \cite{andreatta2021triple}, the authors need to carry out extremely long and complicated computations to prove this. However, in \cite{Molina}, Molina uses a refined version of vector bundles with marked sections to simplify the computations to a large extent. The key idea of his work relies on \cite[Lemma 5.1]{Molina} which proves that the modified integral model $\Sh{H}^{\sharp}_0$ of the de Rham sheaf that one uses to define the interpolation sheaf $\mathbb{W}$, admits a splitting modulo a small power of $p$. Using this one can restrict to a certain well-defined open subspace in the adic geometric vector bundle with marked sections associated to  $\Sh{H}^{\sharp}_0$, such that the Gauss--Manin connection converges faster on the sections of this open subspace. With this the author has been able to construct triple product $p$-adic $L$-functions associated with families of quarternionic automorphic forms in Shimura curves over totally real fields in the finite slope situation.

The main idea of our work is similar to that of \cite{Molina}, in the sense that we use a refined version of vector bundles with marked sections, using a canonical splitting of our integral model of de Rham sheaf $\Hs{A}$ modulo a small power of $p$, and prove that the partial Gauss--Manin connection $\nabla(\sigma)$ associated to any embedding $\sigma \colon L \to \bar{\Q}_p$ of our fixed totally real field $L$ converges fast enough, and that $\nabla(\sigma)$ commutes with $\nabla(\tau)$ for $\sigma \neq \tau$. In the following, we will describe this in more detail. 

Let us describe our approach to proving Theorem \ref{1}. The theory of canonical subgroup tells us that there exists a canonical subgroup of level $n$ for all $1 \leq n \leq r$. Then over the partial Igusa tower $\mathfrak{IG}_{n,r}$ that classifies generic trivializations $\Sh{O}/p^n\Sh{O} \simeq H_n^{\vee}$ of the dual of the canonical subgroup, one can define an integral model $\Omega_{\Sh{A}}$ of the modular sheaf $\omega_{\Sh{A}}$ as the submodule generated by all lifts of the image $\dlog(P^{\textrm{univ}})$ of the universal generator $P^{\textrm{univ}}$ of $H_n^{\vee}$ under the $\dlog$ map. This is an $\Sh{O}_L \otimes \Sh{O}_{\mathfrak{IG}_{n,r}}$-line bundle that is equipped with a marked section $\dlog(P^{\textrm{univ}})$ by construction. Then using the theory of vector bundles with marked section one can define the modular sheaf $\mathfrak{w}_{k}^0$ of weight $k^0$ as in \cite{andreatta2021triple}. This definition coincides with previous definitions in \cite{andreatta2016adic}, \cite{graziani} and \cite{Aycock}.

As mentioned before the key construction in this theory is the definition of interpolation sheaves of de Rham classes. Such a definition appears in \cite{graziani} and \cite{Aycock}. The definition due to Graziani follows closely the analogous definition in the elliptic case due to Andreatta--Iovita. One has to choose a suitable integral model $\Hs{A}$ of $H^1_{\textrm{dR}}(\Sh{A})$ such that the induced Hodge filtration identifies the zeroth filtered piece with $\Omega_{\Sh{A}}$. Our choice of the integral model $\Hs{A}$ differs from that due to Graziani. Let $\underline{\xi}$ be the invertible $\Sh{O}_L \otimes \Sh{O}_{\mathfrak{IG}_{n,r}}$ ideal that satisfies $\Omega_{\Sh{A}} = \underline{\xi}\omega_{\Sh{A}}$. Let $\widetilde{HW}$ be the invertible $\Sh{O}_L \otimes \Sh{O}_{\mathfrak{IG}_{n,r}}$ ideal generated by local lifts of the Hasse-Witt matrix. The first key result towards the goal of defining the interpolation sheaf for de Rham classes is the following. 

\begin{customthm}{C}\label{3}
For suitable choice of $r, n$ as above, there exists a locally free $\Sh{O}_L \otimes \Sh{O}_{\mathfrak{IG}_{n,r}}$ sheaf $\Hs{A} \subset H^1_{\textrm{dR}}(\Sh{A})$ of rank 2, such that the induced Hodge filtration is 
\[
0 \to \Omega_{\Sh{A}} \to \Hs{A} \to \underline{\xi}\widetilde{HW}\omega_{\Sh{A}}^{\vee} \to 0.
\]
Moreover the Hodge filtration admits a canonical splitting modulo $p/\Hdg{p^2}$ that coincides with the unit root splitting over the ordinary locus. The splitting is functorial for the lift of Verschiebung.
\end{customthm}

The proof of this result is the content of \S\ref{S2} and \S\ref{2S31}. We remark that in the case of $\mathbf{GL}_{2,\Q}$ one has an analogous splitting of the integral model $\Hs{E}$ modulo a small power of $p$ for the definition of $\Hs{E}$ used in \cite{andreatta2021triple}.

Let $\Sh{Q}$ be the kernel of the splitting in the theorem above. Using this definition of $\Hs{A}$, we define a notion of vector bundle with marked sections and marked splitting, similar to the definition of Molina \cite[\S4.2]{Molina} as follows.
\[
\V{\Sh{O}_L}(\Hs{A}, \dlog(P^{\textrm{univ}}),\Sh{Q})(R) := \left\{ f \colon \Hs{A}(R) \to \Sh{O}_L \otimes R \, | \, f \textrm{ is } \Sh{O}_L\textrm{-linear}, f(\dlog({P^{\textrm{univ}}}) = 1), f(\Sh{Q}) = 0 \right\} 
\]
We prove that this vector bundle is representable. In fact as an adic space this has a very simple local description. Let $s$ be a local lift of $\dlog(P^{\textrm{univ}})$ and $t$ be a lift of a local generator of $\Sh{Q}$. Let $\beta_n$ be the small power of $p$ such that $\dlog(P^{\textrm{univ}})$ is a section of $\Omega_{\Sh{A}}/\beta_n$. Let $\eta = p/\Hdg{p^2}$. Then locally as adic spaces,
\[
\V{\Sh{O_L}}(\Hs{A}, \dlog(P^{\textrm{univ}}), \Sh{Q}) = \V{\Sh{O}_L}(\Hs{A})\left\langle\frac{s-1}{\beta_n}, \frac{t}{\eta} \right\rangle.
\]
Here $\V{\Sh{O}_L}(\Hs{A})$ is the usual vector bundle whose sections over $R$ are $\Sh{O}_L \otimes R$-linear maps $\Hs{A}(R) \to \Sh{O}_L \otimes R$. (Note since we assume that $p$ is unramified in $L$ it is fine to take $\Sh{O}_L \otimes R$ as the codomain instead of $\mathfrak{d}^{-1}\otimes R$.) We then define the interpolation sheaf $\mathbb{W}_{k}^0$ using this refined version of VBMS and prove the rest of the statements in Theorem \ref{1}. The proof of these results concerns the rest of \S\ref{S3}. 

The next section \S\ref{2S4} is where we define the Gauss--Manin connection on $\mathbb{W}^0_k$, and show that it can be $p$-adically iterated for analytic weights. The strategy is exactly similar to \cite{andreatta2021triple} and \cite{Molina}. Using Grothendieck's description of connections we prove the first part of Theorem \ref{2}:

\begin{theorem*}
The Gauss--Manin connection on $H^1_{\textrm{dR}}(\Sh{A})$ induces a connection $\nabla$ on $\mathbb{W}^0_k$ over the generic fibre that satisfies Griffiths' transversality with respect to the filtration mentioned above.
\end{theorem*}

The definition of $p$-adic iteration of $\nabla$ follows the strategy of \cite{andreatta2021triple}. For each embedding $\sigma \in \Sigma$, the Kodaira--Spencer class corresponding to $\sigma$ gives a partial connection $\nabla(\sigma) \colon \mathbb{W}^0_k[1/p] \to \mathbb{W}^0_{k+2\sigma}[1/p]$. We first study the convergence properties of $\nabla(\sigma)$ over the ordinary locus using $q$-expansions and local coordinates. Here we realize that $\nabla(\sigma)$ behaves exactly like the connection $\nabla$ in the case of elliptic curves \cite{andreatta2021triple}. However, owing to our use of VBMS with marked splitting, we get faster convergence estimates. This is really the key improvement to the technique of \cite{andreatta2021triple} and mimics the results obtained by Molina. In particular \cite[Proposition 4.10]{andreatta2021triple} states that for any $p$-depleted $g \in H^0(\mathfrak{IG}^{\textrm{ord}}_n, \mathbb{W}^0_k)$,
\[
(\nabla^{p-1} - \textrm{id})^p(g) \in pH^0(\mathfrak{IG}^{\textrm{ord}}_n, \mathbb{W}^0).
\]
Here $\mathbb{W}^0$ is the direct image of the structure sheaf of $\V{}(\Hs{E}, \dlog(P^{\textrm{univ}})) \to \mathfrak{IG}^{\textrm{ord}}_n$. Instead, if we work with $\V{}(\Hs{E}, \dlog(P^{\textrm{univ}}), \Sh{Q})$, and let $\mathbb{W}^0$ be the direct image of its structure sheaf, then we can prove that 
\[
(\nabla^{p-1} - \textrm{id})(g) \in pH^0(\mathfrak{IG}^{\textrm{ord}}_n, \mathbb{W}^0).
\]
An exactly analogous statement holds for $\nabla(\sigma)$ for all $\sigma$ in the Hilbert case.
Then for any analytic weight $s = \prod_{\sigma \in \Sigma} s_{\sigma} = \prod \exp(v_{\sigma}\log(\cdot))$ and any Hilbert modular form $g$ of weight $k$ that is depleted at all primes above $p$, satisfying the conditions on $u_{\sigma}, v_{\sigma}$ as mentioned in Theorem \ref{2},  we show that $\nabla(\sigma)^{s_{\sigma}}(g)$ defined using a formal expression
\[
\nabla(\sigma)^{s_{\sigma}}(g) := \exp\left( \frac{v_{\sigma}}{p^{f_{\sigma}}-1} \log\nabla(\sigma)^{p^{f_{\sigma}}-1}\right)(g)
\]
is actually the limit of a Cauchy sequence in $\mathbb{W}^0$ over the ordinary locus. Finally we manage to show that by shrinking the initial radius of overconvergence, it is possible to realize $\nabla(\sigma)^{s_{\sigma}}(g)$ as an overconvergent de Rham class. We can also check on $q$-expansions that $\nabla(\sigma)$ and $\nabla(\tau)$ commutes for different $\sigma, \tau$. Then the definition of $\nabla^s = \prod_{\sigma \in \Sigma} \nabla(\sigma)^{s_{\sigma}}$ is obvious. This then proves the second part of Theorem \ref{2}.

\begin{theorem*}
Let $k = \prod k_{\sigma} = \prod \exp(u_{\sigma}\log(\cdot))$ and $s = \prod s_{\sigma} = \prod \exp(v_{\sigma}\log(\cdot))$ be two analytic weights such that $u_{\sigma}, v_{\sigma} \in \Sh{O}^+_{\Sh{W}^0_p}$. Then for any $g$ of weight $k$, depleted at all primes above $p$, $\nabla^s(g)$ makes sense as an overconvergent de Rham class of weight $k+2s$. 
\end{theorem*}

We remark that while we prove the splitting of $\mathrm{H}^{\sharp}_{\Sh{A}}$ in Theorem \ref{3} modulo a small power of $p$, with a little more work it is possible to prove such a splitting modulo higher powers, under the assumption that the radius of overconvergence is small enough (i.e. $r$ in $\Xr{r}$ is large enough). One important consequence of working with the refined version of VBMS is that we manage to relax the condition on analyticity of the weight $s$ slightly, compared to that in \cite{andreatta2021triple}. Moreover, working with splitting of $\mathrm{H}^{\sharp}_{\Sh{A}}$ modulo higher powers allows one to relax this condition even further to whatever extent desired, at the cost of shrinking the radius of overconvergence. While preparing this manuscript, we came to know about the recent work of Andrew Graham, Vincent Pilloni and Joaqu\'{i}n Rodrigues Jacinto \cite{graham2023padic}. Their work in the elliptic case uses a similar idea of working with rational subdomains of the form $\V{\Sh{O}_L}(\Hs{A})\left\langle\frac{s-1}{\beta_n}, \frac{t}{\eta} \right\rangle$. More precisely, they first construct strict neighbourhoods $\Sh{U}_{\mathrm{HT},n}$ (in their notation) of the Igusa tower at $\infty$ (they work with the Igusa tower which classifies trivializations of both the connected and \'{e}tale part of the $p$-divisible group) for all $n\geq 1$, inside the torsor $P^{\text{an}}_{\text{dR},\Sh{X}_{\text{ord}}}$ of trivializations of the de Rham homology, over the ordinary locus. Base changed to the Igusa tower (along the projection to $\Sh{X}_{\text{ord}}$), the strict neighbood $\Sh{U}_{\mathrm{HT},n}$ is a vector bundle which is a torsor for the formal group (upto connected components)
\[
\Sh{B} = \begin{pmatrix}
    1+p^n\Ga & 0 \\
    p^n\Ga & 1+p^n\Ga 
\end{pmatrix}.
\] 
Next they construct overconvergent extensions $\Sh{U}_{n,r}$ of these strict neighbourhoods over the overconvergent locus $\mathcal{X}_r$ for each $r\geq 1$ \cite[\S3.4]{graham2023padic}. These extensions form a cofinal system of strict neighbourhoods of the Igusa tower inside $P^{\text{an}}_{\text{dR}}$ which is the torsor of trivializations of the de Rham homology over the entire modular curve. Since they define their space of nearly overconvergent modular forms as the direct limit $\varinjlim_{n,r}\Sh{O}(\Sh{U}_{n,r})$, this allows them to define $\nabla^s$ without any conditions on the analyticity on $s$. We remark that by using splittings of $\mathrm{H}^{\sharp}_{\Sh{A}}$ modulo higher powers of $p$, one should be able to show that for $r$ large enough (so as to allow for the existence of the level $n$ canonical subgroup $H_n$), the base change of $\Sh{U}_{n,r}$ to the level $n$ partial Igusa tower is a vector bundle which is a torsor for $\Sh{B}$ as above. We do not explore this direction in this work because we are interested in applying our theory to the construction of Katz--BDP type $p$-adic $L$-functions in the future. In such constructions it is necessary to control the radius of overconvergence of the nearly overconvergent form that one obtains by applying a $p$-adic iterate of $\nabla$ to an overconvergent form. By what we said above, relaxing the condition on analyticity comes at the cost of shrinking the radius of overconvergence. We still encourage the readers to look at \cite{graham2023padic}.

\subsection{Application to construction of \texorpdfstring{$p$}{p}-adic \texorpdfstring{$L$}{L}-function}
In their article \cite{michele}, Blanco-Chacon and Fornea construct a twisted triple product $p$-adic $L$-function associated to two nearly ordinary families. In the special case of a real quadratic extension $L/\Q$, they relate the special values of their $L$-function to syntomic Abel--Jacobi images of \emph{generalized Hirzebruch--Zagier cycles}. In this paper we construct a twisted triple product $p$-adic $L$-function associated to finite slope families of Hilbert modular forms, generalizing the definition of \cite{michele} using the theory of $p$-adic iteration of Gauss--Manin connection.

Let $L/F$ be a quadratic extension of totally real fields. The weights of the arithmetic Hilbert modular forms associated to the group $G_L := \Res_{\Sh{O}_L/\Z}\GL{2}$ are parametrized by pairs $(v,n) \in \Z[\Sigma_L]\times \Z$. These forms correspond to weight $k = 2v + nt_L$ geometric Hilbert modular forms on which the centre $\A{L}{\times}$ acts via a Hecke character with infinity type $-nt_L$, where $t_L = \sum_{\sigma \in \Sigma_L} \sigma$ is the parallel weight of weight 1. Let $g, f$ be arithmetic Hilbert modular forms of weight $(v,n) \in \Z[\Sigma_L]\times \Z$ and $(w,2n) \in \Z[\Sigma_F]\times \Z$ for the groups $G_L = \Res_{\Sh{O}_L/\Z}\GL{2}$ and $G_F := \Res_{\Sh{O}_F/\Z}\GL{2}$ respectively. The natural inclusion of the Shimura varieties induced by $\GL{2}{(\A{F}{})} \to \GL{2}{(\A{L}{})}$ induces a diagonal restriction map $\zeta^* \colon S^{G_L}(K_1(N), (v,n),\C{}) \to S^{G_F}(K_1(N), (v_{|F}, 2n), \C{})$ on the space of arithmetic cuspforms of weight $(v,n)$. 

Assume now that $g, f$ are primitive eigenforms of level $N_g, N_f$ respectively, such that the automorphic representation $\Pi$ of $\GL{2}{(\A{L\times F}{})}$ defined by their unitarization has trivial central character. Assume also that there exists $t\in \N[\Sigma_L]$ such that $(v+t)_{|F} = w$. This condition is called $F$-dominated weights in \cite{michele}, and generalizes the notion of unbalanced triple of weights in the triple product $L$-function case. Then under some assumptions, Ichino's formula relates the central value of the twisted triple product $L$-function associated to $\Pi$ to a Petersson product of cuspforms (Lemma \ref{L3401})
\[
L\left(\frac{1}{2},\Pi, r\right) = (\ast)\left|\langle \zeta^*(\delta^t\breve{g}), \breve{f}^*\rangle\right|^2.
\]
Here $\breve{g}, \breve{f}$ are oldforms coming from $g,f$ respectively at some level $\mathfrak{A}$ that is a common multiple of $N_g,N_f$, $\breve{f}^*$ denotes the form whose Fourier coefficients are complex conjugates of that of $\breve{f}$, and $(\ast)$ are some constants which can be made to be non-zero by suitable choice of $\breve{g}$ and $\breve{f}$. The operator $\delta$ is the Maass--Shimura operator. The value $\langle \zeta^*(\delta^t\breve{g}), \breve{f}^*\rangle/\langle f^*, f^*\rangle$ is algebraic, and can be interpolated. Let $\omega_g, \omega_f$ be overconvergent cuspforms interpolating $g, f$ of weight $(v_g, n_g) \in \Sh{W}^{G_L}(\Lambda_g)$ and $(w_f, m_f) \in \Sh{W}^{G_F}(\Lambda_F)$, such that there exists $r \in \Sh{W}^{G_L}(\Lambda_g)$ satisfying $((v_g+r)_{|F}, {n_g}_{|F}) = (w_f, m_f)$. Then under some analyticity conditions on the weights $v_g, r$ (Assumption \ref{A710}), and for any unramified prime $p$ in $L$ that does not divide the level $\mathfrak{A}$, so that $f$ has slope $\leq a$, we define a $p$-adic $L$-function (Definition \ref{D714}),
\[
\sh{L}^g_p(\breve{\omega}_g, \breve{\omega}_f) = \frac{\langle H^{\dagger,\leq a}(\zeta^*\nabla^r\breve{\omega}_g^{[\Sh{P}]}),\breve{\omega}_f^*\rangle}{\langle \omega_f^*, \omega_f^*\rangle}.
\]
This $p$-adic $L$-function is meromorphic on the space $\mathrm{Spm}(\Lambda_g \hat{\otimes}_{\Lambda^{G_F}} \Lambda_f)$.

\begin{customthm}{D}
    Let $((x,c),(y,2c)) \in \mathrm{Spm}(\Lambda_g\hat{\otimes}\Lambda_f)$ be classical weights, such that $(x,c)$ and $(y,2c)$ are $F$-dominated, i.e. there exists $t \in \N[\Sigma_L]$ such that $(x+t)_{|F} = y$, and such that the specialization of the overconvergent sheaf $\mathfrak{w}_k$ at $(y,2c)$ consists entirely of classical forms, and the fibre of the eigenvariety at $(y,2c)$ is \'{e}tale. Assume the specializations $g_x := {\omega_g}_x$, and $f_y := {\omega_f}_y$ are eigenforms for Hecke operators outside $\mathfrak{A}$, and the Hecke polynomial of $T_{\mathfrak{p}}$ is separable for all primes $\mathfrak{p}|p$. Assume $f_y$ has slope $\leq a$. Then 
    \[
    \sh{L}^g_p(\breve{\omega}_g, \breve{\omega}_f)((x,c),(y,2c)) = \mathrm{E}\cdot\frac{\langle \zeta^*(\delta^t\breve{g}_x), \breve{f}^*_y\rangle}{\langle f^*_y, f^*_y\rangle}, 
    \]
    where $E$ is some Euler factor depending on $g_x, f_y$. Moreover, this $p$-adic $L$-function recovers the $p$-adic $L$-function constructed in \emph{\cite{michele}} for Hida families, at all points where they both are defined.
\end{customthm}

In upcoming work with Ting-Han Huang, we will use this $p$-adic $L$-function to prove a $p$-adic Gross-Zagier formula similar to \cite{michele}.

\subsection{Acknowledgements}
This work would not have been possible without the generous support of Fabrizio Andreatta, who taught me much of what went into this work. I would also like to thank Adrian Iovita for his support and generosity. I have benefited from several discussions with Leonardo Fiore, Michele Fornea, Andrew Graham, Giovanni Rosso, Ting-Han Huang and Antonio Cauchi. My regards to them all. 

%% file: Hilbert1.tex
\section{The setup}\label{S1}
\noindent\textbf{Notation.} Let $L$ be a totally real number field of degree $[L:\Q] = g>1$. Denote by $\mathfrak{d}$ the different ideal of $\Sh{O}_L$. Fix an integer $N \geq 4$. Let $p \nmid N$ be a prime which is unramified in $L$. Suppose $p$ splits as $p = \mathfrak{P}_1 \cdots \mathfrak{P}_h$ and let their inertia degree be $f(\mathfrak{P}_i|p) = f_i$. Fix a finite unramified Galois extension $K$ of $\Q_p$ where $L$ is split and let $\Sigma := \{\sigma \colon L \to K\}$ be the set of embeddings of $L$ in $K$. Let $q = p$ if $p \neq 2$ and $q = 4$ otherwise. 

\subsection{The weight space}
Let $\mathbb{T} = \Res_{\Sh{O}_L/\Z}\Gm$. Then $\mathbb{T}(\Z_p) = (\Sh{O}_L \otimes \Z_p)^{\times}$. Denote by $\Lambda := \Sh{O}_K\llbracket \mathbb{T}(\Z_p)\rrbracket$ the base change of the Iwasawa algebra $\Z_p\llbracket \mathbb{T}(\Z_p)\rrbracket$ to $\Sh{O}_K$. Then $\Lambda = \Sh{O}_K[\Delta]\llbracket \mathbb{T}(\Z_p)_{\textrm{tf}}\rrbracket \simeq \Sh{O}_K[\Delta]\llbracket T_1, \dots, T_g\rrbracket$, where $\Delta \subset \mathbb{T}(\Z_p)$ is the torsion subgroup and we choose an isomorphism  $\mathbb{T}(\Z_p)_{\textrm{tf}} \simeq \Z_p^g$ of the torsion free part with $\Z_p^g$. Under this isomorphism, the standard basis elements $e_i$ of $\Z_p^g$ are sent to $1+T_i$. Let $\Lambda^0 = \Sh{O}_K\llbracket \mathbb{T}(\Z_p)_{\textrm{tf}}\rrbracket$ be the quotient of $\Lambda$ that sends $\Delta \mapsto 1$, and let $\mathfrak{m} = (p, T_1, \dots, T_g)$ be its maximal ideal.

Let $\mathfrak{W} = \Spf{\Lambda}$ and $\mathfrak{W}^0 = \Spf{\Lambda^0}$. Let $\Sh{W} = \Spa(\Lambda, \Lambda)^{\textrm{an}}$ be the analytic adic space associated to $\mathfrak{W}$ and similarly define $\Sh{W}^0 := \Spa(\Lambda^0, \Lambda^0)^{\textrm{an}}$. $\Sh{W}$ satisfies the following universal property: for any complete Huber pair $(R, R^+)$ over $(\Q_p, \Z_p)$,
\[
\Hom_{\Spa(\Q_p,\Z_p)}\big(\Spa(R, R^+), \Sh{W}\big) = \Hom_{\textrm{gp}\mhyphen\textrm{sch}}\big({\mathbb{T}(\Z_p)}_{R^+}, {\Gm}_{R^+}\big) =  \Hom_{\Z}^{\textrm{cont}}\big(\mathbb{T}(\Z_p), {(R^+)}^{\times}\big).
\]
$\Sh{W}^0$ satisfies a similar universal property with respect to $\mathbb{T}(\Z_p)_{\textrm{tf}}$, i.e. 
\[
\Hom_{\Spa(\Q_p, \Z_p)}\big(\Spa(R, R^+), \Sh{W}^0\big) = \Hom_{\Z}^{\textrm{cont}}\big(\mathbb{T}(\Z_p)_{\textrm{tf}}, (R^+)^{\times}\big).
\]

Let $\widetilde{\mathfrak{W}^0}$ be the admissible blow-up of $\mathfrak{W}^0$ along $\mathfrak{m}$. For $\alpha \in \mathfrak{m} \setminus \mathfrak{m}^2$, let $\mathfrak{W}^0_{\alpha}$ be the open in  $\widetilde{\mathfrak{W}^0}$ where $\mathfrak{m}$ is generated by $\alpha$.  Suppose $\mathfrak{W}^0_{\alpha} = \Spf{B^0_{\alpha}}$. Let $\Sh{W}^0_{\alpha}$ be the rational open in $\Sh{W}^0$ obtained by taking the generic fibre of $\mathfrak{W}^0_{\alpha}$. This is given by the affinoid adic $\Spa(B^0_{\alpha}[1/\alpha], B^0_{\alpha})$ and is thus the adic spectrum of a Tate ring. The $\mathfrak{W}^0_{\alpha}$ cover $\widetilde{\mathfrak{W}^0}$ for varying $\alpha$ and hence their adic generic fibre cover $\Sh{W}^0$. In particular, the map induced by $\widetilde{\mathfrak{W}^0} \to \mathfrak{W}^0$ on the associated analytic adic spaces is an isomorphism.

The natural inclusion $\Lambda^0 \to \Lambda$ is finite flat and the induced finite flat morphism $\Sh{W} \to \Sh{W}^0$ realizes $\Sh{W}$ as a disjoint union of copies of $\Sh{W}^0$ indexed by $\Delta$. For $\alpha \in \mathfrak{m} \setminus \mathfrak{m}^2$ we let $\Sh{W}_{\alpha}$ be the inverse image of $\Sh{W}^0_{\alpha}$ under this morphism.

We remark that all classical weights can be realized as points in $\Sh{W}_{p}$.

For $I = [p^a, p^b]$ with $a \in \N \cup \{-\infty\}$ and $b \in \N \cup \{\infty\}$, let $\Sh{W}^0_{\alpha,I} \subset \Sh{W}^0_{\alpha}$ be the rational open subset defined as follows:
\[
\Sh{W}^0_{\alpha,I} := \{ x \in \Sh{W}^0_{\alpha} \, : \, |p(x)| \leq |\alpha(x)|^{p^a} \neq 0, \ |\alpha(x)|^{p^b} \leq |p(x)| \neq 0 \}.  
\]
Let $\Lambda^0_{\alpha,I} := \Gamma(\Sh{W}^0_{\alpha, I}, \Sh{O}_{\Sh{W}^0_{\alpha,I}}^+)$ and $\Lambda_{\alpha, I} := \Gamma(\Sh{W}_{\alpha,I}, {\Sh{O}}^+_{\Sh{W}_{\alpha,I}})$. Let  $\mathfrak{W}^0_{\alpha, I} := \Spf{\Lambda^0_{\alpha,I}}$ and $\mathfrak{W}_{\alpha,I} := \Spf{\Lambda}_{\alpha,I}$.

For $\alpha$ varying in $\mathfrak{m} \setminus \mathfrak{m}^2$, the different $\mathfrak{W}^0_{\alpha,I}$ glue together to form a formal scheme ${\mathfrak{W}^0_I}$ with adic generic fibre $\Sh{W}^0_I$. $\Sh{W}^0_I$ can be described as follows.
\[
\Sh{W}^0_I = \{ x \in \Sh{W}^0 \, : \, (\exists \alpha \in \mathfrak{m}, |p|_x \leq |\alpha^{p^a}|_x \neq 0) \land (\forall \alpha \in \mathfrak{m}, |\alpha^{p^b}|_x \leq |p|_x \neq 0)\}
\]
Then $\Sh{W}^0_{[0,1]} = \Sh{W}^0_p$ and $\Sh{W}^0_{[1, \infty]} = \Sh{W}^0$. At the level of formal models, $\mathfrak{W}^0_{[0,1]} = \mathfrak{W}^0_p$ and $\mathfrak{W}^0_{[1,\infty]} = \widetilde{\mathfrak{W}^0}$.

We fix one such $\alpha$. For the purpose of defining the $p$-adic iteration of the Gauss--Manin connection, which is the technical heart of this work (\S\ref{2S44}) we need to assume $\alpha = p$. For everything before that section we can choose any $\alpha$. But we should also mention that the construction of the main objects of this work, i.e. the interpolation sheaves of modular forms and de Rham classes of varying weight $k$, all take place over a weight space where $p \neq 0$. In particular we do not study the perfect overconvergent modular forms of \cite{Andreatta2018leHS} or \cite{andreatta2016adic}.

\vspace{5pt}
\paragraph{\textbf{Analyticity of the universal character.}}
Let $k^{\textrm{un}} \colon \mathbb{T}(\Z_p) \to \Lambda$ be the universal character. Denote by $k^0 \colon \mathbb{T}(\Z_p) \to \Lambda \to \Lambda^0$ the character obtained by composing the universal character with the projection onto the component of the trivial character, and let $k^0_{\alpha,I} \colon \mathbb{T}(\Z_p) \to \Lambda^0 \to \Lambda^0_{\alpha, I}$ be its restriction to $\Sh{W}^0_{\alpha, I}$.
\begin{lemma}\label{L21}
For $I \subset [0, q^{-1}p^n]$, the restriction of $k^0_{\alpha, I}$ to $1 + qp^{n-1}(\Sh{O}_L \otimes \Z_p)$ is analytic. Thus it extends to a character 
\[
k^0_{\alpha, I} \colon \Sh{W}^0_{\alpha, I} \times \mathbb{T}(\Z_p)(1 + qp^{n-1}\Res_{\Sh{O}_L/\Z}\Ga^+) \to \Gm^+
\]
which restricts to a character
\[
k^0_{\alpha,I} \colon \Sh{W}^0_{\alpha, I} \times (1 + qp^{n+m-1}\Res_{\Sh{O}_L/\Z}\Ga^+) \to 1 + qp^m\Ga^+.
\]
\end{lemma}

\begin{proof}
This is an adaptation of the proof of \cite[Proposition 2.1]{Andreatta2018leHS}. See also \cite[Proposition 2.8]{andreatta2016adic}.
\end{proof}

Since $k^0_{\alpha,I}$ is analytic, and $1+qp^{n-1}\Res_{\Sh{O}_L/\Z}\Ga \simeq \prod_{\sigma \in \Sigma} 1 + qp^{n-1}\Ga$, the universal character splits into components $k^0_{\alpha,I} = \prod_{\sigma \in \Sigma} k_{\sigma}$.

\begin{rem}
Note that the analyticity of the character does not depend on $\alpha$. So we can glue the different $k^0_{\alpha,I}$ together to obtain a character $k^0_I \colon \Sh{W}^0_I \times \mathbb{T}(\Z_p)(1 + qp^{n-1}\Res_{\Sh{O}_L/\Z_p}\Ga^+) \to \Gm^+$.
\end{rem}

\subsubsection{The weight space for the group \texorpdfstring{$\Res_{L/\Q}\mathbf{GL}_{2,L}$}{G}}
In the literature there are two distinct notions of Hilbert modular forms. The first one is realized as sections of a modular sheaf on moduli of abelian varieties with real multiplication and additional data. The associated weight space is the one we defined above. There is also the notion of arithmetic Hilbert modular forms which are sections of automorphic line bundles on the Shimura variety associated to the group $G := \Res_{L/\Q}\mathbf{GL}_{2,L}$. It is necessary to consider these modular forms for arithmetic applications. The relation between these two different notions will be clarified in the next section. Here we will define the weight space associated to the group $G$.

Let $\Lambda^G := \Sh{O}_K\llbracket \mathbb{T}(\Z_p) \times \Z_p^{\times}\rrbracket$, and let $\mathfrak{W}^G := \Spf{\Lambda^G}$. Let $\Sh{W}^G := \Spa(\Lambda^G, \Lambda^G)^{\textrm{an}}$ be the associated analytic adic space. There is a natural morphism $\mathfrak{W}^G \to \mathfrak{W}$ given by the map $\Lambda \to \Lambda^G$ induced by $(\Sh{O}_L \otimes \Z_p)^{\times} \to (\Sh{O}_L \otimes \Z_p)^{\times} \times \Z_p^{\times}$ sending $t \mapsto (t^2, \Nm_{L/\Q}(t))$. This induces a map on the corresponding analytic adic spaces. On classical points this map can be described as sending $(v, w) \in \Sh{W}^G(\C{}_p) \mapsto v^2 \cdot (w \circ \Nm_{L/\Q})$, where $v \colon (\Sh{O}_L \times \Z_p)^{\times} \to \C{\times}_p$ is a continuous character and so is $w \colon \Z_p^{\times} \to \C{\times}_p$. Denote the universal character by $k^{\textrm{un}}_G \colon (\Sh{O}_L \times \Z_p)^{\times} \times \Z_p^{\times} \to (\Lambda^G)^{\times}$.

\subsection{The Hilbert modular variety}
Let $\mathfrak{c}$ be a fractional ideal of $L$ and let $\mathfrak{c}^+$ be the cone of totally positive elements. Let $M(\mu_N, \mathfrak{c})$ be the moduli scheme over $\Z_p$ classifying tuples $(A_{/S}, \iota, \lambda, \psi)$ consisting of (1) an abelian scheme $A \to S$ for any $\Z_p$-scheme $S$, (2) an embedding $\iota \colon \Sh{O}_L \xhookrightarrow{} \End_S(A)$, (3) if $P \subset \Hom_{\Sh{O}_L}(A, A^{\vee})$ is the \'{e}tale sheaf of symmetric $\Sh{O}_L$-linear homomorphisms from $A$ to its dual $A^{\vee}$, and $P^+$ is the cone of polarizations, then an isomorphism $\lambda \colon (P, P^+) \simeq (\mathfrak{c}, \mathfrak{c}^+)$ of \'{e}tale sheaves of invertible $\Sh{O}_L$-modules with a notion of positivity such that the induced map $A \otimes \mathfrak{c} \xrightarrow{\sim} A^{\vee}$ is an isomorphism (the Deligne--Pappas condition), and (4) a closed immersion $\psi \colon \mu_N \otimes \mathfrak{d}^{-1} \xhookrightarrow{} A[N]$ compatible with $\Sh{O}_L$-action.

Let $[\mathfrak{c}] \in \textrm{Cl}^+(L)$ be the class of $\mathfrak{c}$ in the strict class group of $L$. Any two representatives of a class are related via multiplication by a totally positive unit, and hence the corresponding moduli problems are isomorphic. Using this, we henceforth fix $\mathfrak{c}$ coprime to $p$. That will ensure that upon a choice of a generator of $\mathfrak{c} \otimes \Z_p$ as an $\Sh{O}_L \otimes \Z_p$-module, the $p$-divisible groups of $\mathfrak{c}$-polarized abelian varieties satisfying the Deligne--Pappas condition are principally polarized.

Let $\bar{M}(\mu_N, \mathfrak{c})$ be a toroidal compactification of $M(\mu_N, \mathfrak{c})$ and let $M^*(\mu_N, \mathfrak{c})$ be the minimal compactification. There is a semi-abelian scheme $\pi \colon \Sh{A} \to \bar{M}(\mu_N, \mathfrak{c})$ which restricts to the universal abelian scheme over $M(\mu_N, \mathfrak{c})$ and degenerates to a torus at the cusps. Let $\omega_{\Sh{A}}$ be the canonical extension of the sheaf of invariant differentials of the universal abelian scheme to the cusps. There is a largest open subscheme $M^R(\mu_N, \mathfrak{c}) \subset \bar{M}(\mu_N, \mathfrak{c})$, called the Rapoport locus where $\omega_{\Sh{A}}$ is an invertible $\Sh{O}_L \otimes \Sh{O}_{M^R(\mu_N, \mathfrak{c})}$-module. Since we assume $p$ is unramified in $L$, the complement of the Rapoport locus is empty and $\bar{M}(\mu_N, \mathfrak{c})$ is projective and smooth over $\Sp{\Z_p}$ \cite{DelignePappas}. The boundary $D := \bar{M}(\mu_N, \mathfrak{c}) \setminus M(\mu_N, \mathfrak{c})$ is a relative normal crossings divisor. The minimal compactification $M^*(\mu_N, \mathfrak{c})$ is normal and projective.

Henceforth denote $\bar{M}(\mu_N, \mathfrak{c})$ by $X$. Let $\mathfrak{X}$ be its $p$-adic completion. Let $\mathrm{H}_{\Sh{A}}$ be the canonical extension of the relative de Rham sheaf of the universal abelian scheme to the cusps. It is a locally free $\Sh{O}_L \otimes \Sh{O}_{\mathfrak{X}}$-module of rank $2$ and it is endowed with an integrable connection $\nabla \colon \mathrm{H}_{\Sh{A}} \to \mathrm{H}_{\Sh{A}} \otimes \Omega^1_{\mathfrak{X}/\Z_p}(\log(D))$ called the Gauss--Manin connection. It also fits in an exact sequence
\begin{equation}\label{eq:1}
0 \to \omega_{\Sh{A}} \to \mathrm{H}_{\Sh{A}} \to \omega^{\vee}_{\Sh{A}^{\vee}} \to 0
\end{equation}
which defines the Hodge filtration. The maps in the sequence are $\Sh{O}_L$-linear. We also use the principal polarization of the $p$-divisible group of $\Sh{A}$ to henceforth identify $\omega_{\Sh{A}^{\vee}}$ and $\omega_{\Sh{A}}$. 

Fix $\alpha \in \mathfrak{m} \setminus \mathfrak{m}^2$ and $I$ as above. Let $\mathfrak{X}_{\alpha, I} := \mathfrak{X} \times_{\Spf{\Z_p}} \mathfrak{W}^0_{\alpha, I}$. Since $L$ splits in $K$, $\Sh{O}_L \otimes \Sh{O}_{\mathfrak{X}_{\alpha,I}} \simeq \prod_{\sigma \in \Sigma} \Sh{O}_{\mathfrak{X}_{\alpha,I}}$. Later it will be useful to collect the $\sigma$'s in $h$ different groups according to the valuation they induce on $\Sh{O}_L$ (recall $p = \mathfrak{P}_1 \cdots \mathfrak{P}_h$ in $L$). Hence over $\mathfrak{X}_{\alpha,I}$, the exact sequence (\ref{eq:1}) splits into $h$ exact sequences 
\[
0 \to \omega_{\Sh{A}}(\sigma) \to \mathrm{H}_{\Sh{A}}(\sigma) \to \omega^{\vee}_{\Sh{A}}(\sigma) \to 0
\]
where each $\omega_{\Sh{A}}(\sigma)$ is an invertible $\Sh{O}_{\mathfrak{X}_{\alpha,I}}$-module and $\mathrm{H}_{\Sh{A}}(\sigma)$ is a locally free $\Sh{O}_{\mathfrak{X}_{\alpha, I}}$-module of rank 2. 

Let $\xi \colon \X/(\alpha) \xhookrightarrow{} \X$ be the closed subscheme defined by $\alpha = 0$. The Hasse invariant is a section $\mathrm{Ha} \in (\Lambda^g (\xi^*\omega_{\Sh{A}}))^{\otimes (p-1)}$. Define the Hasse ideal to be $\underline{\mathrm{Ha}} := \mathrm{Ha}\cdot (\Lambda^g (\xi^*\omega_{\Sh{A}}))^{\otimes (1-p)}$. 

\begin{theorem}\label{T2101}
The Hasse invariant vanishes with multiplicity one along the irreducible components of its divisor.
\end{theorem}

\begin{proof}
\cite[Corollary 8.18]{AndreattaGoren}.
\end{proof}

With the notation above, for each $r \geq 1$, consider the inverse image of $\underline{\mathrm{Ha}}^{p^{r+1}}$ under the map $\Sh{O}_{\X} \to \Sh{O}_{\mathfrak{X}_{\alpha, I}/(\alpha)}$ and call this ideal $\mathrm{Hdg}_r$. We call a local lift of a generator of $\underline{\mathrm{Ha}}$ as $\mathrm{Hdg}$. Then locally $\mathrm{Hdg}_r$ is the ideal $(\alpha, \mathrm{Hdg}^{p^{r+1}})$. 

Recall a classical weight is an element of $\Z[\Sigma]$. The Hilbert modular sheaf of a classical weight $k = \sum n_{\sigma}\sigma$ is defined as 
\[
\omega_{\Sh{A}}^k := \hat{\bigotimes_{\sigma}} \omega_{\Sh{A}}(\sigma)^{n_{\sigma}}.
\]
Note that $\omega^k_{\Sh{A}}$ is obtained as the image of $\omega_{\Sh{A}}$ under the map induced by change of the structural group 
\[H^1(\mathfrak{X}_{\alpha,I}, (\Sh{O}_L \otimes \Sh{O}_{\mathfrak{X}_{\alpha, I}})^{\times}) \xrightarrow{k} H^1(\mathfrak{X}_{\alpha, I}, \Sh{O}^{\times}_{\mathfrak{X}_{\alpha, I}}).
\]

\subsubsection{The Shimura variety associated to $G$}\label{2S121}
Let $\mathrm{Sh}_K(G)$ be the Shimura variety associated to the group $G = \Res_{L/\Q}\mathbf{GL}_{2,L}$ and level subgroup $K = K_1(N)$,
\[
K_1(N) = \left\{ \begin{pmatrix}
a & b \\
c & d
\end{pmatrix}  \in \GL{2}(\hat{\Sh{O}}_L) \, | \, d \equiv 1, c \equiv 0 \textrm{ mod }N \right\}
\]
whose complex points are $\mathrm{Sh}_K(G)(\C{}) = G(\Q)\backslash (\mathfrak{h}^{\pm})^{\Sigma} \times G(\A{}{\infty})/K$. Here $\mathfrak{h}$ is the Poincar\'{e} upper-half plane endowed with the usual action of $\GL{2}(\R)$ via M\"{o}bius transformation. The Shimura variety $\mathrm{Sh}_K(G)$ is defined over its reflex field $\Q$.

The Shimura variety $\mathrm{Sh}_K(G)$ classifies abelian varieties with real multiplication upto isogeny in the following sense. Let $\mathbf{A}_L^{\Q}$ be the fibred category over $\Q$-$\catname{Sch}$, whose objects are abelian varieties with real multiplication by $\Sh{O}_L$, and whose morphisms are $\Hom_{F}^{\Q}(A,A') := \Hom_{\Sh{O}_L}(A,A')\otimes_{\Z} \Q$. Then $\mathrm{Sh}_K(G)(\C{})$ classifies triples $(A,{\lambda}, \eta K)$ upto isomorphism, where $A$ is an abelian variety with real multiplication by $\Sh{O}_L$, ${\lambda}$ is an $\Sh{O}_L$-linear polarization $\lambda \colon A \to A^{\vee}$, and $\eta K$ is the $K$-orbit of a trivialization $\eta\colon (\A{L}{\infty})^2 \xrightarrow{\sim} H_1(A, \A{}{\infty})$ of the rational Tate module \cite[133]{Hidabook}. Two objects $(A,\lambda, \eta K)$ and $(A', \lambda', \eta' K)$ are isomorphic if $A \simeq A'$ in $\mathbf{A}^{\Q}_L$, with the isomorphism taking $\lambda$ to an $F^{\times}_+$-multiple of $\lambda'$, and $\eta K$ to $\eta' K$. The isomorphism class $[(A, \bar{\lambda}, \eta K)]$ of a triple consists of special representatives $(A_0, {\lambda}, \eta K)$ for which $\eta(\hat{\Sh{O}}^2_L) = H_1(A_0, \hat{\Z}) = \prod_{\ell} T_{\ell}(A_0)(\C{})$. The complex variety $\mathrm{Sh}_K(G)(\C{})$ can then be viewed as classifying isomorphism classes of triples $(A, \bar{\lambda}, \eta)$, where $A$ is a complex abelian variety with real multiplication by $\Sh{O}_L$, $\bar{\lambda}$ is the $\Sh{O}^{\times, +}_L$-class of a polarization $\lambda\colon A \to A^{\vee}$, and $\eta \colon \mu_N \otimes \mathfrak{d}^{-1} \xhookrightarrow{} A[N]$ is a $\Gamma_1(N)$-level structure. An isomorphism of triples $(A, \bar{\lambda}, \eta) \simeq (A', \bar{\lambda}', \eta')$ is induced by an actual isomorphism $A\simeq A'$ of abelian varieties (as opposed to an isogeny). 

Concretely, corresponding to the point $[(\tau, g)] \in G(\Q)\backslash (\mathfrak{h}^{\pm})^{\Sigma} \times G(\A{}{\infty})/K$, one gets the triple $(A,\bar{\lambda}, \eta)$ as follows. To give the abelian variety $A$, it's enough to give its period lattice $H_1(A, \Z) \subset \C{\Sigma}$. Let $H_1(A, \Q) = L\cdot \tau + L\cdot 1$, i.e. we choose a basis $a\colon H_1(A, \Q) \to L^2$, where $a(\tau) = e_1$, and $a(1) = e_2$. The element $g \in \GL{2}{(\A{L}{\infty})}$ induces an isomorphism \[a^{-1}\circ(g^t)^{-1} \colon (\A{L}{\infty})^2 \xrightarrow{(g^t)^{-1}} (\A{L}{\infty})^2 \xrightarrow{a^{-1}} H_1(A, \Q) \otimes_{\Q} \A{}{\infty}.\]
The period lattice is then given by $H_1(A, \Z) := H_1(A, \Q) \cap a^{-1}\circ (g^t)^{-1}(\hat{\Sh{O}}^{2}_L)$. The class of $a^{-1}\circ (g^t)^{-1}(e_2) \mod N\hat{\Sh{O}}_L$ determines the level structure $\eta$. Finally, the polarization module $(\ker{\lambda})^{-1}$ is given by the fractional ideal $[\det(g)]\mathfrak{d}^{-1}$, where $[\det(g)]$ is the ideal generated by $\det(g)$.

The determinant map $\det \colon G \to \Res_{L/\Q}\Gm$ gives a bijection between the set of geometrically connected components of $\mathrm{Sh}_K(G)$ and the strict class group $\mathrm{Cl}^+(L)$ \cite[\S 2.3]{TianXiao}. For any fractional ideal $\mathfrak{c}$ coprime to $p$, let $\mathrm{Sh}^{\mathfrak{c}}_K(G)$ be the connected component of $\mathrm{Sh}_K(G)$ corresponding to the class of $\mathfrak{cd}$. This space is related to the moduli of $\mathfrak{c}$-polarized abelian varieties in the following manner. 

The moduli scheme $M(\mu_N, \mathfrak{c})$ is defined over $\Z[1/N]$. Consider the action of $\Sh{O}_L^{\times, +}$ on $M(\mu_N, \mathfrak{c})$ defined on points by $\epsilon \cdot (A, \iota, \lambda, \psi) = (A, \iota, \epsilon\lambda, \psi)$. Notice that for $\epsilon = \eta^2$, with $\eta \in U_N := 1 + N\Sh{O}_L$, the isomorphism $\eta \colon A \to A$ induces an isomorphism $\epsilon\cdot (A, \iota, \lambda, \psi) = \eta^*(A, \iota, \lambda, \psi) \simeq (A, \iota, \lambda, \psi)$. Thus the action of $\Sh{O}_L^{\times, +}$ factors through the finite quotient $\Gamma := \Sh{O}_L^{\times, +}/U_N^2$. We have then the following proposition.

\begin{prop}\label{P211}
There exists an isomorphism between the quotient $M(\mu_N, \mathfrak{c})(\C{})/\Gamma$ and $\mathrm{Sh}^{\mathfrak{c}}_K(G)(\C{})$. In other words $\mathrm{Sh}^{\mathfrak{c}}_K(G)(\C{})$ is a coarse moduli space over $\C{}$ of $\mathfrak{c}$-polarized abelian varieties with real multiplication by $\Sh{O}_L$. Moreover the quotient map $p \colon \bar{M}(\mu_N, \mathfrak{c}) \to \bar{M}(\mu_N, \mathfrak{c})/\Gamma$ is finite \'{e}tale with Galois group $\Gamma$.
\end{prop}

\begin{proof}
For the first part see \cite[Proposition 2.4]{TianXiao}. For the \'{e}taleness of the quotient see \cite[Lemma 8.1]{andreatta2016adic}.
\end{proof}

Choosing a set of representatives and letting $\GL{2}(L)_+$ the subgroup of matrices with totally positive determinant, one can write 
\[
\GL{2}(\A{L}{\infty}) = \bigsqcup_{[\mathfrak{c}]\in\mathrm{Cl}^+(L)} \GL{2}(L)_+\begin{pmatrix}
    \mathfrak{c} & 0 \\
    0 & \mathfrak{d}
\end{pmatrix}K_1(N).
\]
Let $x_{\mathfrak{c}} = \left(\begin{smallmatrix} \mathfrak{c} & 0 \\
    0 & \mathfrak{d}\end{smallmatrix}\right)$. Let $\Gamma_1(\mathfrak{c}, N) = x_{\mathfrak{c}}K_1(N)x_{\mathfrak{c}}^{-1}\cap \GL{2}(L)_+$, and let $\Gamma_1^1(\mathfrak{c}, N) = x_{\mathfrak{c}}K_1(N)x_{\mathfrak{c}}^{-1}\cap \SL{2}(L)$. Then one can identify $\mathrm{Sh}^{\mathfrak{c}}_K(G)(\C{}) = \Gamma_1(\mathfrak{c}, N)\backslash \mathfrak{h}^g$, and $M(\mu_N, \mathfrak{c})(\C{}) = \Gamma_1^1(\mathfrak{c},N)\backslash \mathfrak{h}^g$, where the action of the congruence subgroups on $\mathfrak{h}^g$ is via the usual M\"{o}bius action. 
    
    The map $\rho_{\epsilon}\colon M(\mu_N, \mathfrak{c})(\C{}) \to M(\mu_N, \mathfrak{c})(\C{})$ sending $(A, \iota, \lambda, \psi)\mapsto (A, \iota, \epsilon\lambda, \psi)$ corresponds to $\tau \mapsto \epsilon^{-1}\tau$ for any $\tau \in \mathfrak{h}^g$.

\subsubsection{Hilbert modular forms for the group \texorpdfstring{$G$}{G}}
Associated to a weight $(v, w) \in \Z[\Sigma]\times \Z$ we define an integral model for the sheaf of Hilbert modular forms of weight $(v,w)$ as follows. Let $L^{\textrm{Gal}}$ be the Galois closure of $L$. Let $R$ be an $\Sh{O}^{\textrm{Gal}}_{L,(p)}$-algebra. Let $t_L = \sum_{\tau\in \Sigma} \tau$. For $k = 2v + wt_L$, consider the sheaf $\omega^{(k,v)}_{\Sh{A}} =  \omega^k_{\Sh{A}}\otimes (\wedge^2 \mathrm{H}_{\Sh{A}})^{-v}$ on $\bar{M}(\mu_N, \mathfrak{c})_R$. Here the wedge product is taken as $\Sh{O}_{\bar{M}(\mu_N, \mathfrak{c})_R} \otimes \Sh{O}_L$-modules. The sheaf $\wedge^2 \mathrm{H}_{\Sh{A}}$ is canonically isomorphic to $\Sh{O}_{\bar{M}(\mu_N, \mathfrak{c})_R}\otimes \mathfrak{c}{\mathfrak{d}}^{-1}$ on $\bar{M}(\mu_N, \mathfrak{c})_R$, as can be seen from the Hodge filtration 
\[
0 \to \omega_{\Sh{A}} \to \mathrm{H}_{\Sh{A}} \to \omega_{\Sh{A}}^{\vee} \otimes \mathfrak{c} \to 0.
\]
Over $\C{}$, this canonical basis of $\wedge^2 \mathrm{H}_{\Sh{A}}$ corresponds to the canonical ordered basis $(\tau, 1)$ of $H_1(A_{\tau}, \Z)$, with $A_{\tau} = \C{g}/\mathfrak{c}^{-1}\cdot \tau + \mathfrak{d^{-1}}$, for any $\tau \in \mathfrak{h}^g$, under the Betti--de Rham isomorphism. For any $\epsilon \in \Sh{O}_{L}^{\times, +}$, the map $\rho_{\epsilon}\colon \bar{M}(\mu_N, \mathfrak{c})_{\C{}} \to \bar{M}(\mu_N, \mathfrak{c})_{\C{}}$, sending $(A, \iota, \lambda, \psi) \mapsto (A, \iota, \epsilon\lambda, \psi)$ induces a natural map $\rho_{\epsilon}^*\colon \wedge^2 \mathrm{H}_{\Sh{A}} \to \wedge^2 \mathrm{H}_{\Sh{A}}$ via pullback. If $h_{\text{can}}$ is the canonical basis of $\wedge^2 \mathrm{H}_{\Sh{A}}$, then $\rho_{\epsilon}^*h_{\text{can}} = \epsilon h_{\text{can}}$. This gives an action of $\Sh{O}_L^{\times, +}$ on $\omega_{\Sh{A}}^{(k,v)}$ which factors through $\Gamma$. We want to define the sheaf of Hilbert modular forms for the group $G$ to be the invariants $(p_*\omega_{\Sh{A}}^{(k,v)})^{\Gamma}$ under this action, with $p \colon \bar{M}(\mu_N, \mathfrak{c})_R \to \bar{M}(\mu_N, \mathfrak{c})_R/\Gamma$ the quotient map. 

Using the trivialization of $\wedge^2 \mathrm{H}_{\Sh{A}}$, the sheaves $\omega_{\Sh{A}}^{(k,v)}$ and $\omega_{\Sh{A}}^k$ are isomorphic as $\Sh{O}_{\bar{M}(\mu_N, \mathfrak{c})_R}$-modules. Define an action of $\Sh{O}_L^{\times, +}$ on $\omega^k_{\Sh{A}}$ as follows. For a section $f$ of $\omega^k_{\Sh{A}}$, define $\epsilon \cdot f$ as the section whose evaluation at points $(A_{/R}, \iota, \lambda, \psi, \omega)$ for an $\Sh{O}_L \otimes R$-generator $\omega$ of $\omega_{\Sh{A}}$ satisfies
\[
(\epsilon\cdot f)(A, \iota, \lambda, \psi, \omega) = \epsilon^{-v}f(A, \iota, \epsilon\lambda, \psi, \omega).
\]
We can then check that the action of $U_N^2$ is trivial on $\omega^k_{\Sh{A}}$. Indeed, for $\eta \in U_N$ and $\epsilon = \eta^2$ we have
\[
f(A, \iota, \lambda, \psi, \omega) = f(A, \iota, \epsilon \lambda, \psi, \eta \omega) = \eta^{-k} f(A, \iota, \epsilon\lambda, \psi, \omega) = \epsilon^{-v} f(A, \iota, \epsilon\lambda, \psi, \omega)
\]
which proves the claim. We remark that this action only depends on $k$ and not the pair $(v, w)$. 

\begin{defn}\label{D211}
Define the sheaf of Hilbert modular forms for the group $G$ of tame level $\mu_N$, $\mathfrak{c}$-polarization and weight $(v, w)$ with coefficients in $R$ to be \[\underline{\omega}_R^{(v,w)} := (p_*\omega^k_{\Sh{A}, R})^\Gamma\]
where $p \colon \bar{M}(\mu_N, \mathfrak{c})_R \to \bar{M}(\mu_N, \mathfrak{c})_R/\Gamma$ is the quotient map. Alternatively we will sometimes call them arithmetic Hilbert modular forms.
\end{defn}

Let $\mathrm{M}(\mu_N, \mathfrak{c}, k; R) := H^0(\bar{M}(\mu_N, \mathfrak{c})_R, \omega^k_{\Sh{A}})$ be the space of geometric Hilbert modular forms. Let $\mathrm{M}^G(\mu_N, \mathfrak{c}, (v, w); R) = H^0(\bar{M}(\mu_N, \mathfrak{c})_R/\Gamma, \underline{\omega}_R^{(v,w)})$ be the space of arithmetic Hilbert modular forms. If $\#\Gamma \in R^{\times}$, then $\mathrm{M}^G(\mu_N, \mathfrak{c}, (v, w); R) = \mathrm{M}(\mu_N, \mathfrak{c}, k = 2v + wt_L; R)^{\Gamma}$ can be realized as the image of the projector:
\[
e := \frac{1}{\#\Gamma} \sum_{\epsilon \in \Gamma} \epsilon.
\]

Let $\textrm{Frac}(L)^{(p)}$ be the group of fractional ideals which are coprime to $p$ and let $\textrm{Princ}(L)^{+, (p)}$ be the group of principal ideals generated by totally positive elements which are coprime to $p$. Then we have $\textrm{Frac}(L)^{(p)}/\textrm{Princ}(L)^{+,(p)} \simeq \mathrm{Cl}^+(L)$. Let $x \in L^{\times, +}$ be coprime to $p$. The map $(A, \iota, \lambda, \psi) \mapsto (A, \iota, x\lambda, \psi)$ induces a natural isomorphism $\rho_{(x,\mathfrak{c})}\colon \bar{M}(\mu_N, \mathfrak{c})_R \to \bar{M}(\mu_N, x\mathfrak{c})_R$. Over $\C{}$, this induces an isomorphism $\rho_{(x,\mathfrak{c})}\colon \mathrm{Sh}^{\mathfrak{c}}_K(G)(\C{}) \to \mathrm{Sh}^{x\mathfrak{c}}_K(G)(\C{})$, where 
\begin{align*}
\mathrm{Sh}^{\mathfrak{c}}_K(G)(\C{}) &= \GL{2}(L)_+\backslash \mathfrak{h}^g \times \GL{2}(L)_+\left(\begin{smallmatrix}
    \mathfrak{c} & 0 \\ 0 & \mathfrak{d}
\end{smallmatrix}\right)K_1(N)/K_1(N) \\ \mathrm{Sh}^{x\mathfrak{c}}_K(G)(\C{}) &= \GL{2}(L)_+\backslash \mathfrak{h}^g \times \GL{2}(L)_+\left(\begin{smallmatrix}
    x\mathfrak{c} & 0 \\ 0 & \mathfrak{d}
\end{smallmatrix}\right)K_1(N)/K_1(N).
\end{align*}
This isomorphism is given by $\left[\left(\tau, \left(\begin{smallmatrix}
    \mathfrak{c}&0\\ 0&\mathfrak{d}
\end{smallmatrix}\right)\right)\right] \mapsto \left[\left(x\tau, \left(\begin{smallmatrix}
    x\mathfrak{c}&0 \\ 0& \mathfrak{d}
\end{smallmatrix}\right)\right)\right]$. The Shimura variety $\mathrm{Sh}_K(G)(\C{})$ can then be written as 
\[
\mathrm{Sh}_K(G)(\C{}) = \left(\bigsqcup_{\mathfrak{c}\in\text{Frac}(L)^{(p)}} \mathrm{Sh}^{\mathfrak{c}}_K(G)(\C{})\right)/\text{Princ}(L)^{+,(p)}.
\]

Let $\omega_{\Sh{A}}^{(k,v)}(\mathfrak{c}) = \omega_{\Sh{A}}^k \otimes (\wedge^2 \mathrm{H}_{\Sh{A}})^{-v}$ on $\bar{M}(\mu_N,\mathfrak{c})_{\C{}}$. Then the identity
\[ 
[1]\colon A_{\tau,\mathfrak{c}} := \frac{\C{g}}{\mathfrak{c}^{-1}\cdot \tau + \mathfrak{d^{-1}}} \to \frac{\C{g}}{x^{-1}\mathfrak{c}^{-1}\cdot x\tau + \mathfrak{d^{-1}}} =: A_{x\tau,x\mathfrak{c}}
\]
induces a map $[1]^*\colon \rho^*_{(x, \mathfrak{c})}\omega_{\Sh{A}}^{(k,v)}(x\mathfrak{c})\to \omega_{\Sh{A}}^{(k,v)}(\mathfrak{c})$ via pullback. The canonical differential $\D\underline{z}$ of $\omega_{A_{x\tau, x\mathfrak{c}}}$ maps to the section $\D\underline{z}$ of $\omega_{A_{\tau,\mathfrak{c}}}$, and the canonical basis $h_{x\tau, x\mathfrak{c}}$ of $\wedge^2 \mathrm{H}_{A_{x\tau, x\mathfrak{c}}}$ maps to the section $x^{-1}h_{\tau, \mathfrak{c}}$ of $\wedge^2 \mathrm{H}_{A_{\tau,\mathfrak{c}}}$. In particular, 
\[
[1]^*(\D\underline{z}^k\otimes h_{x\tau,x\mathfrak{c}}^{-v}) = x^{v}(\D\underline{z}^k\otimes h_{\tau,\mathfrak{c}}^{-v}).
\]

This motivates the following definition. Consider the isomorphism $L_{(x\mathfrak{c}, \mathfrak{c})} \colon \mathrm{M}(\mu_N, x\mathfrak{c}, k; R) \xrightarrow{\sim} \mathrm{M}(\mu_N, \mathfrak{c}, k; R)$ given by 
\[
L_{(x\mathfrak{c}, \mathfrak{c})}(f)(A, \iota, \lambda, \psi, \omega) := x^{v}f(A, \iota, x\lambda, \psi, \omega).
\]
This map descends to an isomorphism $L_{(x\mathfrak{c}, x\mathfrak{c})} \colon \mathrm{M}^G(\mu_N, \mathfrak{c}, (v,w); R) \xrightarrow{\sim} \mathrm{M}^G(\mu_N, \mathfrak{c}, (v,w); R)$. 

\begin{defn}
Define the $R$-module of Hilbert modular forms for $G$ of tame level $N$ and weight $(v, w)$ to be
\[
\mathrm{M}^G(\mu_N, (v,w); R) := \left( \bigoplus_{\mathfrak{c} \in \textrm{Frac}(L)^{(p)}}\mathrm{M}^G(\mu_N, \mathfrak{c}, (v, w); R) \right)/\left( L_{(x\mathfrak{c}, \mathfrak{c})}(f) - f \right)_{x \in \textrm{Princ}(L)^{+, (p)}}.
\]
\end{defn}

Upon choosing representatives $\mathfrak{c}_1, \dots, \mathfrak{c}_{h^+_L}$ of $\mathrm{Cl}^+(L)$ in $\textrm{Frac}(L)^{(p)}$ we have a non-canonical isomorphism
\[
\mathrm{M}^G(\mu_N, (v, w); R) \simeq \bigoplus_{i = 1}^{h^+_L} \mathrm{M}^G(\mu_N, \mathfrak{c}_i, (v, w); R)
\]
which shows that $\mathrm{M}^G(\mu_N, (v,w); R)$ is a finite $R$-module.
\begin{rem}
For $\mathfrak{c} = \mathfrak{d}^{-1}$, the moduli scheme $M(\mu_N, \mathfrak{d}^{-1})$ is an integral model for the Shimura variety associated to the group $G^* := G \times_{\Res_{L/\Q}\Gm} \Gm$ where the arrow $G \to \Res_{L/\Q}\Gm$ is the determinant and $\Gm \to \Res_{L/\Q}\Gm$ is the diagonal embedding \cite{Rapoport}.
\end{rem}

\subsubsection{Adelic Hilbert modular forms}\label{S123}

Let $f \in \mathrm{M}^G(\mu_N, (v,w); \C{}) = (f_{\mathfrak{c}_i})_{i=1}^{h^+_L}$ be a Hilbert modular form as above. Then we can define an automorphic form $\phi_f \colon \GL{2}(L)\backslash \GL{2}(\A{L}{})/K_1(N) \to \C{}$ corresponding to $f$ as 
\[
\phi_f\left(u_{\infty}, \left(\begin{smallmatrix}
    \mathfrak{c}_i & 0 \\
    0 & \mathfrak{d} 
\end{smallmatrix} \right)\right) = f_{\mathfrak{c}_i}(u_{\infty}\cdot \mathbf{i})j_{k,v}(u_{\infty}, \mathbf{i})^{-1}
\]
where $u_{\infty} \in \GL{2}(L \otimes \R)_+$, $\mathbf{i} = (i,\dots, i) \in \C{g}$, and for $\gamma = \left(\begin{smallmatrix}
    a & b \\
    c & d
\end{smallmatrix} \right) \in \GL{2}(L\otimes \R)$, $j_{k,v}(\gamma, z) = (ad-bc)^{-v}(cz+d)^k$.

The function $\phi_f$ satisfies the following properties: \cite[\S2.1]{michele}
\begin{itemize}
    \item For every finite adelic point $x \in \GL{2}(\A{L}{\infty})$, the well-defined function $f_x\colon \mathfrak{h}^g \to \C{}$ given by $f_x(z) = \phi_f(u_{\infty}x)j_{k,v}(u_{\infty}, \mathbf{i})$ is holomorphic, where for each $z \in \mathfrak{h}^g$, we choose $u_{\infty} \in \GL{2}(L\otimes \R)_+$ such that $u_{\infty}\cdot\mathbf{i} = z$.
    \item For all adelic points $x \in \GL{2}(\A{L}{})$, and for all additive measures on $L \backslash \A{L}{}$, we have 
    \[
    \int_{L\backslash\A{L}{}} \phi_f\left(\left(\begin{smallmatrix}
        1 & a \\
        0 & 1
    \end{smallmatrix}\right)x\right)da = 0.
    \]
\end{itemize}

Let $\phi_f$ be an automorphic form as above. Choosing a set of representatives
\[
t_i = 
\begin{pmatrix}
  c_i^{-1} & 0 \\
  0 & 1
\end{pmatrix}, \quad i = 1, \dots, h_L^+
\]
for $\GL{2}(L)_+\backslash \GL{2}(\A{L}{})/K_1(N)$, we get modular forms $f_i \colon \mathfrak{h}^g \to \C{}$ defined as \[f_i(z) = \phi_f(u_{\infty}t_i)j_{k,v}(u_{\infty}, \mathbf{i}).\]
Each $f_i$ is a Hilbert modular form of $\mathfrak{c}_i\mathfrak{d}^{-1}$-polarization in the sense of Definition \ref{D211}, where $\mathfrak{c}_i$ is the ideal generated by $c_i$. The modular form $f_i$ has a Fourier expansion, which is the same as the $q$-expansion associated to the Tate object $\text{Tate}_{\mathfrak{d}^{-1}, \mathfrak{c}_i}(q) = \Gm \otimes \Sh{O}_L/\underline{q}(\mathfrak{c}_i)$ (see \S\ref{S4.3}) with level structure being given by the point $\exp(2\pi i/N) \otimes 1 \in \Gm \otimes \Sh{O}_L/\underline{q}(\mathfrak{c}_i)$.
\[
f_i(z) = \sum_{\xi \in (\mathfrak{c}_i\mathfrak{d}^{-1})_+} a(\xi, f_i)e_L(\xi z), \quad e_L(\xi z) = \exp(2\pi i \sum_{\tau \in \Sigma} \tau(\xi)z_{\tau}).
\]
Fix a finite extension $L_0$ of $L^{\textrm{Gal}}$ such that for any ideal $\mathfrak{a} \subset \Sh{O}_L$, and any embedding $\tau \in \Sigma$, $\mathfrak{a}^{\tau}\Sh{O}_{L_0}$ is principal. Choose a generator $\{\mathfrak{q}^{\tau}\} \in \Sh{O}_{L_0}$ of $\mathfrak{q}^{\tau}\Sh{O}_{L_0}$ for each prime ideal $\mathfrak{q}$, and extend it to all fractional ideals multiplicatively. Fix an idele $\mathrm{d} \in \A{L}{\infty, \times}$ whose ideal is the different $\mathfrak{d}$ of $L/\Q$.

Every idele $y \in \A{L,+}{\times} = \A{L}{\infty,\times}L_{\infty,+}^{\times}$ can be written as $y = \xi c_i^{-1}\mathrm{d}u$ for $\xi \in L_{+}^{\times}$ and $u \in \det{K_{1}({N})}L_{\infty, +}^{\times}$. Define two functions $c( \cdot, \phi_f) \colon \A{L,+}{\times} \to \C{}$ and $c_p( \cdot,\phi_f) \colon \A{L,+}{\times} \to \bar{\Q}_p$ as follows.
\[
c(y, \phi_f) := a(\xi, f_i)\{y^{v-t_L}\}\xi^{t_L - v}|c_i|_{\A{L}{}} \quad  \quad c_p(y, \phi_f) := a(\xi, f_i)y_p^{v-t_L}\xi^{t_L - v}\Sh{N}_L(c_i)^{-1}
\]
if $y \in \hat{\Sh{O}}_LL_{\infty,+}^{\times}$ and $0$ otherwise. Here $\Sh{N}_L$ is defined by $y \mapsto y_p^{-t_L}|y^{\infty}|_{\A{L}{}}^{-1}$. The function $c_p( \cdot, \phi_f)$ makes sense only if the coefficients $a(\xi, f_i)$ are algebraic for all $i$. Moreover, for our choice of the $\mathfrak{c}_i$ as being coprime to $p$, we have 
\[
c_p(y, \phi_f) = c(y, \phi_f)\{y^{t_L - v}\}y_p^{v-t_L}.
\]

\begin{theorem}\label{T28}
Consider the map $e_L \colon \C{\Sigma} \to \C{\times}$ defined by $e_L(z) = \exp(2\pi i \sum_{\tau \in \Sigma} z_{\tau})$ and the unique additive character of the adeles $\chi_L \colon \A{L}{}/L \to \C{\times}$ which satisfies $\chi_L(x_{\infty}) = e_L(x_{\infty})$. Each Hilbert cuspform of weight $(v,w)$ has an adelic $q$-expansion of the form
\[
\phi_f\left(\left(\begin{smallmatrix}y & x \\ 0 & 1  \end{smallmatrix}\right)\right) = |y|_{\A{L}{}}\sum_{\xi \in L_+} c(\xi y \mathrm{d}, \phi_f)\{(\xi y\mathrm{d})^{t_L - v}\}(\xi y_{\infty})^{v-t_L}e_L(\mathbf{i}\xi y_{\infty})\chi_L(\xi x)
\]
for $y \in \A{L,+}{\times}$, $x \in \A{L}{\times}$, where $c(\cdot, \phi_f) \colon \A{L,+}{\times} \to \C{}$ vanishes outside $\hat{\Sh{O}}_LL_{\infty,+}^{\times}$ and depends only on the coset $y^{\infty}\det{K_{1}({N})}L_{\infty,+}^{\times}$. The adelic $q$-expansion agrees with the Fourier expansions of the $f_i$ in the following sense.
\[
\phi_f\left(\left(\begin{smallmatrix}y_{\infty} & x_{\infty} \\ 0 & 1 \end{smallmatrix}\right)t_i\right) = y_{\infty}^{v}\sum_{\xi \in (\mathfrak{c}_i\mathfrak{d}^{-1})_+} a(\xi, f_i)e_L(\xi z).
\]
Moreover, formally replacing $e_L(\mathbf{i}\xi y)\chi_L(\xi x)$ by $q^{\xi}$, and $(\xi y_{\infty})^{v-t_L}$ by $(\xi \mathrm{d}_p y_p)^{v-t_L}$, we have the $q$-expansion
\[
\phi_f = \Sh{N}(y)^{-1}\sum_{\xi \in L_{+}} c_p(\xi y\mathrm{d}, \phi_f)q^{\xi}.
\]
\end{theorem}

\begin{proof}
\cite[Theorem 1.1]{Hida}.
\end{proof}

\paragraph{\textbf{Atkin--Lehner involution.}}
Let $f \in \mathrm{M}^G(\mu_N, (v,w); \C{})$ be a modular form as above. The Atkin--Lehner involution is described in the language of adelic Hilbert modular forms as follows \cite[655]{Shimura}. For any $2\times 2$ matrix $x$, let $x^{\iota}$ be the adjugate matrix, i.e. $xx^{\iota} = x^{\iota}x = \det{x}$.

\begin{defn}
    The Atkin--Lehner involution $w_N$ on the space $\mathrm{M}^G(\mu_N, (v,w);\C{})$ of Hilbert modular forms is an operator, such that if $\phi_f \colon \GL{2}(\A{L}{}) \to \C{}$ is a Hilbert modular form, then 
    \[
    (\phi_f|w_N)(x) = \phi_f\left(x^{-\iota}\left(\begin{smallmatrix}
        0 & -1 \\
        N & 0
    \end{smallmatrix}\right)\right).
    \]
\end{defn}

One can give a geometric definition of the Atkin--Lehner operator, which allows it to be defined for families of Hilbert modular forms.

Fix a base ring $R$ that is an $\Sh{O}_{L,(p)}^{\mathrm{Gal}}$ algebra containing a primitive $N$-th root $\zeta_N$ of $1$. Let $\mathfrak{c}$ be an integral ideal coprime to $N$, and let $M(\mu_N, \mathfrak{c}^{-1}\mathfrak{d}^{-1})/\Sp{R}$ be the Hilbert modular variety of $\mathfrak{c}^{-1}\mathfrak{d}^{-1}$-polarized abelian varieties over $R$. The Atkin--Lehner involution is induced via pullback by a morphism $w_N \colon M(\mu_N, \mathfrak{c}^{-1}\mathfrak{d}^{-1})\to M(\mu_N, \mathfrak{cd}^{-1}N)$ which is defined on points as follows. 

Let $(A,\iota, \lambda, \psi) \in M(\mu_N, \mathfrak{c}^{-1}\mathfrak{d}^{-1})$. Suppose $P = \psi(1) \in A[N]$. Let $A' = A\otimes \mathfrak{c}^{-1}/(P)$, i.e. the quotient of $A\otimes \mathfrak{c}^{-1}$ by the cyclic subgroup generated by the image of $P$. We note that $\lambda$ induces a $\mathfrak{cd}^{-1}N$-polarization $\lambda'$ on $A'$. Since $A\otimes \mathfrak{c}^{-1} = A^{\vee} \otimes \mathfrak{d}$, the Weil pairing 
\[A[N] \xrightarrow{\sim} \Hom_{\Sh{O}_L}(A^{\vee}[N], \mu_N \otimes \mathfrak{d}^{-1}) = \Hom_{\Sh{O}_L}(A^{\vee}\otimes \mathfrak{d}[N], \mu_N \otimes \Sh{O}_L)
\]
defines a point $P^{\vee} \in A^{\vee}\otimes \mathfrak{d}[N]$, such that $\langle P, P^{\vee}\rangle = \zeta_N \otimes 1 \in \mu_N \otimes \Sh{O}_L$. Then $P^{\vee}$ defines a level $N$ structure $\psi'$ on $A'$, which thus gives a point $(A',\iota',\lambda', \psi') \in M(\mu_N, \mathfrak{cd}^{-1}N)$. Letting $\pi \colon A \to A'$ be the projection, the Atkin--Lehner operator is defined as
\begin{align*}
    w_N \colon \mathrm{M}(\mu_N, \mathfrak{cd}^{-1}N, (v,w); R) &\to \mathrm{M}(\mu_N, \mathfrak{c}^{-1}\mathfrak{d}^{-1}, (v,w); R) \\
    (f|w_N)(A,\iota, \lambda, \psi, \omega) &= f(A', \iota', \lambda', \psi', (\pi^*)^{-1}\omega).
\end{align*}

This operator commutes with the action of $\Gamma$ and $\text{Princ}(L)^{+,(p)}$, and thus defines an operator $w_N$ on $\mathrm{M}^G(\mu_N, (v,w); R)$.

\begin{theorem}
    If $f \in \mathrm{M}^G(\mu_N, (v,w);\C{})$ is a primitive newform, then $f|w_N$ is a constant times $f^*$ where $f^*$ is the modular form whose adelic $q$-expansion coefficients are the complex conjugates of $f$.
\end{theorem}

\begin{proof}
    \cite[Proposition 2.10]{Shimura}.
\end{proof}

\subsection{The partial Igusa tower}
Fix an $I = [p^a, p^b]$ and $\alpha \in \mathfrak{m} \setminus \mathfrak{m}^2$. We now work over $\X$. Recall the ideal $\mathrm{Hdg}_r$ was given locally by $(\alpha, \mathrm{Hdg}^{p^{r+1}})$ for a local lift $\mathrm{Hdg}$ of the Hasse invariant. Let $g_r \colon \Xra{r} \to \X$ be the open subscheme of the blow-up of $\X$ with respect to the ideal $\mathrm{Hdg}_r$, where the inverse image ideal is generated by $\mathrm{Hdg}^{p^{r+1}}$. For any integer $n$ with $1 \leq n \leq r$ if $I = [0,1]$ and $1 \leq n \leq a+r$ let $\lambda = \mathrm{Hdg}^{\frac{p^n-1}{p-1}}$. Note that $\frac{p}{\lambda} \in \Sh{O}_{\Xra{r}}$.

\begin{prop}
For $I, r, n, \alpha$ as above the semiabelian scheme $\Sh{A} \to \Xra{r}$ has a canonical subgroup $H_n$ of order $p^n$ \cite[Appendice A]{Andreatta2018leHS}. This is a finite, locally free subgroup scheme that satisfies the following properties:
\begin{enumerate}
    \item $H_n$ lifts $\ker{F^n}$ modulo $\frac{p}{\lambda}$.
    \item For any $\alpha$-adically complete admissible $\Lambda^0_{\alpha,I}$-algbera $R$, together with a morphism $f \colon \Spf{R} \to \Xra{r}$, 
    \[
    H_n(R) = \{ s \in \Sh{A}[p^n](R) \ | \ s\ \mathrm{mod}\ \frac{p}{\lambda} \in \ker{F^n}\}.
    \]
    \item Suppose $L_n = \Sh{A}[p^n]/H_n$. Then $\omega_{L_n}$ is killed by $\lambda$, and we have $\omega_{L_n} \simeq \omega_{\Sh{A}}/\lambda \omega_{\Sh{A}}$.
    \item $\Sh{A}^{\vee}[p^n]/H_n(\Sh{A}^{\vee}) \simeq H_n^{\vee}$ through the Weil pairing and it is \'{e}tale over the adic generic fibre $\xra{r}$ of $\Xra{r}$.
\end{enumerate}
\end{prop}

\begin{proof}
\cite[Appendice A]{Andreatta2018leHS}.
\end{proof}

\begin{defn}
For every $r, n$ as above define $\Sh{IG}_{n,r,I} \to \xra{r}$ to be the adic space classifying isomorphisms $\Sh{O}_L/p^n\Sh{O}_L \xrightarrow{\sim} H_n^{\vee}$ of the group scheme $H_n^{\vee} \to \xra{r}$. Define $\Ig{n}{r} \to \Xra{r}$ to be the normalization of $\Xra{r}$ in $\Sh{IG}_{n,r,I}$. 
\end{defn}

\begin{prop}
$\Sh{IG}_{n,r,I} \to \xra{r}$ is an \'{e}tale, Galois morphism with Galois group $(\Sh{O}_L/p^n\Sh{O}_L)^{\times}$. The morphism $\mathfrak{IG}_{n,r,I} \to \Xra{r}$ is finite and is endowed with an action of $(\Sh{O}_L/p^n\Sh{O}_L)^{\times}$ induced by the action on the generic fibre.
\end{prop}

\begin{proof}
\cite[\S3.3.1]{andreatta2016adic}.
\end{proof}

(Note we suppressed the index $\alpha$ in our notation for the partial Igusa tower to avoid clumsiness.)

%% file: Hilbert2.tex
\section{Splitting of de Rham sheaf}\label{S2}
In the following we put an overline on the names of objects (abelian schemes, sheaves etc.) to denote they are obtained by base change along the closed immersion $i \colon X_{\F{p}} \xhookrightarrow{} \mathfrak{X}$. Denote by $\bar{\pi} \colon \bar{\Sh{A}} \to X_{\F{p}}$  the base change of $\pi$ along $i$. Let $\bar{\omega}_{\Sh{A}} := i^*\omega_{\Sh{A}} = \bar{\omega}_{\Sh{A}}$. The Verschiebung $V \colon \bar{\Sh{A}}^{(p)} \to \bar{\Sh{A}}$ induces a map on the Lie algebra $HW \colon \omega^{\vee}_{\bar{\Sh{A}}^{(p)}} \to {\omega}^{\vee}_{\bar{\Sh{A}}}$, whose determinant is the Hasse invariant $\mathrm{Ha} \in (\Lambda^g \bar{\omega}_{\Sh{A}})^{\otimes (p-1)}$. Let $\underline{\mathrm{Ha}} := \mathrm{Ha} \cdot (\Lambda^g \bar{\omega}_{\Sh{A}})^{\otimes (1-p)}$ be the ideal generated by the values of $\mathrm{Ha}$. This is an invertible ideal with zeroes of order $1$ along each of the prime divisors that appear in $\mathrm{Div}(\underline{\mathrm{Ha}})$ (Theorem \ref{T2101}). Denote by $\bar{\mathrm{H}}_{\Sh{A}}$ the pullback of $\mathrm{H}_{\Sh{A}}$ along $i$. Let $j \colon X_{\F{p}}^{\textrm{ord}} \xhookrightarrow{} X_{\F{p}}$ be the ordinary locus, which is the open subscheme of $X_{\F{p}}$ where $\underline{\mathrm{Ha}} = \Sh{O}_{X_{\F{p}}}$. 

Let $\varphi \colon X_{\F{p}} \to X_{\F{p}}$ be the Frobenius. The Frobenius induces a $\varphi$-linear endomorphism of $\bar{\mathrm{H}}_{\Sh{A}}$.

\begin{prop}
Over the ordinary locus $X^{\textrm{ord}}_{\F{p}}$ we have the unit root splitting which is a canonical splitting $\psi_{\Frob} \colon j^*\bar{\mathrm{H}}_{\Sh{A}} \to j^*\bar{\omega}_{\Sh{A}}$ of the Hodge filtration on $\bar{\mathrm{H}}_{\Sh{A}}$, that respects the Frobenius action. The kernel of $\psi_{\Frob}$ is called the unit root subspace. It is characterized by the property that it is stable under the Frobenius action and Frobenius acts invertibly on it.
\end{prop}

\begin{proof}
Suppose $\Sp{R} \subset X_{\F{p}}$ is a local chart for which $\bar{\omega}_{\Sh{A}}, \bar{\mathrm{H}}_{\Sh{A}}$ are trivial and choose a basis compatible with the Hodge filtration. With respect to such a basis we can write the matrix of the Frobenius action on $\bar{\mathrm{H}}_{\Sh{A}}$ in $g \times g$ blocks as follows.
\[
\textrm{Frob} = 
\begin{pmatrix}
0 && C \\
0 && HW.
\end{pmatrix}
\]
Here we abuse notation to write $HW$ for the matrix corresponding to the $\varphi$-linear map induced on $\bar{\omega}_{\Sh{A}}^{\vee}$ by $\mathrm{F}_{\textrm{abs}} \colon \bar{\Sh{A}} \to \bar{\Sh{A}}$. For the base change of $j_*j^*\bar{\mathrm{H}}_{\Sh{A}}(R) = \bar{\mathrm{H}}_{\Sh{A}}(R)[1/\Ha]$ given by the matrix 
\[
P = 
\begin{pmatrix}
\mathrm{Id} && C\cdot HW^{-1} \\
0 && \mathrm{Id}
\end{pmatrix},
\]
the matrix of Frobenius becomes
\[
P^{-1} \textrm{Frob} P = 
\begin{pmatrix}
0 && 0 \\
0 && HW
\end{pmatrix}
\]
Note that $P$ is only defined over the ordinary locus. Hence we have a splitting $\psi_{\textrm{Frob}} \colon j^*\bar{\mathrm{H}}_{\Sh{A}} \to j^*\bar{\omega}_{\Sh{A}}$ of the Hodge filtration over the ordinary locus that respects the Frobenius action. The kernel of this splitting is \emph{uniquely} characterized by the fact that it is stable under the Frobenius action and Frobenius acts invertibly on it.
\end{proof}

Consider the map $\psi \colon \bar{\mathrm{H}}_{\Sh{A}} \to j_*j^*\bar{\mathrm{H}}_{\Sh{A}} \xrightarrow{\psi_{\textrm{Frob}}} j_*j^*\bar{\omega}_{\Sh{A}}$. Then let $\bar{\mathrm{H}}'_{\Sh{A}} := \psi^{-1}\bar{\omega}_{\Sh{A}}$. The inclusion $\bar{\omega}_{\Sh{A}} \to \bar{\mathrm{H}}'_{\Sh{A}}$ admits a retraction given by the map $\psi$. As a subsheaf of $\bar{\mathrm{H}}_{\Sh{A}}$ containing $\bar{\omega}_{\Sh{A}}$, $\bar{\mathrm{H}}'_{\Sh{A}}$ is equipped with the induced Hodge filtration. In the following lemma we describe the 1st graded piece of this Hodge filtration.

\begin{lemma}\label{L601}
The sheaf $\bar{\mathrm{H}}'_{\Sh{A}}$ sits in the following split exact sequence.
\begin{equation}\label{eq:2}
0 \to \bar{\omega}_{\Sh{A}} \to \bar{\mathrm{H}}'_{\Sh{A}} \to HW(\omega^{\vee}_{\bar{\Sh{A}}^{(p)}}) \to 0.
\end{equation}
\end{lemma}

\begin{proof}
Choose a local chart $\Sp{R}$ as above. Suppose $e_1, \dots, e_g, f_1, \dots, f_g$ form an $R$-basis of $\bar{\mathrm{H}}_{\Sh{A}}$ such that $e_1, \dots, e_g$ span $\bar{\omega}_{\Sh{A}}$ and the images of $f_1, \dots, f_g$ span $\bar{\omega}^{\vee}_{\Sh{A}}$. Assume that the matrix of Frobenius with respect to this basis is given as 
\[
\textrm{Frob} = 
\begin{pmatrix}
0 && C \\
0 && HW.
\end{pmatrix}
\]
Then the matrix of $\psi_{\text{Frob}}$ with respect to this basis is given by $\psi_{\text{Frob}} \colon R[1/\mathrm{Ha}]^{2g} \to R[1/\mathrm{Ha}]^g$
\[
\psi_{\text{Frob}} = \begin{pmatrix}
    1 & -C\cdot HW^{-1}
\end{pmatrix}.
\]
Thus $\bar{\mathrm{H}}'_{\Sh{A}}$ is given by all elements $(x_1, \dots, x_{2g})^t$ such that $(\begin{smallmatrix}
    1 & -C\cdot HW^{-1}
\end{smallmatrix})(x_1, \dots, x_{2g})^t \in R^g$. In particular, we want to solve for $v = (x_{g+1}, \dots, x_{2g}) \in R^g$, such that $C\cdot HW^{-1}(v) \in R^g$. Letting $HW^{\iota}$ be the matrix such that $HW \cdot HW^{\iota} = HW^{\iota} \cdot HW = \mathrm{Ha}$, we then have $C\cdot HW^{-1}(v) \in R^g \iff C \cdot HW^{\iota} \in \underline{\mathrm{Ha}}\cdot R^g$. Therefore, $HW^{\iota}(v)$ lies in the kernel of $C \textrm{ mod }\underline{\mathrm{Ha}}$ as well as in the kernel of $HW \textrm{ mod }\underline{\mathrm{Ha}}$. Thus in particular, denoting by $\Sp{k(y)} \to Y_{\F{p}}$ a generic point of a prime divisor of $\underline{\mathrm{Ha}}$, 
\[
\begin{pmatrix}
    0 & C \\
    0 & HW
\end{pmatrix}\begin{pmatrix}
    0 \\
    HW^{\iota}(v)
\end{pmatrix} \otimes_R k(y) = 0.
\]
But the image of Frobenius over $k(y)$ has rank $g$, and hence its kernel is precisely $\omega_{\bar{A}_{k(y)}}$. Thus $HW^{\iota}(v) \in \underline{\mathrm{Ha}}\cdot R^g$, implying $v \in HW(R^g)$. This proves the lemma.
\end{proof}

\begin{cor}
The sheaf $\bar{\mathrm{H}}'_{\Sh{A}}$ is stable under the $\Sh{O}_L$ action. It is a locally free $\Sh{O}_L \otimes \Sh{O}_{X_{\F{p}}}$-module of rank 2 and $\mathrm{H}'_{\Sh{A}} = \bar{\omega}_{\Sh{A}} \oplus \Frob(\bar{\mathrm{H}}_{\Sh{A}})$.
\end{cor}

\begin{proof}
The subsheaf $\Frob(\bar{\mathrm{H}}_{\Sh{A}})$ is the image of the map $F^*_{\Sh{A}} \colon \mathrm{H}_{\bar{\Sh{A}}^{(p)}} \to \bar{\mathrm{H}}_{\Sh{A}}$ induced by the relative Frobenius. It is killed by the unit root splitting and maps surjectively onto $HW(\omega^{\vee}_{\bar{\Sh{A}}^{(p)}})$. Hence $\bar{\mathrm{H}}'_{\Sh{A}} = \bar{\omega}_{\Sh{A}} \oplus \Frob(\bar{\mathrm{H}}_{\Sh{A}})$. Since the relative Frobenius commutes with the $\Sh{O}_L$-action, $HW(\omega^{\vee}_{\bar{\Sh{A}}^{(p)}})$ is stable under the $\Sh{O}_L$-action and hence so is $\bar{\mathrm{H}}'_{\Sh{A}}$. Moreover $HW \colon \omega^{\vee}_{\bar{\Sh{A}}^{(p)}} \to \bar{\omega}_{\Sh{A}}$ is an $\Sh{O}_L \otimes \Sh{O}_{X_{\F{p}}}$-linear map of invertible $\Sh{O}_L \otimes \Sh{O}_{X_{\F{p}}}$-modules such that $\Nm_{\Sh{O}_L \otimes \Sh{O}_{X_{\F{p}}}/\Sh{O}_{X_{\F{p}}}}HW = \det{HW} = \mathrm{Ha}$ is a non-zero divisor. Hence $HW(\omega^{\vee}_{\bar{\Sh{A}}^{(p)}})$ is an invertible $\Sh{O}_L \otimes \Sh{O}_{X_{\F{p}}}$-module.
\end{proof}

In the following we will construct a locally free $\Sh{O}_L \otimes \Sh{O}_{\Xra{r}}$ subsheaf $\mathrm{H}'_{\Sh{A}} \subset \mathrm{H}_{\Sh{A}}$ of rank 2, together with the induced Hodge filtration such that its reduction modulo a small power of $p$ will give us the split exact sequence (\ref{eq:2}).

Let $i \colon \X/(p) \xhookrightarrow{} \X$ be the base change of $X_{\F{p}} \xhookrightarrow{} \mathfrak{X}$ to $\X$. Let $i_0 \colon \Xra{r}/(p\mathrm{Hdg}^{-1}) \xhookrightarrow{} \Xra{r}$ be the closed subscheme defined by the ideal $\frac{p}{\mathrm{Hdg}}$. Thus we have a commutative diagram as follows.
\[
\begin{tikzcd}
\Xra{r}/(p\mathrm{Hdg}^{-1}) \arrow[d, "q"] \arrow[r, "i_0"] & \Xra{r} \arrow[d, "g_r"] \\
\X/(p) \arrow[r, "i"]   & \X                   
\end{tikzcd}
\]
Let $\bar{\omega}^{\vee}_{\Sh{A},0} := HW(\omega_{\bar{\Sh{A}}^{(p)}})$ where we  now denote by $\bar{\Sh{A}}$ the pullback of $\Sh{A}$ along $i \colon \X/(p) \xhookrightarrow{} \X$. Let  $\tilde{\omega}^{\vee}_{\Sh{A}} := (i^{\sharp})^{-1}(i_*\bar{\omega}^{\vee}_{\Sh{A},0})$ where $i^{\sharp} \colon \omega^{\vee}_{\Sh{A}} \to i_*i^*\omega^{\vee}_{\Sh{A}}$ is the unit of the adjunction. Note that $\tilde{\omega}^{\vee}_{\Sh{A}}$ is stable under the $\Sh{O}_L$-action.

\begin{lemma}\label{L602}
$\omega^{\vee}_{\Sh{A},0} := \im{(g_r^*{\tilde{\omega}^{\vee}_{\Sh{A}}} \to g_r^*\omega^{\vee}_{\Sh{A}})}$ is a locally free $\Sh{O}_L \otimes \Sh{O}_{\Xra{r}}$-module of rank $1$.
\end{lemma}

\begin{proof}
Choose a local chart $\Spf{R} = U \subset \X$, such that $\omega^{\vee}_{\Sh{A}}$ is free as an $\Sh{O}_L \otimes R$-module over $\Spf{R}$, and $\underline{\mathrm{Ha}}$ is free over $\Sp{R/(p)}$. Let $v$ be an $\Sh{O}_L \otimes R$ basis of ${\omega^{\vee}_{\Sh{A}}}_{|U}$ and let $(\Sh{O}_L \otimes R)\cdot v =  \oplus_{i=1}^g Rv_i$ be the decomposition induced by the splitting of $\Sh{O}_L$ in $R$. Let $\bar{v}$ be the image of $v$ in $\bar{\omega}^{\vee}_{\Sh{A}} := i^*{\omega}^{\vee}_{\Sh{A}}$. Let $\bar{w} = HW(\bar{v})$ and pick a lift $w$ of $\bar{w}$. Let $w = (w_i)_i$ be its components. Then ${\tilde{\omega}^{\vee}_{\Sh{A}}} = \sum_{i = 1}^g \Sh{O}_Uw_i + p\omega^{\vee}_{\Sh{A}}$. Consider the $\Sh{O}_L$-linear map $\widetilde{HW}$ which sends $v \mapsto w$ and which reduces to $HW$ mod $p$. Then $\det{\widetilde{HW}}$ is a lift of $\mathrm{Ha} = \det{HW}$. Keeping with previous notation, we will call this lift $\mathrm{Hdg}$. Consider now $V = g_r^{-1}(U) = \Spf{R\langle \frac{\alpha}{\mathrm{Hdg}^{p^{r+1}}}}\rangle$. Using $\frac{p}{\mathrm{Hdg}} \in \Sh{O}_{\Xr{r}}$, we see that over $\Spf{R\langle \frac{\alpha}{\mathrm{Hdg}^{p^{r+1}}}\rangle}$, $p\cdot g_r^*\omega^{\vee}_{\Sh{A}} \subset \sum_{i=1}^n \Sh{O}_{V} g_r^*(w_i)$ as submodules of $g_r^*\omega^{\vee}_{\Sh{A}}$. Thus $\omega^{\vee}_{\Sh{A},0} = \sum_{i=1}^n \Sh{O}_{V} g_r^*(w_i)$. The sum is direct because $\mathrm{Hdg}$ is a non-zero divisor in $\Sh{O}_{\Xr{r}}$. It is clearly stable under the $\Sh{O}_L$-action. Since locally $HW$ can be seen as a non-zero divisor in $\Sh{O}_L \otimes \Sh{O}_{\X/(p)}$, the same is true of the lift $\widetilde{HW}$, and since $\omega^{\vee}_{\Sh{A},0} = \widetilde{HW}\cdot\omega^{\vee}_{\Sh{A}}$, the lemma follows.
\end{proof}

Following the proof of Lemma \ref{L602}, let $\widetilde{HW}$ be the $\Sh{O}_L \otimes \Sh{O}_{\Xra{r}}$-ideal sheaf defined by $\omega^{\vee}_{\Sh{A},0} = \widetilde{HW}\cdot \omega^{\vee}_{\Sh{A}}$.

\begin{defn}\label{D601}
Define $\mathrm{H}'_{\Sh{A}}$ to be the inverse image of $\omega^{\vee}_{\Sh{A},0}$ in $\mathrm{H}_{\Sh{A}}$ under the projection coming from the Hodge filtration.
\end{defn}

We have the following commutative diagram of sheaves over $\Xra{r}$.
\[
\begin{tikzcd}[equals/.style = {draw = none, "=" description, sloped}]
0 \arrow[r] & \omega_{\mathcal{A}} \arrow[r] \arrow[d, equals] & \mathrm{H}'_{\mathcal{A}} \arrow[d] \arrow[r] & \omega^{\vee}_{\Sh{A},0} \arrow[d] \arrow[r] & 0 \\
0 \arrow[r] & \omega_{\mathcal{A}} \arrow[r]           & \mathrm{H}_{\Sh{A}} \arrow[r] & \omega^{\vee}_{\Sh{A}} \arrow[r]              & 0
\end{tikzcd}
\]

\begin{prop}\label{P601}
There is an isomorphism of $\Sh{O}_L \otimes \Sh{O}_{\Xra{r}/(p\mathrm{Hdg}^{-1})}$-modules $i_0^*\omega^{\vee}_{\Sh{A},0} \xrightarrow{\sim} q^*\bar{\omega}^{\vee}_{\Sh{A},0}$, that commutes with the induced maps to $\bar{\omega}^{\vee}_{\Sh{A}}$. (We abuse notation to denote $q^*\bar{\omega}^{\vee}_{\Sh{A}}$ by $\bar{\omega}^{\vee}_{\Sh{A}}$.)
\[
\begin{tikzcd}[equals/.style = {draw = none, "=" description, sloped}]
i_0^*\omega^{\vee}_{\Sh{A},0} \arrow[d] \arrow[r, "\sim"] & q^*\bar{\omega}^{\vee}_{\Sh{A},0} \arrow[d] \\
\bar{\omega}^{\vee}_{\Sh{A}} \arrow[r, equals]             & \bar{\omega}^{\vee}_{\Sh{A}}              
\end{tikzcd}
\]
\end{prop}

\begin{proof}
In fact we prove the following. The natural surjective map $i^*{\tilde{\omega}^{\vee}_{\Sh{A}}} \to \bar{\omega}^{\vee}_{\Sh{A},0}$ induces by pullback a surjective map $q^*i^*{\tilde{\omega}^{\vee}_{\Sh{A}}} = i_0^*g_r^*{\tilde{\omega}^{\vee}_{\Sh{A}}} \to q^*\bar{\omega}^{\vee}_{\Sh{A},0}$ that commutes with the induced maps to $\bar{\omega}^{\vee}_{\Sh{A}}$. We will show that this map factors naturally as $i_0^*g_r^*{\tilde{\omega}^{\vee}_{\Sh{A}}} \to i_0^*\omega^{\vee}_{\Sh{A},0} \to q^*\bar{\omega}^{\vee}_{\Sh{A},0}$, and the last two sheaves being both locally free of rank $g$, the last arrow is an isomorphism. Also since $i_0^*g_r^*\tilde{\omega}^{\vee}_{\Sh{A}} \to q^*\bar{\omega}^{\vee}_{\Sh{A},0}$ is $\Sh{O}_L$-linear and $i_0^*g_r^*\tilde{\omega}^{\vee}_{\Sh{A}} \to i_0^*\omega^{\vee}_{\Sh{A},0}$ is surjective $\Sh{O}_L$-linear, the induced isomorphism is $\Sh{O}_L$-linear too. This will be the isomorphism claimed in the proposition.

We use the notation of the proof of Lemma \ref{L602}, except that to avoid clumsiness we write $v_i$ (resp. $w_i$) instead of $g_r^*(v_i)$ (resp. $g_r^*(w_i)$). Since ${\tilde{\omega}^{\vee}_{\Sh{A}}}$ is generated by the $w_i$ and $pv_i$ for $i = 1, \dots, g$, there is a surjective map $\Sh{O}_{V}^{2g} \to g_r^*{\tilde{\omega}^{\vee}_{\Sh{A}}}$ that sends $e_i \mapsto w_i$ for $1 \leq i \leq g$, and $e_j \mapsto pv_j$ for $g+1 \leq j \leq 2g$. Using the basis $\bar{w}_i$ for $\bar{\omega}^{\vee}_{\Sh{A},0}$, we have a surjective map $M \colon \Sh{O}^{2g}_{i_0^{-1}V} \to \Sh{O}^g_{i_0^{-1}V}$ given by $e_i \mapsto e_i$ for $1 \leq i \leq g$, and $e_j \mapsto 0$ for $g+1 \leq j \leq 2g$ and which induces the map $i_0^*g_r^*{\tilde{\omega}^{\vee}_{\Sh{A}}} \to q^*\bar{\omega}^{\vee}_{\Sh{A},0}$. On the other hand, the images of $g_r^*(w_i)$ form a basis for $\omega^{\vee}_{\Sh{A},0}$. With respect to this basis, $pv_i = \frac{p}{\mathrm{Hdg}}\mathrm{adj}(\widetilde{HW})(e_i)$. Thus the surjective map $g_r^*{\tilde{\omega}^{\vee}_{\Sh{A}}} \to \omega^{\vee}_{\Sh{A},0}$ is induced by the map $N \colon \Sh{O}^{2g}_V \to \Sh{O}^g_V$ that sends $e_i \mapsto e_i$ for $1 \leq i \leq g$, and $e_{i+g} \mapsto \frac{p}{\mathrm{Hdg}}\mathrm{adj}(\widetilde{HW})(e_i)$ for $1 \leq i \leq g$. Suppose $(a_1, \dots, a_{2g}) \in \ker{N}$. Since $N\big(\sum_{i = g+1}^{2g} \Sh{O}_Ve_i\big) \subset \frac{p}{\mathrm{Hdg}}\cdot \Sh{O}^g_V$, we see that $a_i \in \frac{p}{\mathrm{Hdg}}$ for $1 \leq i \leq g$. Thus $M$ kills the kernel of the pullback of $N$ to $\Xra{r}/(p\mathrm{Hdg}^{-1})$, $i_0^*N \colon \Sh{O}^{2g}_{i_0^{-1}V} \to \Sh{O}^g_{i_0^{-1}V}$. This proves the proposition.
\end{proof}

\begin{prop}\label{P202}
The pullback of the exact sequence 
\[
0 \to \omega_{\Sh{A}} \to \mathrm{H}'_{\Sh{A}} \to \omega^{\vee}_{\Sh{A},0} \to 0
\]
along $i_0 \colon \Xra{r}/(p\mathrm{Hdg}^{-1}) \xhookrightarrow{} \Xra{r}$ admits a canonical splitting induced by the splitting of (\ref{eq:2}) which moreover commutes with the splitting induced by the decomposition $\Sh{O}_L \otimes \Sh{O}_{\Xra{r}} \simeq \prod_{\sigma \in \Sigma} \Sh{O}_{\Xra{r}}$.
\end{prop}

\begin{proof}
This is immediate from Proposition \ref{P601} and the $\Sh{O}_L$-linearity of the splitting.
\end{proof}

%% file: Hilbert3.tex
\section{\texorpdfstring{$p$}{p}-adic interpolation of modular and de Rham sheaves}\label{S3}

Henceforth fix $n$ a positive integer. Fix $I = [p^a, p^b]$ such that $k^0_{\alpha,I}$ is analytic on $1 + p^{n-1}(\Sh{O}_L \otimes \Z_p)$ and $r$ such that $H_n$ is defined on $\Xra{r}$. Depending on the two cases $I = [0,1]$ (i.e. $\alpha = p$) or $I = [p^a, p^b]$ for $a, b \in \N$, these conditions are satisfied if
\begin{enumerate}
    \item $I = [0,1]$, $r \geq 2$ if $p \neq 2$ and $2 \leq n \leq r$, or $r \geq 4$ if $p = 2$ and $4 \leq n \leq r$,
    \item $I = [p^a, p^b]$ with $a, b \in \N$, $r \geq 1$ and $r+a \geq b+2$ if $p \neq 2$ and $b+2 \leq n \leq a+r$, or $r \geq 2$ and $r+a \geq b+4$ if $p = 2$, and $b+4 \leq n \leq r+a$.
\end{enumerate}

In this section we construct overconvergent  modular and de Rham sheaves, denoted $\w{}$ and $\W{}$ on the Hilbert modular scheme $\Xra{r}\times_{\mathfrak{W}^0_{\alpha,I}} \mathfrak{W}_{\alpha,I}$ for the universal weight $k = k^{\textrm{un}} \colon (\Sh{O}_L \otimes \Z_p)^{\times} \to \Lambda^{\times}_{\alpha,I}$. The modular sheaf interpolates $\omega_{\Sh{A}}^k$  for classical weights $k$. The construction can be summarized as follows. By passing to a partial Igusa tower depending on $n$, we construct a modified modular sheaf $\Omega_{\Sh{A}}$ and a modified de Rham sheaf $\Hs{A}$ together with a modified unit root splitting. We decompose the universal character into its components induced by the splitting of $\Sh{O}_L$ in $\Sh{O}_K$. We decompose $\Omega_{\Sh{A}}$ and $\Hs{A}$ likewise. On each component we carry out the construction of interpolation for the corresponding component of the universal weight following the technique developed in \cite{andreatta2021triple}. Finally we define $\w{}$ and $\W{}$ by taking the tensor product of these individual components.

The construction of $\w{}$ appears in the joint work of Andreatta, Iovita, Pilloni and Stevens \cite{AIS}, \cite{Andreatta2016Hilbert}, \cite{andreatta2016adic}. Our construction is similar to \cite{andreatta2016adic}. However using the theory of vector bundles with marked sections we make it more explicit by actually constructing sections that generate $\w{}$ locally. This is inspired by \cite{andreatta2021triple}. At the end of the section we compare our construction of $\w{}$ with the construction in \cite{andreatta2016adic} and show why they are isomorphic. The definition of $\W{}$ is new, but as we will see it is inspired by \cite{andreatta2021triple} and the improved technique of using modified unit root subspace. The main theorem of this section is the following.

\begin{theorem*}
For $n,r,\alpha,I$ as above and $k = k^{\textrm{un}}$ the universal weight on $\mathfrak{W}_{\alpha,I}$, there are formal sheaves $\w{}$ and $\W{}$ on $\bar{\mathfrak{M}}_{r,\alpha,I} := \Xra{r} \times_{\mathfrak{W}^0_{\alpha,I}} \mathfrak{W}_{\alpha,I}$. For $\alpha = p$ and $I = [0,1]$, viewing a classical weight $\kappa$ as a point of $\Sh{W}_{p, [0,1]}(\C{}_p) = \Sh{W}_p(\C{}_p)$, the restriction of the sheaf $\w{}[1/p]_{|k \mapsto \kappa}$ on the associated analytic adic space ${\Sh{X}_{r, p}} \times_{\Sh{W}^0_p} {\Sh{W}_p}_{|\kappa}$ gives the sheaf $\omega_{\Sh{A}}^{\kappa}$ of classical Hilbert modular forms of weight $\kappa$. The sheaf $\W{}$ is equipped with a filtration by coherent $\Sh{O}_{\bar{\mathfrak{M}}_{r, \alpha,I}}$-modules $\{\Fil{i}\W{}\}_{i\geq 0}$, and $\W{}$ is the $\alpha$-adic completion of $\varinjlim_i \Fil{i}\W{}$. Moreover $\Fil{0}\W{} = \w{}$.
\end{theorem*}

\subsection{The sheaves \texorpdfstring{$\Omega_{\Sh{A}}$}{OmegaA} and \texorpdfstring{$\mathrm{H}^{\sharp}_{\Sh{A}}$}{HA}}\label{2S31}

The trivialization of $H^{\vee}_n$ on $\Sh{IG}_{n,r,I}$ induces an equality of groups $H^{\vee}_n(\Ig{n}{r}) = H^{\vee}_n(\Sh{IG}_{n,r,I}) \simeq \Sh{O}_L/p^n\Sh{O}_L$. Let $P^{\textrm{univ}}$ be the image of $1 \in \Sh{O}_L/p^n\Sh{O}_L$ in $H^{\vee}_n(\Ig{n}{r})$. We have a map of $\Sh{O}_L \otimes \mathfrak{IG}_{n,r,I}$-modules,
\begin{equation}\label{eq:3}
\begin{tikzcd}[iso/.style = {draw = none, "\xrightarrow{\sim}" description, sloped}]
& \omega_{\Sh{A}} \arrow[d] & \\
H^{\vee}_n(\Sh{IG}_{n,r,I}) \otimes_{\Z} \Sh{O}_{\mathfrak{IG}_{n,r,I}} = H^{\vee}_n(\mathfrak{IG}_{n,r,I}) \otimes_{\Z} \Sh{O}_{\mathfrak{IG}_{n,r,I}} \arrow[r, "{\dlog \otimes 1}"] & {\omega}_{H_n} \arrow[r, iso] & \omega_{\Sh{A}}/p^n\mathrm{Hdg}^{-\frac{p^n-1}{p-1}}\omega_{\Sh{A}}.
\end{tikzcd}
\end{equation}
\begin{defn}
Define the sheaf $\Omega_{\Sh{A}}$ to be the inverse image under the map $\omega_A \to {\omega}_{H_n}$ of the image of $\dlog \otimes 1$. We call this the modified modular sheaf.
\end{defn}

\begin{prop}
The sheaf $\Omega_{\Sh{A}}$ is a locally free $\Sh{O}_L \otimes \Sh{O}_{\Ig{n}{r}}$ sheaf of rank $1$. The cokernel of $\Omega_{\Sh{A}} \subset \omega_{\Sh{A}}$ is killed by $\Hdg{\frac{1}{p-1}}$. Moreover $\dlog$ induces an isomorphism
\[
\dlog \otimes 1 \colon H^{\vee}_n(\Ig{n}{r}) \otimes \Sh{O}_{\Ig{n}{r}}/p^n\mathrm{Hdg}^{-\frac{p^n}{p-1}} \xrightarrow{\sim} \Omega_A \otimes \Sh{O}_{\Ig{n}{r}}/p^n\mathrm{Hdg}^{-\frac{p^n}{p-1}}.
\]
\end{prop}

\begin{proof}
\cite[Proposition 4.1]{andreatta2016adic}.
\end{proof}

Since $\Omega_{\Sh{A}} \subset \omega_{\Sh{A}}$ is an invertible $\Sh{O}_L \otimes \Sh{O}_{\Ig{n}{r}}$- module, there exists an invertible ideal $\underline{\xi} \subset \Sh{O}_L \otimes \Sh{O}_{\Ig{n}{r}}$, such that $\Omega_{\Sh{A}} = \underline{\xi} \omega_{\Sh{A}}$. Lemma \ref{L2301} shows that $\det \underline{\xi} = \Hdg{\frac{1}{p-1}}$.

\begin{defn}
Define the sheaf $\mathrm{H}^{\sharp}_{\Sh{A}} := \underline{\xi} \mathrm{H}'_{\Sh{A}}$. We call this the modified de Rham sheaf.
\end{defn}

\begin{cor}\label{C301}
We have a short exact sequence of locally free $\Sh{O}_L \otimes \Sh{O}_{\Ig{n}{r}}$-modules
\begin{equation}\label{eq:4}
0 \to \Omega_{\Sh{A}} \to \mathrm{H}^{\sharp}_{\Sh{A}} \to \underline{\xi}  \omega^{\vee}_{\Sh{A},0} \to 0
\end{equation}
which splits upon pulling back via $i_n \colon \Ig{n}{r}/(p\Hdg{-1}) \xhookrightarrow{} \Ig{n}{r}$.
\end{cor}

\begin{proof}
This is immediate from Proposition \ref{P202}.
\end{proof}

Letting $s = \dlog(P^{\textrm{univ}}) \in \Omega_{\Sh{A}}/p^n\Hdg{\frac{-p^n}{p-1}}\Omega_{\Sh{A}}$, and taking its $\sigma$-components $s_{\sigma}$ for $\sigma \in \Sigma$, we get a vector bundle with marked sections $(\Omega_{\Sh{A}}, \{s_{\sigma}\}_{\sigma \in \Sigma})$. Associated to this pair, one can consider the geometric vector bundle with marked sections $\V{}(\Omega_{\Sh{A}}, \{s_{\sigma}\}_{\sigma \in \Sigma})$ in the sense of Definition \ref{D1303}. We will show that the points of this geometric vector bundle have a natural interpretation as $\Sh{O}_L$-linear functions on $\Omega_{\Sh{A}}$ that evaluate to $1$ on $s$. But before that we need to study functoriality of the sheaf $\Hs{A}$ with respect to the $U$ correspondence. Note that we are not interested in $\Hs{A}$ solely for its structure as a vector bundle with marked sections, but we are also interested in the splitting modulo some small power of $p$. We have already seen above that such a splitting exists modulo $p\Hdg{-1}$. In studying functoriality of $\Hs{A}$ for the $U$ correspondence, we will pin down the small power of $p$ for which the splitting is functorial too.

Let $\beta_n := p\Hdg{-\frac{p^n}{p-1}}$.

\subsubsection{Functoriality}\label{2S311}
Consider the projection $\lambda \colon \Sh{A} \to  \Sh{A}' := \Sh{A}/H_1$. Let $\lambda' \colon {\Sh{A}'} \to \Sh{A}$ be the isogeny such that $\lambda' \circ \lambda = [p]$. Then $\lambda'$ maps $H_n(\Sh{A}')$ to $H_n(\Sh{A})$ and induces an isomorphism of canonical subgroups $H_n(\Sh{A}') \simeq H_n(\Sh{A})$ on the generic fibres. The generic trivialization of $H_n^{\vee}(\Sh{A}')$ induced by this isomorphism defines a map $\tilde{F} \colon \Ig{1}{r} \to \Ig{1}{r-1}$ that sends $\Sh{A} \mapsto \Sh{A}'$ together with this trivialization of the generic fibre of the dual of the canonical subgroup. Note that a priori if $\Sh{A}$ is $\mathfrak{c}$-polarized then $\Sh{A}'$ is $p\mathfrak{c}$-polarized. But since multiplication by $p$ induces a canonical isomorphism $M(\mu_N, \mathfrak{c}) \to M(\mu_N, p\mathfrak{c})$, we indeed get a map $\tilde{F}$ as above. By abuse of notation we also denote by $\tilde{F}$ the map $\Xra{r} \to \Xra{r-1}$ induced by sending $\Sh{A} \mapsto \Sh{A}'$. The following diagram is commutative with $h$ and $h'$ being the usual projections.
\[
\begin{tikzcd}
\Ig{1}{r} \arrow[r, "\tilde{F}"] \arrow[d, "h"] & \Ig{1}{r-1} \arrow[d, "h'"] \\
\Xra{r} \arrow[r, "\tilde{F}"]                   & \Xra{r-1}                   
\end{tikzcd}
\]
The functoriality of the $\dlog$ map provides the following diagram:
\[
\begin{tikzcd}
0 \arrow[r] & \Omega_{\Sh{A}} \arrow[r] \arrow[d, "\simeq"] & \omega_{\Sh{A}} \arrow[d, "(\lambda')^*"] \\
0 \arrow[r] & \Omega_{\Sh{A}'} \arrow[r]                    & \omega_{\Sh{A}'}                         
\end{tikzcd}
\]
The map $(\lambda')^*$ is the adjoint of a lift of the Hasse-Witt map $HW \colon \omega^{\vee}_{\Sh{A}'/(p\Hdg{-1})} \to \omega^{\vee}_{\Sh{A}/(p\Hdg{-1})}$ modulo $p\Hdg{-1}$ because $\lambda'$ is a lift of the Verschiebung. As a map between invertible $\Sh{O}_L \otimes \Sh{O}_{\Ig{1}{r}}$-modules, $(\lambda')^*$ corresponds to multiplication by an invertible $\Sh{O}_L \otimes \Sh{O}_{\Ig{1}{r}}$ ideal ${\widetilde{HW}}$. Since $\Omega_{\Sh{A}} = \underline{\xi}\omega_{\Sh{A}}$, we have the following relation between $\Sh{O}_L \otimes \Sh{O}_{\Ig{1}{r}}$ ideals.

\begin{lemma}\label{L2301}
${\tilde{F}}^*\underline{\xi} = \underline{\xi} {\widetilde{HW}}$.
\end{lemma}

\begin{proof}
Follows from the discussion above.
\end{proof}

\begin{defn}
Define the $\sigma$-components $\widetilde{HW}(\sigma)$ of $\widetilde{HW}$ as the partial Hasse ideals.
\end{defn}

\begin{cor}
$\prod_{\sigma} \widetilde{HW}(\sigma) = \Hdg{}$.
\end{cor}

\begin{proof}
Immediate as the determinant of $\widetilde{HW}$ is $\Hdg{}$.
\end{proof}

We will prove a result relating the $\sigma$-components of $\underline{\xi}$ and $\widetilde{HW}$. For that we need to choose a numbering of $\Sigma$. Recall $p$ splits as $p = \mathfrak{P}_1 \cdots \mathfrak{P}_h$ in $\Sh{O}_L$ with their inertia degree $f(\mathfrak{P}_i | p) = f_i$. For each $i$, let $\Sigma_{\mathfrak{P}_i}$ be the subset of embeddings that induce the $\mathfrak{P}_i$-adic valuation on $L$. Choose a bijection $\Xi_i \colon \{i_1, \dots, i_{f_i}\} \simeq \Sigma_{\mathfrak{P}_i}$ such that $\Xi_i(i_{n+1}) = \Xi_i(i_n) \circ \varphi_{\mathfrak{P}_i}$, where $\varphi_{\mathfrak{P}_i}$ is the lift of the Frobenius in $D(\mathfrak{P}_i)$. 

\begin{cor}\label{C233}
$(\tilde{F}^*\omega_{\Sh{A}})(i_j) = \tilde{F}^*(\omega_{\Sh{A}}(i_{j-1}))$. In particular, $\tilde{F}^*(\underline{\xi}(i_{j-1})) = \underline{\xi}(i_j)\widetilde{HW}(i_j)$.
\end{cor}

\begin{proof}
Recall that for any $\Sh{O}_L \otimes \Sh{O}_{\Ig{1}{r}}$-module $\Sh{F}$, $\Sh{F}(i_j)$ is the component on which $\Sh{O}_L$ acts via $\Xi_i(i_j)$. The claim follows immediately by noting that modulo $p\Hdg{-1}$, $\tilde{F}$ induces a morphism such that the following diagram commutes with $\varphi$ being the Frobenius.
\[
\begin{tikzcd}[column sep=large]
\Ig{1}{r}/(p\Hdg{-1}) \arrow[d] \arrow[r, "\tilde{F}"] & \Ig{1}{r-1}/(p\Hdg{-1}) \arrow[d] \\
\mathfrak{X}_{\alpha,I}/(p) \arrow[r, "\varphi"]                                    & \mathfrak{X}_{\alpha,I}/(p)                 
\end{tikzcd}
\]
\end{proof}

\begin{lemma}
$\Xra{r}$ and $\Ig{1}{r}$ are normal schemes.
\end{lemma}

\begin{proof}
For the case of $\Xra{r}$ see \cite[Corollary 3.8]{andreatta2016adic}. The map $\Sh{IG}_{1,r,I} \to \xra{r}$ is finite \'{e}tale. Hence $\Sh{IG}_{1,r,I}$ is normal. Thus $\Ig{1}{r}$ is normal being the normalization of $\Xra{r}$ in $\Sh{IG}_{1,r,I}$.
\end{proof}

\begin{lemma}\label{L403}
For all $i_j$, $\tilde{F}^*(\underline{\xi}(i_j))$ is a $p$-th power at all height 1 localizations.
\end{lemma}

\begin{proof}
Since this is an equality of ideals in a normal scheme it is enough to check the statement locally at height 1 primes. So choose an affine open $U$ in $\mathfrak{X}_{\alpha,I}$ such that the pullback of $\omega_{\Sh{A}}$ to $U/(p)$ is trivial as an $\Sh{O}_L \otimes \Sh{O}_{\mathfrak{X}_{\alpha,I}}/(p)$-module. Let $\Spf{R_{r-1}}$ and $\Spf{R_r}$ be its inverse image in $\Xra{r-1}$ and $\Xra{r}$ respectively. Let $S_{r-1}$ and $S_r$ be their respective inverse images in $\Ig{1}{r-1}$ and $\Ig{1}{r}$. So $\tilde{F}$ induces a commutative diagram as follows.
\[
\begin{tikzcd}
S_{r-1} \arrow[r, "\tilde{F}^*"]                & S_r                \\
R_{r-1} \arrow[r, "\tilde{F}^*"] \arrow[u, "h'"] & R_r \arrow[u, "h"]
\end{tikzcd}
\]
Pick a height 1 prime $\mathfrak{p} \in \Spf{S_r}$ that contains a local generator of $\underline{\xi}(i_j)$. Since $\prod_{i,j} \underline{\xi}(i_j)^{p-1} = \Hdg{}$ \cite[Proposition A.3]{Andreatta2018leHS}, $\mathfrak{p} \in V(\Hdg{})$. Let $\mathfrak{q} = h^{-1}\mathfrak{p}$. Then $\mathfrak{q}$ is a height 1 prime containing $\Hdg{}$. In particular $\mathfrak{q}{R_r}_q$ is generated by $\Hdg{}$ as $\Hdg{}$ has simple zeroes along degree 1 divisors in $\mathfrak{X}_{\alpha,I}/(p)$ (Theorem \ref{T2101}). Let $\mathfrak{p}' = (\tilde{F}^*)^{-1}\mathfrak{p}$ and $\mathfrak{q}' = (\tilde{F}^*)^{-1}\mathfrak{q}$. Localizing at the primes gives a diagram as follows.
\[
\begin{tikzcd}
{S_{r-1}}_{\mathfrak{p}'} \arrow[r, "\tilde{F}^*"]                 & {S_r}_{\mathfrak{p}}                \\
{R_{r-1}}_{\mathfrak{q}'} \arrow[r, "\tilde{F}^*"] \arrow[u, "h'"] & {R_r}_{\mathfrak{q}} \arrow[u, "h"]
\end{tikzcd}
\]
All the rings are DVR. The bottom arrow has ramification index $p$. $h$ and $h'$ are tamely ramified. This forces the upper arrow to be ramified of index $p$. This proves the lemma.
\end{proof}

\begin{rem}
It seems that in fact $\tilde{F}^*(\underline{\xi}(i_j)) = \underline{\xi}(i_j)^p$. This would imply that $\underline{\xi}(i_{j-1})^p = \underline{\xi}(i_j)\widetilde{HW}(i_j)$ which reflects the fact that the partial Hasse invariant of degree $i_j$ is of weight $(p, -1)$ concentrated at degree $(i_{j-1}, i_j)$ \cite[Theorem 2.1]{gorenhasse}. Moreover this shows a posteriori that $\prod_{i,j}\underline{\xi}(i_j)^{p-1} = \Hdg{}$. But as of now we are not able to prove this.
\end{rem}

\begin{prop}\label{P232}
There exists an $r$ large enough such that the map $(\lambda')^* \colon \mathrm{H}_{\Sh{A}} \to \mathrm{H}_{\Sh{A}'}$ restricts to a well-defined map $(\lambda')^* \colon \Hs{A} \to \Hs{A'}$ that sends marked sections to marked sections. Moreover, let $\Sh{Q} \subset \Hs{A}/p\Hdg{}(\Sh{A})^{-(p+1)}$ be the kernel of the marked splitting, and let $\Sh{Q}' \subset \Hs{A'}/p\Hdg{}(\Sh{A})^{-(p+1)}$ be the same for $\Sh{A}'$. Then $(\lambda')^*$ sends $\Sh{Q}$ to $\Sh{Q}'$.
\end{prop}

\begin{proof}
Since $(\lambda')^*$ maps $\Omega_{\Sh{A}}$ isomorphically onto $\Omega_{\Sh{A}'}$ sending the marked section to the marked section, it is enough to show that the induced map $\Hs{A}/\Omega_{\Sh{A}} \to \mathrm{H}_{\Sh{A}'}/\Omega_{\Sh{A}'}$ factors through the inclusion $\Hs{A'}/\Omega_{\Sh{A}'} \xhookrightarrow{} \mathrm{H}_{\Sh{A}'}/\Omega_{\Sh{A}'}$. Choosing suitable local generators $\xi$ of $\underline{\xi}({\Sh{A}})$ and $\widetilde{HW}$ of $\widetilde{HW}({\Sh{A}})$ respectively as $\Sh{O}_L \otimes \Sh{O}_{\Ig{n}{r}}$-modules, we have a diagram as follows.
\[\begin{tikzcd}
	& {\Hs{A'}/\Omega_{\Sh{A}'}} & \tilde{F}^*(\underline{\xi}({\Sh{A}})\widetilde{HW}({\Sh{A}}))\cdot\omega^{\vee}_{\Sh{A}'} \\
	{\Hs{A}/\Omega_{\Sh{A}}} & {\mathrm{H}_{\Sh{A}'}/\Omega_{\Sh{A}'}} & {\mathrm{H}_{\Sh{A}'}/\omega_{\Sh{A}'} \simeq \omega^{\vee}_{\Sh{A}'}} \\
	& \underline{\xi}({\Sh{A}})\widetilde{HW}({\Sh{A}})\cdot\omega^{\vee}_{\Sh{A}}
	\arrow["(\lambda')^*", from=2-1, to=2-2]
	\arrow["\pi", from=2-2, to=2-3]
	\arrow[hook, from=1-2, to=2-2]
	\arrow["\simeq"', from=2-1, to=3-2]
	\arrow["\simeq", "i"', from=1-2, to=1-3]
	\arrow["{p\widetilde{HW}^{-1}\xi\widetilde{HW}}"', from=3-2, to=2-3]
	\arrow["\tilde{F}^*(\xi\widetilde{HW})", from=1-3, to=2-3]
\end{tikzcd}\]
Here the bottom right diagonal arrow is the map induced on the Lie algebra by $\lambda \colon \Sh{A} \to \Sh{A}'$. Choosing basis of $\omega^{\vee}_{\Sh{A}}$ and of $\omega^{\vee}_{\Sh{A}'}$, this $\Sh{O}_L \otimes \Sh{O}_{\Ig{n}{r}}$-linear map is multiplication by $p\widetilde{HW}^{-1}$ upto a unit. Hence we get a description of the arrow as in the diagram.

Let $\Hdg{} = \Hdg{}({\Sh{A}})$ in this proof. Since multiplication by $p\xi$ is injective, we see that the image of $\Hs{A}/\Omega_{\Sh{A}}$ under $(\lambda')^*$ does not intersect $\omega_{\Sh{A}'}/\Omega_{\Sh{A}'}$. We first show that $\pi \circ (\lambda')^*$ factors through the submodule $\tilde{F}^*(\underline{\xi}({\Sh{A}})\widetilde{HW}({\Sh{A}}))\omega^{\vee}_{\Sh{A}'} \subset \omega^{\vee}_{\Sh{A}'}$. For this it is enough to show that $p\underline{\xi}({\Sh{A}}) \subset \tilde{F}^*(\underline{\xi}({\Sh{A}})\widetilde{HW}({\Sh{A}}))$. Using Lemma \ref{L2301} this reduces to showing $p/(\widetilde{HW}\tilde{F}^*(\widetilde{HW})) \subset \Sh{O}_L \otimes \Sh{O}_{\Ig{n}{r}}$ which can be ensured by choosing large enough $r$ since $\det{\widetilde{HW}} = \Hdg{}$. This proves that the map $(\lambda')^* - i^{-1}\circ \pi \circ (\lambda')^*$ factors through $\omega_{\Sh{A}'}/\Omega_{\Sh{A}'}$ which is a torsion module killed by $\tilde{F}^*(\underline{\xi}({\Sh{A}})) = \underline{\xi}({\Sh{A}})\widetilde{HW}({\Sh{A}})$ (Lemma \ref{L2301}). Now since $\Hs{A} = \Omega_{\Sh{A}} + \underline{\xi}({\Sh{A}})\widetilde{HW}({\Sh{A}})\cdot\mathrm{H}_{\Sh{A}}$, the difference $(\lambda')^* - i^{-1}\circ \pi \circ (\lambda')^* = 0$. This proves the first claim.

The second claim follows from a local computation. Choose an open affine $\Spf{R} = U \subset \X$ such that $\mathrm{H}_{\Sh{A}}$ admits an $\Sh{O}_L \otimes R$ basis $\{e, f\}$ over $U$ with $e$ a basis of $\omega_{\Sh{A}}$. Let $\{\bar{e}, \bar{f}\}$ be their image over $\Spf{R}/p$. The unit root subspace is generated by a vector $\bar{v} = \bar{C}\bar{e} + HW\bar{f}$ for some $\bar{C} \in (\Sh{O}_L \otimes R)/p$. Let $\varphi \colon R/p \to R/p$ be the Frobenius. With respect to the basis $\{\bar{e},\bar{f}\}$ of ${\mathrm{H}}_{\bar{\Sh{A}}}$ and $\{\bar{e}^{(p)} := \varphi^*(\bar{e}), \bar{f}^{(p)} := \varphi^*(\bar{f})\}$ of ${\mathrm{H}}_{\bar{\Sh{A}}^{(p)}}$, the matrix of Verschiebung $V \colon {\mathrm{H}}_{\bar{\Sh{A}}} \to {\mathrm{H}}_{\bar{\Sh{A}}^{(p)}}$ can be written as
\[
V = 
\begin{pmatrix}
HW & \bar{B} \\
0 & 0
\end{pmatrix}.
\]
Since Verschiebung kills the unit root subspace, $V(\bar{C}\bar{e} + HW\bar{f}) = (\bar{C}HW + \bar{B}HW)\bar{e} = 0$. This shows that $\bar{B} = -\bar{C}$. Let $\Spf{R}_{n, r} \subset \Ig{n}{r}$ (resp. $\Spf{R}_{n, r-1} \subset \Ig{n}{r-1}$) be the inverse image of $U$ in $\Ig{n}{r}$ (resp. $\Ig{n}{r-1}$). Then as discussed in \S\ref{S2}, $\mathrm{H}'_{\Sh{A}}$ over $\Ig{n}{r}$ is generated by the pullback of $e$ and a lift $v = Ce + \widetilde{HW}f$ of $\bar{v}$. For notational simplicity we will write these sections as $\{e,v\}$ still. Similarly, $\mathrm{H}'_{\Sh{A}'}$ is generated by pulling back via $\tilde{F} \colon \Spf{R}_{n, r} \to \Spf{R}_{n,r-1}$ the pullbacks of $e$ and $v$ to $\Spf{R}_{n,r-1}$. We will write these as $\{\tilde{F}^*e, \tilde{F}^*v\}$. Then with respect to $\{e, v\}$ and $\{\tilde{F}^*e, \tilde{F}^*v\}$ the matrix of $(\lambda')^*$ can be described as
\begin{align*}
(\lambda^{\vee})^* &= 
{\begin{pmatrix}
1 & -\tilde{F}^*(C\widetilde{HW}^{-1}) \\
0 & \tilde{F}^*(\widetilde{HW}^{-1})
\end{pmatrix}}\cdot
\begin{pmatrix}
\widetilde{HW} & B \\
0 & p\widetilde{HW}^{-1}
\end{pmatrix}\cdot
\begin{pmatrix}
1 & C \\
0 & \widetilde{HW}
\end{pmatrix} \\
&= 
\begin{pmatrix}
\widetilde{HW} & C\widetilde{HW}{} + B\widetilde{HW} - p\tilde{F}^*(C\widetilde{HW}^{-1}) \\
0 & p\tilde{F}^*(\widetilde{HW}^{-1})
\end{pmatrix}.
\end{align*}
Here $B$ is a lift of $\bar{B}$ modulo $p\Hdg{-1}$. Therefore with respect to the basis $\{\xi e, \xi v\}$ and $\{\tilde{F}^*(\xi e), \tilde{F}^*(\xi v)\}$, the matrix of $(\lambda')^* \colon \Hs{A} \to \Hs{A'}$ is written as 
\[
(\lambda')^* =
\begin{pmatrix}
1 & C + B - p\tilde{F}^*(C\widetilde{HW}^{-1})\widetilde{HW}^{-1} \\
0 & p\widetilde{HW}^{-1}\tilde{F}^*(\widetilde{HW}^{-1})
\end{pmatrix}.
\]
Since $\det{\widetilde{HW}} = \Hdg{}$, this proves the second claim of the lemma.
\end{proof}

\begin{cor}\label{C314}
    The kernel of the marked splitting $\Sh{Q}$ equals the kernel of Verschiebung modulo $p\Hdg{-(p+1)}$.
\end{cor}

\begin{proof}
    Follows from the proof of Proposition \ref{P232}.
\end{proof}

\begin{defn}\label{D315}
    For a prime $\mathfrak{P}|p$, define the map $\tilde{F}^*_{\mathfrak{P}} \colon \Sh{O}_L\otimes \Sh{O}_{\Xra{r-1}} \to \Sh{O}_L \otimes \Sh{O}_{\Xra{r}}$ (and similarly for $\Ig{n}{r-1})$ that coincides with $\tilde{F}^*$ on the $\mathfrak{P}$-component and identity on the rest.
\end{defn}

\begin{rem}
    For each $\mathfrak{P} | p$, one can consider the partial Hasse invariant $\Hdg{}_{\mathfrak{P}} = \prod_{\sigma | \mathfrak{P}}\widetilde{HW}(\sigma)$. Then for a multi-index $\mathbf{r} = (r_1, \dots, r_h)$, one can consider the formal model $\Xra{\mathbf{r}}$ whose generic fibre is defined by the relations $|p|\leq \prod_{i=1}^h|\Hdg{}_{\mathfrak{P}_i}|^{p^{r_i}}$, constructed in the same fashion as $\Xra{r}$ using formal admissible blow-ups. In particular, for $r = \max\{r_1, \dots, r_h\}$, and $s = \min\{r_1, \dots, r_h\}$, we have $\Xra{r} \subset \Xra{\mathbf{r}} \subset \Xra{s}$. The map on the generic fibre of the overconvergent moduli schemes defined by sending $A \mapsto A/H_1\cap A[\mathfrak{P}_i]$, induces a map $\tilde{F}_{\mathfrak{P}_i}\colon \Xra{{r}}^{\mathfrak{c}} \to \Xra{\mathbf{r}'}^{\mathfrak{cP}_i}$, where $\mathbf{r}'$ is the multi-index with $r-1$ in the $i$-{th} place and $r$ elsewhere. It is easy to see that the pullback of the modular sheaf over $\Xra{\mathbf{r}'}^{\mathfrak{cP}_i}$ along $\tilde{F}_{\mathfrak{P}_i}$ is isomorphic to the scalar extension of the modular sheaf over $\Xra{r-1}$ along $\tilde{F}_{\mathfrak{P}_i}^*$ as defined in Definition \ref{D315}.
\end{rem}

Let $\Spf{R} \to \Ig{n}{r}$ be an $\alpha$-adically complete, admissible, normal $\Lambda^0_{\alpha,I}$-algebra. In particular, it corresponds to an abelian scheme $A'/R$ with a generic trivialization $\Sh{O}_L/p^n\Sh{O}_L \xrightarrow{\sim} H_n^{\vee}(A')$. Suppose that $\lambda_{\mathfrak{P}_i} \colon A' \to A$ is an isogeny, such that $\ker \lambda_{\mathfrak{P}_i}$ is generically \'{e}tale locally isomorphic to $\Sh{O}_L/\mathfrak{P}_i$, and such that $\ker \lambda_{\mathfrak{P}_i} \cap H_1(A') = 0$. We note that $A/H_1[\mathfrak{P}_i] \simeq A' \otimes \mathfrak{P}_i^{-1}$. Therefore, $\Hdg{}(A') = \Hdg{}(A)\Hdg{}_{\mathfrak{P}_i}(A)^{p-1}$.

\begin{prop}\label{P317}
    There exists an $r$ large enough such that the map $\lambda_{\mathfrak{P}_i}^* \colon \mathrm{H}_{A} \to \mathrm{H}_{A'}$ restricts to a well-defined map $\lambda_{\mathfrak{P}_i}^* \colon \mathrm{H}^{\sharp}_{A} \to \mathrm{H}^{\sharp}_{A'}$ that sends marked sections to marked sections. It also maps the kernel $\Sh{Q} \subset \mathrm{H}^{\sharp}_{A}/p\Hdg{}(A)^{-(p+1)}$ to the kernel $\Sh{Q}' \subset \mathrm{H}^{\sharp}_{A'}/p\Hdg{}(A)^{-(p+1)}$.
\end{prop}

\begin{proof}
    We first note that $\lambda^*_{\mathfrak{P}_i}\colon \Omega_{A}\xrightarrow{\sim} \Omega_{A'}$ is an isomorphism, since $\lambda_{\mathfrak{P}_i}$ induces an isomorphism $H_n^{\vee}(A) \xrightarrow{\sim} H_n^{\vee}(A')$. Then as in Proposition \ref{P232}, we have a diagram as follows.
    \[\begin{tikzcd}
	& {\mathrm{H}^{\sharp}_{A'}/\Omega_{{A}'}} & \tilde{F}^*_{\mathfrak{P}_i}(\underline{\xi}({{A}})\widetilde{HW}({{A}}))\cdot\omega^{\vee}_{{A}'} \\
	{\mathrm{H}^{\sharp}_{A}/\Omega_{{A}}} & {\mathrm{H}_{{A}'}/\Omega_{{A}'}} & {\mathrm{H}_{{A}'}/\omega_{{A}'} \simeq \omega^{\vee}_{{A}'}} \\
	& \underline{\xi}({{A}})\widetilde{HW}({{A}})\cdot\omega^{\vee}_{{A}}
	\arrow["{\lambda^*_{\mathfrak{P}_i}}", from=2-1, to=2-2]
	\arrow["\pi", from=2-2, to=2-3]
	\arrow[hook, from=1-2, to=2-2]
	\arrow["\simeq"', from=2-1, to=3-2]
	\arrow["\simeq", "i"', from=1-2, to=1-3]
	\arrow["{\mathfrak{P}_i\widetilde{HW}_{\mathfrak{P}_i}^{-1}\xi\widetilde{HW}}"', from=3-2, to=2-3]
	\arrow["\tilde{F}^*_{\mathfrak{P}_i}(\xi\widetilde{HW})", from=1-3, to=2-3]
    \end{tikzcd}\]
    The exact same reasoning as in Proposition \ref{P232} proves the first claim. By Corollary \ref{C314}, $\Sh{Q}$ is the kernel of Verschiebung modulo $p\Hdg{-(p+1)}$. The second claim then follows from the observations that $\lambda_{\mathfrak{P}_i}$ is a lift of the partial Verschiebung at $\mathfrak{P}_i$ of $A \otimes \mathfrak{P}_i$, and Verschiebung commutes with the partial Verschiebung.
\end{proof}

\begin{defn}
Let $\Hs{A}, \Omega_{\Sh{A}}, s$ be the modified de Rham sheaf, the modified modular sheaf and the marked section respectively. Let $\Sh{Q} \subset \Hs{A}/p\Hdg{{-p^2}}$ be the kernel of the splitting $\psi \colon \Hs{A}/p\Hdg{{-p^2}} \to \Omega_{\Sh{A}}/p\Hdg{{-p^2}}$. We call this marked splitting the modified unit root splitting and $\Sh{Q}$ the modified unit root subspace.
\end{defn}

We remark that by Proposition \ref{P232} the modified unit root subspace is functorial for $\lambda' \colon \Sh{A}' \to \Sh{A}$.

Recall the notation $\beta_n = p^n\Hdg{\frac{-p^n}{p-1}}$. Let $\eta := p\Hdg{{-p^2}}$.

\subsection{Vector bundles with marked sections and marked splitting}\label{1S31}
In this section we recall the general formalism of vector bundles with marked sections following \cite[\S2]{andreatta2021triple}. We then define a subfunctor of a vector bundle with a marked section, which we call ``vector bundle with marked sections and a marked splitting" and show representability of this functor.

Let $S$ be a formal scheme with an ideal of definition $\sh{I}$ which is invertible. Let $i \colon S/\sh{I} \xhookrightarrow{} S$ be the closed subscheme defined by the ideal $\sh{I}$. Let $\catname{FSch}^{\sh{I}}_S$ be the category of formal schemes $f \colon T \to S$ such that $f^{-1}\sh{I}\cdot\Sh{O}_T$ is an invertible ideal in $T$.

\begin{defn}
A formal vector bundle of rank $n$ is a formal vector group scheme $X \to S$ which is isomorphic to $\Ga^n$ locally over $S$.
\end{defn}

\begin{defn}\label{D1302}
Let $\Sh{E}$ be a locally free sheaf of rank $n$ on $S$. The formal vector bundle $\mathbb{V}(\Sh{E})$ of rank $n$ is defined as the functor on  $\catname{FSch}^{\sh{I}}_S$
\[
\mathbb{V}(\Sh{E})(t \colon T \to S) := \Sh{E}^{\vee}(T) = \Hom_{\Sh{O}_T}(t^*\Sh{E}, \Sh{O}_T).
\]
\end{defn}

\begin{lemma}
$\mathbb{V}(\Sh{E})$ is representable by the formal scheme $\Spf{\widehat{\Sym^{\bullet}\Sh{E}}} \to S$, where the completion is with respect to the $\sh{I}$-adic topology. This formal scheme is a formal vector bundle of rank $n$. Moreover, the contravariant functor $\mathbb{V}$ defines an equivalence between locally free sheaves on $S$ of constant rank and formal vector bundles of finite rank over $S$, and the equivalence preserves the notion of rank. 
\end{lemma}

\begin{proof}
\cite[Lemma 2.2]{andreatta2021triple}.
\end{proof}

Let $\Sh{E}$ be a locally free sheaf of rank $n$ on $S$ such that $\bar{\Sh{E}} := i^*\Sh{E}$ has sections $s_1, \dots, s_m \in \Gamma(S/\sh{I}, \bar{\Sh{E}})$ for $m \leq n$, satisfying the following two properties:
\begin{enumerate}
    \item The subsheaf $\bar{\Sh{E}'} \subset \bar{\Sh{E}}$ generated by $s_1, \dots, s_m$ is locally free,
    \item $\bar{\Sh{E}}/\bar{\Sh{E}'}$ is locally free.
\end{enumerate}

\begin{defn}\label{D1303}
Let $\Sh{E}$ be a locally free sheaf with marked sections $s_1, \dots, s_m$ as above. The formal vector bundle with marked sections $\V{}(\Sh{E}, s_1, \dots, s_m)$ is defined as the functor on $\catname{FSch}^{\sh{I}}_S$
\[
\V{}(\Sh{E}, s_1, \dots, s_m)(t \colon T \to S) := \{f \in \mathbb{V}(\Sh{E})(T) \, | \, (f \textrm{ mod }t^*\sh{I})(t^*s_i) = 1 \, \forall i\}
\]
\end{defn}

\begin{lemma}\label{L1302}
$\V{}(\Sh{E}, s_1, \dots, s_m)$ is represented by an open formal subscheme of an admissible formal blow-up of $\mathbb{V}(\Sh{E})$.
\end{lemma}

\begin{proof}
This is \cite[Lemma 2.4]{andreatta2021triple}. We recall the construction. $s_1, \dots, s_m$ define an ideal $\sh{J} \subset {\Sym^{\bullet}\Sh{E}}/\sh{I}$ given by $\sh{J} := (s_1 - 1, \dots, s_m - 1)$. Let $\tilde{\sh{J}}$ be the inverse image of $\sh{J}$ in $\Sh{O}_{\widehat{\Sym^{\bullet}\Sh{E}}}$. Let $\mathbb{B}$ be the blow-up of $\mathbb{V}(\Sh{E})$ along $\tilde{\sh{J}}$. Then take the $\sh{I}$-adic completion of the open in $\mathbb{B}$ where the inverse image ideal of $\tilde{\sh{J}}$ coincides with the inverse image ideal of $\sh{I}$. Then as shown in loc. cit. this formal scheme represents $\mathbb{V}_0(\Sh{E}, s_1, \dots, s_m)$.

In local coordinates $\V{}(\Sh{E}, s_1, \dots, s_m)$ can be described as follows. Let $\Spf{R} \subset S$ be an open which trivializes $\Sh{E}$ and $\sh{I}$, and let $X_i$ be a lift of $s_i$ for all $i$ and extend it to a basis of $\Sh{E}$. Then $\mathbb{V}(\Sh{E})_{|\Spf{R}} = \Spf R\langle X_1, \dots, X_n\rangle$. Then $\V{}(\Sh{E}, s_1, \dots, s_m)_{|\Spf{R}}$ is given by 
\[\Spf{R}\langle
\{X_i\}_{i=1}^n\rangle\langle \{ \frac{X_i - 1}{\alpha}\}_{i=1}^m \rangle \simeq \Spf{R} \langle \{Z_i\}_{i=1}^m , X_{m+1}, \dots, X_n \rangle
\]
where $\alpha$ is a generator for $\sh{I}$, and the projection $\V{}(\Sh{E}, s_1, \dots, s_m) \to \mathbb{V}(\Sh{E})$ corresponds to the ring map that sends $X_i \mapsto 1 + \alpha Z_i$ for $1 \leq i \leq m$ and $X_i \mapsto X_i$ otherwise.
\end{proof}

Let $\Sh{E}$ be a locally free sheaf of rank $n$ with marked sections $s_1, \dots, s_m$ as above. Suppose the short exact sequence of locally free sheaves on $S/\sh{I}$
\[
0 \to \bar{\Sh{E}'} \to \bar{\Sh{E}} \to \bar{\Sh{E}}/\bar{\Sh{E}'} \to 0
\]
admits a splitting $\psi \colon \bar{\Sh{E}} \to \bar{\Sh{E}'}$. Here $\bar{\Sh{E}'}$ is the subsheaf generated by the $s_i$'s as above. Let $\Sh{Q} := \ker{\psi}$.

\begin{defn}\label{D1304}
Given a locally free sheaf $\Sh{E}$ on $S$ of rank $n$, marked sections $s_1, \dots, s_m$, and a subsheaf $\Sh{Q} \subset \bar{\Sh{E}}$ corresponding to a splitting as above, we define the vector bundle with marked sections and marked splitting as the functor on $\catname{FSch}^{\sh{I}}_S$
\[
\V{}(\Sh{E}, s_1, \dots, s_m, \Sh{Q})(t \colon T \to S) := \{ f \in \V{}(\Sh{E}, s_1, \dots, s_m)(T) \, | \, (f \textrm{ mod }t^*\sh{I})(t^*\Sh{Q}) = 0\}
\]
\end{defn}

\begin{lemma}\label{L1303}
$\V{}(\Sh{E}, s_1, \dots, s_m , \Sh{Q})$ is represented by an open formal subscheme of an admissible formal blow-up of $\V{}(\Sh{E}, s_1, \dots, s_m)$. 
\end{lemma}

\begin{proof}
The subsheaf $\Sh{Q} \subset \bar{\Sh{E}}$ defines an ideal in $\Sym^{\bullet}\bar{\Sh{E}}$, viz. $\Sh{Q}\Sym^{\bullet}\bar{\Sh{E}}$. Let $\tilde{\Sh{Q}'}$ be the inverse image of this ideal in $\widehat{\Sym^{\bullet}\Sh{E}}$. If $f \colon \V{}(\Sh{E}, s_1, \dots, s_m) \to \mathbb{V}(\Sh{E})$ is the projection, let $\tilde{\Sh{Q}} := f^{-1}\tilde{\Sh{Q}'}\cdot \Sh{O}_{\V{}(\Sh{E}, s_1, \dots, s_m)}$ be the inverse image ideal in $\V{}(\Sh{E}, s_1, \dots, s_m)$. Consider blow-up of $\V{}(\Sh{E}, s_1, \dots, s_m)$ along $\tilde{\Sh{Q}}$. Let $X$ be the $\sh{I}$-adic completion of the open in this blow-up where the inverse image ideal of $\tilde{\Sh{Q}}$ coincides with the inverse image of $\sh{I}$. We claim that $X$ represents $\V{}(\Sh{E}, s_1, \dots, s_m, \Sh{Q})$.

To prove this, suppose $\tilde{t} \colon T \to X$ be a lift of $t \colon T \to S$. Then certainly $\tilde{t}$ defines a point of $\mathbb{V}(\Sh{E})$ and hence corresponds to an element $f \in \Hom_{\Sh{O}_T}(t^*\Sh{E}, \Sh{O}_T)$. Moreover this satisfies $(f \textrm{ mod }t^*\sh{I})(t^*s_i) = 1$. Since over $X$, $\sh{I}$ coincides with $\tilde{\Sh{Q}}$, $(f \textrm{ mod }t^*\sh{I})$ kills $t^*\Sh{Q}$. Conversely, to prove that for any element $f \in \Hom_{\Sh{O}_T}(t^*\Sh{E}, \Sh{O}_T)$ seen as a point of $\mathbb{V}(\Sh{E})$, that sends $s_i \mapsto 1$ and kills $\Sh{Q}$ modulo $t^*\sh{I}$, there exists a unique lift $\tilde{t} \colon T \to X$, it will be enough to prove this when $T = \Spf{R'}$ and $t \colon T \to S$ factors through an open $\Spf{R} \subset S$ where $\bar{\Sh{E}'}, \bar{\Sh{E}}$ and $\bar{\Sh{E}}/\bar{\Sh{E}'}$ are locally free. In that case picking lifts $X_i$ of $s_i$ as in the proof of Lemma \ref{L1302} and lifts $Y_{m+1}, \dots, Y_n$ of a basis $y_{m+1}, \dots, y_n$ of $\Sh{Q}$, we can write 
\begin{align*}
X_{|\Spf{R}} &\simeq \Spf{R}\langle \{X_i\}_{i=1}^m\{Y_j\}_{j=m+1}^n\rangle\langle \{\frac{X_i - 1}{\alpha}\}_{i=1}^m\{\frac{Y_j}{\alpha}\}_{j=m+1}^n\rangle \\ &\simeq R\langle \{Z_i\}_{i=1}^m, \{W_{j}\}_{j=m+1}^n\rangle
\end{align*}
where the projection $X \to \V{}(\Sh{E}, s_1, \dots, s_m)$ corresponds to $Z_i \mapsto Z_i$ and $Y_j \mapsto \alpha W_j$. Now $t^*\Sh{E} = \oplus_{i=1}^m R'(t^*X_i) \bigoplus \oplus_{j=m+1}^n R'(t^*Y_j)$. An $R'$-linear map $f \colon t^*\Sh{E} \to R'$ that satisfies $f(X_i) = 1 \textrm{ mod }(\alpha)$ and $f(Y_i) = 0 \textrm{ mod }(\alpha)$ can be written uniquely as \[f = \sum_{i=1}^m (1 + \alpha r_i)(t^*X_i)^{\vee} + \sum_{j=m+1}^n \alpha q_j (t^*Y_j)^{\vee}\] for $r_i, q_j \in R'$. Then we can define a map $R'\langle \{Z_i\}\{W_j\}\rangle$ that sends $Z_i \mapsto r_i$ and $W_j \mapsto q_j$ for all $i, j$. This determines a point $\tilde{t} \in X(T)$. 
\end{proof}

\begin{lemma}\label{L1304}
The functor $\V{}(\Sh{E}, s_1, \dots, s_m, \Sh{Q})$ is functorial in tuples $(\Sh{E}, s_1, \dots, s_m, \Sh{Q})$, i.e. given a map of locally free sheaves of equal rank $\rho \colon \Sh{E} \to \Sh{E}'$, where $\Sh{E}$ (resp. $\Sh{E}'$) is equipped with marked sections $s_1, \dots, s_m$ (resp. $s'_1, \dots, s'_m$) and a subsheaf $\Sh{Q}$ (resp. $\Sh{Q}'$), corresponding to a marked splitting, if $\bar{\rho}(s_i) = s'_i$ and $\bar{\rho}(\Sh{Q}) \subset \Sh{Q}'$, then there are natural transformations $\V{}(\Sh{E}', s'_1, \dots, s'_m, \Sh{Q}') \to \V{}(\Sh{E}, s_1, \dots, s_m, \Sh{Q})$ and $\V{}(\Sh{E}', s'_1, \dots, s'_m) \to \V{}(\Sh{E}, s_1, \dots, s_m)$ such that the following diagram commutes.
\[\begin{tikzcd}
	{\V{}(\Sh{E}', s'_1, \dots, s'_m, \Sh{Q}')} & {\V{}(\Sh{E}, s_1, \dots, s_m, \Sh{Q})} \\
	{\V{}(\Sh{E}', s'_1, \dots, s'_m)} & {\V{}(\Sh{E}, s_1, \dots, s_m)}
	\arrow[from=1-1, to=1-2]
	\arrow[from=1-1, to=2-1]
	\arrow[from=2-1, to=2-2]
	\arrow[from=1-2, to=2-2]
\end{tikzcd}\]
\end{lemma}

\begin{proof}
Follows easily from the definition.
\end{proof}

\subsection{Formal \texorpdfstring{$\Sh{O}_L$}{OL}-vector bundles with marked sections and marked splitting}

In this section we define the relevant vector bundles with marked sections and marked splitting enriched with an action of $\Sh{O}_L$. Although we focus on the sheaves relevant for our purpose, i.e. $\Omega_{\Sh{A}}$ and $\Hs{A}$, the theory can be developed more generally. 

Recall $s = \dlog(P^{\textrm{univ}}_n) \in \Omega_{\Sh{A}}/\beta_n \Omega_{\Sh{A}}$ is the image of the universal generator of $H^{\vee}_n(\Ig{n}{r})$ under the $\dlog$ map. Let $s_{\sigma}$ be its $\sigma$-component under the splitting $\Omega_{\Sh{A}} = \prod_{\sigma \in \Sigma} \Omega_{\Sh{A}}(\sigma)$. Following the VBMS formalism explained in \S\ref{1S31} we define the following formal $\Sh{O}_L$-vector bundles with marked sections.

\begin{defn}\label{D303}
Define $\mathbb{V}^{\Sh{O}_L}(\Omega_{\Sh{A}})$ as the functor that associates to any $\alpha$-admissible formal scheme $\mathfrak{Z} \xrightarrow{\gamma} \Ig{n}{r}$ the following set:
\[
\mathbb{V}^{\Sh{O}_L}(\Omega_{\Sh{A}})(\mathfrak{Z} \xrightarrow{\gamma} \Ig{n}{r}) := \Hom_{\Sh{O}_L \otimes \Sh{O}_{\mathfrak{Z}}}(\gamma^*\Omega_{\Sh{A}}, \Sh{O}_L \otimes \Sh{O}_{\mathfrak{Z}}).
\]
Similarly, define $\mathbb{V}^{\Sh{O}_L}(\mathrm{H}^{\sharp}_{\Sh{A}})$ as the functor
\[
\mathbb{V}^{\Sh{O}_L}(\mathrm{H}^{\sharp}_{\Sh{A}})(\mathfrak{Z} \xrightarrow{\gamma} \Ig{n}{r}) := \Hom_{\Sh{O}_L \otimes \Sh{O}_{\mathfrak{Z}}}(\gamma^*\mathrm{H}^{\sharp}_{\Sh{A}}, \Sh{O}_L \otimes \Sh{O}_{\mathfrak{Z}}).
\]
\end{defn}

\begin{defn}\label{D304}
Define $\V{\Sh{O}_L}(\Omega_{\Sh{A}}, s)$ as the functor that associates to any $\alpha$-admissible formal scheme $\mathfrak{Z} \xrightarrow{\gamma} \Ig{n}{r}$, the following set:
\begin{equation*}
    \V{\Sh{O}_L}(\Omega_{\Sh{A}}, s)(\mathfrak{Z} \xrightarrow{\gamma} \Ig{n}{r}) := \left\{h \in \Hom_{\Sh{O}_L \otimes \Sh{O}_{\mathfrak{Z}}}\big(\gamma^*\Omega_{\Sh{A}}, \Sh{O}_L \otimes \Sh{O}_{\mathfrak{Z}}\big) \, | \, (h \textrm{ mod } \gamma^*\beta_n)(\gamma^*s) = 1 \right\}.
\end{equation*}
Similarly, define $\V{\Sh{O}_L}(\mathrm{H}^{\sharp}_{\Sh{A}}, s)$ as the functor
\begin{equation*}
    \V{\Sh{O}_L}(\mathrm{H}^{\sharp}_{\Sh{A}}, s)(\mathfrak{Z} \xrightarrow{\gamma} \Ig{n}{r}) := \left\{h \in \Hom_{\Sh{O}_L \otimes \Sh{O}_{\mathfrak{Z}}}\big(\gamma^*\mathrm{H}^{\sharp}_{\Sh{A}}, \Sh{O}_L \otimes \Sh{O}_{\mathfrak{Z}}\big) \, | \, (h \textrm{ mod } \gamma^*\beta_n)(\gamma^*s) = 1 \right\}.
\end{equation*}
\end{defn}

\begin{prop}
\begin{enumerate}
    \item We have natural isomorphisms of functors 
    \begin{align*}
    \mathbb{V}^{\Sh{O}_L}(\Omega_{\Sh{A}}) \simeq \mathbb{V}(\Omega_{\Sh{A}}) = \prod_{\sigma \in \Sigma} \mathbb{V}\big(\Omega_{\Sh{A}}(\sigma)\big), \quad \quad \mathbb{V}^{\Sh{O}_L}(\mathrm{H}^{\sharp}_{\Sh{A}}) \simeq \mathbb{V}(\Hs{A}) = \prod_{\sigma \in \Sigma} \mathbb{V}\big(\mathrm{H}^{\sharp}_{\Sh{A}}(\sigma)\big).
    \end{align*}
    where $\mathbb{V}(\Sh{E})$ for any locally finite free sheaf $\Sh{E}$ is defined as in Definition \ref{D1302}.
    \item We have natural isomorphisms of functors
    \begin{align*}
    \V{\Sh{O}_L}(\Omega_{\Sh{A}}, s) \simeq \prod_{\sigma \in \Sigma} \V{}\big(\Omega_{\Sh{A}}(\sigma), s_{\sigma}\big), \quad \quad \V{\Sh{O}_L}(\mathrm{H}^{\sharp}_{\Sh{A}}, s) \simeq  \prod_{\sigma \in \Sigma}\V{}\big(\mathrm{H}^{\sharp}_{\Sh{A}}(\sigma), s_{\sigma}\big).
    \end{align*}
    where $\V{}(\Sh{E}, s) \subset \mathbb{V}(\Sh{E})$ for any locally finite free sheaf $\Sh{E}$ with a marked section $s$ is defined as in Definition \ref{D1303}.
\end{enumerate}
\end{prop}

\begin{proof}
We note that since $K$ is Galois, the natural map $\Sh{O}_L \otimes \Sh{O}_K \to \Sh{O}_K^{\Sigma}$ is an isomorphism. The claims of the proposition then follow immediately from the definitions.
\end{proof}

Taking into account the modified unit root subspace $\Sh{Q} \subset \Hs{A}/p\Hdg{-p^2}$, we define a geometric $\Sh{O}_L$-vector bundle with marked section and marked splitting as follows.

\begin{defn}\label{D305}
\begin{enumerate}
    \item For any $\sigma \in \Sigma$, define $\V{}\big(\mathrm{H}^{\sharp}_{\Sh{A}}(\sigma), s_{\sigma}, \Sh{Q}(\sigma)\big)$ as the functor that associates to any $\alpha$-admissible formal scheme $\mathfrak{Z} \xrightarrow{\gamma} \Ig{n}{r}$ the following set:
    \begin{align*}
        \V{}\big(\mathrm{H}^{\sharp}_{\Sh{A}}(\sigma), s_{\sigma}, \Sh{Q}(\sigma)\big)(\mathfrak{Z}) := \left\{ h \in \V{}\big(\mathrm{H}^{\sharp}_{\Sh{A}}(\sigma), s_{\sigma}\big)(\mathfrak{Z}) \, | \, (h \textrm{ mod } \gamma^*\eta)(\gamma^*\Sh{Q}(\sigma)) = 0 \right\}.
    \end{align*}
    \item Define $\V{\Sh{O}_L}(\mathrm{H}^{\sharp}_{\Sh{A}}, s, \Sh{Q})$ as the functor that associates to any $\alpha$-admissible formal scheme $\mathfrak{Z} \xrightarrow{\gamma} \Ig{n}{r}$ the following set:
    \begin{align*}
        \V{\Sh{O}_L}(\mathrm{H}^{\sharp}_{\Sh{A}}, s, \Sh{Q})(\mathfrak{Z}) := \left\{ h \in \V{\Sh{O}_L}(\mathrm{H}^{\sharp}_{\Sh{A}}, s)(\mathfrak{Z}) \, | \, (h \textrm{ mod }\gamma^*\eta)(\gamma^*\Sh{Q}) = 0 \right\}.
    \end{align*}
\end{enumerate}
\end{defn}

\begin{prop}
We have a natural isomorphism of functors 
\[
\V{\Sh{O}_L}(\mathrm{H}^{\sharp}_{\Sh{A}}, s, \Sh{Q}) \simeq \prod_{\sigma \in \Sigma} \V{}\big(\mathrm{H}^{\sharp}_{\Sh{A}}(\sigma), s_{\sigma}, \Sh{Q}(\sigma)\big).
\]
\end{prop}

\begin{proof}
Clear from the definitions.
\end{proof}

\begin{prop}
The functors $\V{\Sh{O}_L}(\Omega_{\Sh{A}}, s), \V{\Sh{O}_L}(\Hs{A}, s)$ and $\V{\Sh{O}_L}(\Hs{A}, s, \Sh{Q})$ are representable.
\end{prop}

\begin{proof}
Follows immediately from the discussion in \S\ref{1S31}.
\end{proof}

\subsubsection{Formal group action on formal $\Sh{O}_L$-vector bundles}\label{2S321}

The vector bundles $\V{\Sh{O}_L}(\Omega_{\Sh{A}}, s), \V{\Sh{O}_L}(\mathrm{H}^{\sharp}_{\Sh{A}}, s)$ and $\V{\Sh{O}_L}(\mathrm{H}^{\sharp}_{\Sh{A}}, s, \Sh{Q})$ carry an action of the formal group $\mathfrak{T} := 1 + \beta_n \Res_{\Sh{O}_L/\Z}\Ga$ over $\Ig{n}{r}$. This action realizes $\V{\Sh{O}_L}(\Omega_{\Sh{A}}, s)$ as a $\mathfrak{T}$-torsor over $\Ig{n}{r}$. Moreover, there is a natural action of $\mathbb{T}(\Z_p) = {(\Sh{O}_L \otimes \Z_p)}^{\times}$ on the aforementioned formal vector bundles over $\Xra{r}$ which we will describe too. Together we get an action of $\mathfrak{T}^{\textrm{ext}} := \Z_p^{\times}(1 + \beta_n \Res_{\Sh{O}_L/\Z}\Ga)$ on these vector bundles over $\Xra{r}$.

\begin{enumerate}
    \item $\mathfrak{T}^{\textrm{ext}}$-action on $\V{\Sh{O}_L}(\Omega_{\Sh{A}}, s):$ Let $(\rho, h) \in \V{\Sh{O}_L}(\Omega_{\Sh{A}}, s)(R)$ be a ${R}$-valued point of $\V{\Sh{O}_L}(\Omega_{\Sh{A}},s)$. Here $\rho \colon \Spf{R} \to \Ig{n}{r}$ is a morphism of formal schemes and $h \in \Hom_{\Sh{O}_L \otimes R}(\rho^*\Omega_{\Sh{A}}, \Sh{O}_L \otimes R)$. Then $\lambda \in \mathfrak{T}(R)$ acts as $\lambda \ast (\rho, h) := (\rho, \lambda h)$. Let $\lambda \in \mathbb{T}(\Z_p)$. Denote its class in $\Sh{O}_L \otimes \Z/p^n\Z$ by $\bar{\lambda}$. Then $\lambda$ induces an isomorphism ${[\lambda]} \colon \Ig{n}{r} \xrightarrow{\sim} \Ig{n}{r}$ that induces the map 
    \begin{align*}
    H^{\vee}_n(R) &\to H^{\vee}_n(R) \\
    P &\mapsto \bar{\lambda}^{-1}P
    \end{align*}
    on the $R$-valued points of $\Ig{n}{r}$. There is a natural isomorphism $\gamma_{\lambda} \colon [\lambda]^*\Omega_{\Sh{A}} \xrightarrow{\sim} \Omega_{\Sh{A}}$ such that $(\gamma_{\lambda} \textrm{ mod }\beta_n)([\lambda]^*s) = \bar{\lambda}^{-1}s$. Then we define the $\mathbb{T}(\Z_p)$-action as $\lambda \ast (\rho, h) := ([\lambda] \circ \rho, \lambda h \circ \gamma_{\lambda})$.
    
    \item $\mathfrak{T}^{\textrm{ext}}$-action on $\V{\Sh{O}_L}(\mathrm{H}^{\sharp}_{\Sh{A}}, s):$ Let $(\rho, h) \in \V{\Sh{O}_L}(\mathrm{H}^{\sharp}_{\Sh{A}}, s)(R)$, with $\rho \colon \Spf{R} \to \Ig{n}{r}$ a morphism of formal schemes and $h \in \Hom_{\Sh{O}_L \otimes R}(\rho^*\mathrm{H}^{\sharp}_{\Sh{A}}, \Sh{O}_L \otimes R)$. Then $\lambda \in \mathfrak{T}(R)$ acts as $\lambda \ast (\rho, h) := (\rho, \lambda h)$. If $\lambda \in \mathbb{T}(\Z_p)$, then as before we have an isomorphism $[\lambda] \colon \Ig{n}{r} \xrightarrow{\sim} \Ig{n}{r}$. This gives a natural isomorphism $\gamma_{\lambda} \colon [\lambda]^*\mathrm{H}^{\sharp}_{\Sh{A}} \xrightarrow{\sim} \mathrm{H}^{\sharp}_{\Sh{A}}$ such that $(\gamma_{\lambda} \textrm{ mod }\beta_n)([\lambda]^*s) = \bar{\lambda}^{-1}s$. Then we define the $\mathbb{T}(\Z_p)$-action as $\lambda \ast (\rho, h) := ([\lambda] \circ \rho, \lambda h \circ \gamma_{\lambda})$.
    \item $\mathfrak{T}^{\textrm{ext}}$-action on $\V{\Sh{O}_L}(\mathrm{H}^{\sharp}_{\Sh{A}}, s, \Sh{Q}):$ This is defined by restricting the action defined on $\V{\Sh{O}_L}(\mathrm{H}^{\sharp}_{\Sh{A}}, s)$.
\end{enumerate}

\begin{lemma}\label{R301}
The formal group $\mathfrak{T}$ decomposes as $\mathfrak{T} = \prod_{\sigma \in \Sigma} (1 + \beta_n \Ga)$ over $\Ig{n}{r}$. The action of $\mathfrak{T}$ on $\V{\Sh{O}_L}(\Omega_{\Sh{A}}, s)$ and $\V{\Sh{O}_L}(\mathrm{H}^{\sharp}_{\Sh{A}}, s)$ is compatible with the splitting of $\mathfrak{T}$ and the vector bundles. That is to say, if $\lambda = (\lambda_{\sigma}) \in \prod_{\sigma}(1 + \beta_n\Ga)(R)$, and $(\rho, h) \in \V{\Sh{O}_L}(\Omega_{\Sh{A}}, s)(R)$, with $h = (h_{\sigma}) \in \prod_{\sigma} \Hom_{R}(\rho^*\Omega_{\Sh{A}}(\sigma), R)$, then $\lambda \ast (\rho, h) = \prod_{\sigma} (\rho, \lambda_{\sigma}h_{\sigma})$. Similarly for $\V{\Sh{O}_L}(\mathrm{H}^{\sharp}_{\Sh{A}}, s)$.
\end{lemma}

\begin{proof}
This is clear.
\end{proof}

\paragraph{\textbf{$\mathfrak{T}$-action in terms of local coordinates.}}
Based on Lemma \ref{R301}, the action of $\mathfrak{T}$ on $\V{\Sh{O}_L}(\Omega_{\Sh{A}}, s)$ and on $\V{\Sh{O}_L}(\Hs{A}, s, \Sh{Q})$ can be described on local coordinates in the following manner. 

Let $\Spf{R} \xhookrightarrow{} \Ig{n}{r}$ be an open subscheme such that $\Omega_{\Sh{A}}$ and $\mathrm{H}^{\sharp}_{\Sh{A}}$ are trivialized as $\Sh{O}_L \otimes \Sh{O}_{\Ig{n}{r}}$-modules, and such that $\Sh{Q}$ is trivial too over $\Spf{R}/\eta$. Let $X \in \Omega_{\Sh{A}}(R)$ be a lift of $s$ and $Y \in \mathrm{H}^{\sharp}_{\Sh{A}}(R)$ be a lift of a local $\Sh{O}_L \otimes \Sh{O}_{\Ig{n}{r}}$ generator $t$ of $\Sh{Q}$. Let $X_{\sigma}, Y_{\sigma}$ be their $\sigma$-components. The formal schemes $\V{}(\Omega_{\Sh{A}}(\sigma), s_{\sigma})$ and $\V{}(\Hs{A}(\sigma), s_{\sigma}, \Sh{Q}(\sigma))$ are realized as admissible blow-ups of $\mathbb{V}(\Omega_{\Sh{A}}(\sigma))$ and $\mathbb{V}(\Hs{A}(\sigma))$ respectively. Then as elaborated in \S\ref{1S31} the blow-ups are described by the following diagrams.
\begin{equation}\label{EQ205}
\begin{tikzcd}[column sep = large]
	{\Ig{n}{r}} & {\mathbb{V}(\Omega_{\Sh{A}}(\sigma))} & {\V{}(\Omega_{\Sh{A}}(\sigma), s_{\sigma})} \\
	{\Spf{R}} & {\Spf{R}\langle X_{\sigma}\rangle} & {\Spf{R}\langle Z_{\sigma} \rangle}
	\arrow[from=1-2, to=1-1]
	\arrow[from=1-3, to=1-2]
	\arrow[from=2-2, to=2-1]
	\arrow["{X_{\sigma} \mapsto 1 + \beta_n Z_{\sigma}}", from=2-3, to=2-2]
	\arrow[hook, from=2-1, to=1-1]
	\arrow[hook, from=2-2, to=1-2]
	\arrow[hook, from=2-3, to=1-3]
\end{tikzcd}
\end{equation}

\begin{equation}\label{eq:5}
\begin{tikzcd}[column sep = large]
\Ig{n}{r} & \mathbb{V}(\mathrm{H}^{\sharp}_{\Sh{A}}(\sigma)) \arrow[l]    & {\V{}(\mathrm{H}^{\sharp}_{\Sh{A}}(\sigma), s_{\sigma})} \arrow[l] & {\V{}(\mathrm{H}^{\sharp}_{\Sh{A}}, s_{\sigma}, \Sh{Q}(\sigma))} \arrow[l] \\
\Spf{R} \arrow[u, hook]       & {\Spf{R}\langle X_{\sigma}, Y_{\sigma}\rangle} \arrow[u, hook] \arrow[l] & {\Spf{R}\langle Z_{\sigma}, Y_{\sigma}\rangle} \arrow[u, hook] \arrow[l, "{X_{\sigma} \mapsto 1 + \beta_n Z_{\sigma}}"]   & {\Spf{R}\langle Z_{\sigma}, W_{\sigma}\rangle} \arrow[u, hook] \arrow[l, "{Y_{\sigma} \mapsto \alpha W_{\sigma}}"]     
\end{tikzcd}
\end{equation}

Let $\lambda \in \mathfrak{T}(R)$ and let $\lambda = (\lambda_{\sigma})_{\sigma}$ be its decomposition into coordinates. The action of $\lambda_{\sigma}$ on $Z_{\sigma}$ is such that $\lambda_{\sigma} \ast (1+\beta_n Z_{\sigma}) = \lambda_{\sigma}(1+\beta_n Z_{\sigma})$. In other words, \[\lambda_{\sigma} \ast Z_{\sigma} = \frac{\lambda_{\sigma}-1}{\beta_n} + \lambda_{\sigma}Z_{\sigma}.\]
Similarly, $\lambda_{\sigma}$ acts on $W_{\sigma}$ via \[\lambda_{\sigma} \ast W_{\sigma} = \lambda_{\sigma}W_{\sigma}.\]

\subsection{\texorpdfstring{$p$}{p}-adic interpolation of \texorpdfstring{$\Omega_{\Sh{A}}$}{OmegaA} and \texorpdfstring{$\mathrm{H}^{\sharp}_{\Sh{A}}$}{HA}}
Let $n, r, \alpha, I$ be as fixed in the beginning of the section. Denote by $\rho' \colon \V{\Sh{O}_L}(\mathrm{H}^{\sharp}_{\Sh{A}}, s, \Sh{Q}) \xrightarrow{} \Ig{n}{r}$ and $\nu' \colon \V{\Sh{O}_L}(\Omega_{\Sh{A}}, s) \xrightarrow{} \Ig{n}{r}$ the projections.
\begin{defn}
\begin{enumerate}
    \item For $k = k^0_{\alpha, I} \colon (\Sh{O}_L \otimes \Z_p)^{\times} \to \Lambda^0_{\alpha,I}$ the universal character, define \[\mathfrak{w}'_{k,\alpha, I} := \nu'_*\Sh{O}_{\V{\Sh{O}_L}(\Omega_{\Sh{A}}, s)}[k].\]
    The sections of this sheaf by definition are the functions $f \in \nu'_*\Sh{O}_{\V{\Sh{O}_L}(\Omega_{\Sh{A}}, s)}$ that transform via $\lambda \ast f = k(\lambda)f$ under the $\mathfrak{T}$-action.
    \item For $k = k^0_{\alpha, I}$ define \[\mathbb{W}'_{k,\alpha,I} := \rho'_*\Sh{O}_{\V{\Sh{O}_L}(\mathrm{H}^{\sharp}, s, \Sh{Q})}[k].\] 
\end{enumerate}
\end{defn}

Let $k_{\sigma} \colon 1 + \beta_n \Ga \to \Gm$ be the restriction of $k^0_{\alpha, I}$ to the $\sigma$-component of $\mathfrak{T} = \prod_{\sigma} (1 + \beta_n \Ga)$.

\begin{prop}\label{P303}
Let $\nu'_{\sigma} \colon \V{}(\Omega_{\Sh{A}}(\sigma), s_{\sigma}) \to \Ig{n}{r}$ be the projection for each $\sigma$. Then 
\[
\mathfrak{w}'_{k,\alpha, I} = \hat{\otimes}_{\sigma}{(\nu'_{\sigma})}_*\Sh{O}_{\V{}(\Omega_{\Sh{A}}(\sigma), s_{\sigma})}[k_{\sigma}].
\]
In particular, $\mathfrak{w}'_{k,\alpha,I}$ is a line bundle on $\Ig{n}{r}$.
\end{prop}

\begin{proof}
Denote $\hat{\otimes}_{\sigma}{(\nu'_{\sigma})}_*\Sh{O}_{\V{}(\Omega_{\Sh{A}}(\sigma), s_{\sigma})}[k_{\sigma}]$ by $\tilde{w}$. Then indeed $\tilde{\omega} \subset \mathfrak{w}'_{k,\alpha,I}$ as follows from Lemma \ref{R301}. 

Take a Zariski open $\Spf{
R} \subset \Ig{n}{r}$ that trivializes $\Omega_{\Sh{A}}$. Then (\ref{EQ205}) shows that ${(\nu')}^{-1}(\Spf{R}) \simeq \Spf{R\langle \{Z_{\sigma}\}_{\sigma \in \Sigma}\rangle}$. 

The description of the $\mathfrak{T}$-action on $\V{\Sh{O}_L}(\Omega_{\Sh{A}}, s)$ in terms of local coordinates imply by \cite[Lemma 3.9]{andreatta2021triple}  \[\tilde{\omega}(\Spf{R}) = R \cdot \prod_{\sigma} k_{\sigma}(1 + \beta_n Z_{\sigma}).\]
Denote $\prod_{\sigma} k_{\sigma}(1 + \beta_n Z_{\sigma})$ by $k(1 + \beta_n Z)$. 

If $f \in \mathfrak{w}'_{k,\alpha,I}(\Spf{R})$, then $f/k(1 + \beta_n Z) \in {R\langle \{ Z_{\sigma} \}_{\sigma} \rangle}^{\mathfrak{T}(R)}$. Then the problem reduces to showing that the $\mathfrak{T}$ invariant functions are simply $R$. 

For the one variable case, this follows from an application of the Weierstrass preparation theorem. If $1 + \beta_n R$ acts on $R\langle Z\rangle$ via $t \ast Z = \frac{t-1}{\beta_n} + tZ$, then take $f \in {R\langle Z \rangle}$ invariant under the action of $1 + \beta_n R$. Suppose $f = \sum a_n Z^n$. Then for any $a \in R$, $\sum a_n Z^n = (1+\beta_n a) \ast (\sum a_n Z^n) = \sum a_n (a + (1+\beta_n a)Z)^n$. Letting $Z = 0$, we see that $a_0 = \sum a_n a^n$ for any $a \in R$, which shows that $a_n = 0$ for all $n > 0$.

For the general case the result follows by induction on the number of variables. Choose a bijection $\Sigma \simeq \{1, \dots, g\}$. Suppose $f = \sum a_n Z_g^n \in {R\langle Z_1, \dots, Z_g \rangle}^{\mathfrak{T}(R)}$, with $a_n \in R\langle Z_1, \dots, Z_{g-1}\rangle$ for all $n$. Then for any element $\lambda = (\lambda_i) \in \mathfrak{T}(R)$, such that $\lambda_g = 1$, $f = \lambda \ast f = \sum (\lambda \ast a_n) Z_g^n$. This shows that $\lambda \ast a_n = a_n$ for all $n$, and then by induction $a_n \in R$. Finally $a_n = 0$ for all $n > 0$ by the same argument as above.  
\end{proof}

\begin{rem}
The isomorphism classes of $\Sh{O}_L \otimes \Sh{O}_{\Ig{n}{r}}$ line bundles can be naturally identified with elements of $H^1(\Ig{n}{r}, \Res_{\Sh{O}_L/\Z} \Gm)$. The subgroup $H^1(\Ig{n}{r}, 1 + \beta_n \Res_{\Sh{O}_L/\Z} \Ga)$ classifies precisely $\Sh{O}_L \otimes \Sh{O}_{\Ig{n}{r}}$ line bundles $\sh{L}$ with a marked section $s \in \sh{L}/\beta_n \sh{L}$. Thus the isomorphism class of $(\Omega_{\Sh{A}}, \dlog(P_n^{\textrm{univ}}))$ defines an element of $H^1(\Ig{n}{r}, \mathfrak{T})$. Then $\mathfrak{w}'_{k,\alpha,I}$ defined as above is nothing but its image under the map induced by extension of structural group $H^1(\Ig{n}{r}, \mathfrak{T}) \xrightarrow{k} H^1(\Ig{n}{r}, \Gm)$.
\end{rem}

Next we give a local description of $\mathbb{W}'_{k,\alpha,I}$.

\begin{prop}
Let $\rho'_{\sigma} \colon \V{}(\mathrm{H}^{\sharp}_{\Sh{A}}(\sigma), s_{\sigma}, \Sh{Q}(\sigma)) \to \Ig{n}{r}$ be the projection for each $\sigma$. Then 
\[
\mathbb{W}'_{k,\alpha,I} = \hat{\otimes}_{\sigma} (\rho'_{\sigma})_* \Sh{O}_{\V{}(\mathrm{H}^{\sharp}_{\Sh{A}}(\sigma), s_{\sigma}, \Sh{Q}(\sigma))}[k_{\sigma}].
\]
\end{prop}

\begin{proof}
Let $\tilde{\mathbb{W}} = \hat{\otimes}_{\sigma} (\rho'_{\sigma})_* \Sh{O}_{\V{}(\mathrm{H}^{\sharp}_{\Sh{A}}(\sigma), s_{\sigma}, \Sh{Q}(\sigma))}[k_{\sigma}]$. Then clearly $\tilde{\mathbb{W}} \subset \mathbb{W}'_{k,\alpha,I}$. 

Take a Zariski open $\Spf{R} \subset \Ig{n}{r}$ that trivializes $\Hs{A}$ compatibly with a trivialization of $\Omega_{\Sh{A}}$ and of $\Sh{Q}$ modulo $\eta$. Choosing coordinates as in the local description (\ref{eq:5}), we see that $(\rho')^{-1}(\Spf{R}) \simeq \Spf{R}\langle \{Z_{\sigma}, W_{\sigma}\}_{\sigma \in \Sigma}\rangle$.

The description of $\mathfrak{T}$-action on $\V{\Sh{O}_L}(\Hs{A}, s, \Sh{Q})$ in local coordinates, together with \cite[Lemma 3.13]{andreatta2021triple} shows that \[\tilde{\mathbb{W}}(\Spf{R}) = R\langle \{\frac{W_{\sigma}}{1+\beta_nZ_{\sigma}}\}_{\sigma \in \Sigma}\rangle \cdot k(1+\beta_n Z)\]
where we recall from the previous Proposition that $k(1+\beta_n Z) = \prod_{\sigma} k_{\sigma}(1+\beta_n Z_{\sigma})$. Since $k(1+\beta_n Z)$ is a unit, in order to prove the reverse inclusion, it will be sufficient to show that $R\langle \{Z_{\sigma}, W_{\sigma}\}_{\sigma}\rangle^{\mathfrak{T}(R)} = R\langle \{\frac{W_{\sigma}}{1+\beta_nZ_{\sigma}}\}_{\sigma \in \Sigma}\rangle$. We prove this by induction on the cardinality of $\Sigma$. 

For the one variable case, this has been proved in loc. cit. For the general case, choose a bijection $\Sigma \simeq \{1, \dots, g\}$. Let $V_g := \frac{W_g}{1+\beta_nZ_g}$. The inclusion $R\langle\{Z_i, W_i\}_{i=1}^{g-1}\rangle\langle Z_g, V_g\rangle \to R\langle \{Z_i, W_i\}_{i=1}^g\rangle$ is an isomorphism of topological rings. Let $f \in R\langle \{Z_i, W_i\}_{i=1}^g\rangle^{\mathfrak{T}(R)}$. Write $f = \sum_{n\geq 0} A_n(Z_g)V_g^n$ for $A_n(Z_g) \in R\langle \{Z_i, W_i\}_{i=1}^{g-1}\rangle \langle Z_g \rangle$. For any $\lambda \in \mathfrak{T}(R)$ with $\lambda_i = 1$ for all $i \neq g$, we have $A_n(Z_g) = A_n(\lambda_g \ast Z_g)$ for all $n$. Thus for $\lambda_g = 1 + \beta_n a$ for $a \in R$, we have $A_n(Z_g) = A_n(a + \lambda_g Z_g)$. Putting $Z_g = 0$, we have $A_n(0) = A_n(a)$ for any $a \in R$. The Weierstrass preparation theorem then implies that $A_n(Z_g) = A_n(0) \in R\langle \{Z_i, W_i\}_{i=1}^{g-1}\rangle$. Thus $f = \sum A_n V_g^n$ with $A_n \in R\langle \{Z_i, W_i\}_{i=1}^{g-1}\rangle$. The induction hypothesis then implies that $A_n \in R\langle \{\frac{W_i}{1+\beta_n Z_i}\}_{i=1}^{g-1}\rangle$. This proves the claim.
\end{proof}

\begin{cor}\label{C302}
Let $\Spf{R} \subset \Ig{n}{r}$ be a Zariski open subset where $\mathrm{H}^{\sharp}_{\Sh{A}}, \Omega_{\Sh{A}}$ and $\Sh{Q}$ are trivialized. Then with the notation of (\ref{eq:5}),
\[
\mathbb{W}'_{k,\alpha,I}(\Spf{R}) \simeq R\left\langle \left\{V_{\sigma}\right\}_{\sigma \in \Sigma}\right\rangle \cdot k(1+\beta_n Z), \quad V_{\sigma} := \frac{W_{\sigma}}{1+\beta_n Z_{\sigma}}.
\]
\end{cor}

Let $\nu \colon \V{\Sh{O}_L}(\Omega_{\Sh{A}}, s) \xrightarrow{\nu'} \Ig{n}{r} \xrightarrow{h_n} \Xra{r}$ and $\rho \colon \V{\Sh{O}_L}(\Hs{A}, s, \Sh{Q}) \xrightarrow{\rho'} \Ig{n}{r} \xrightarrow{h_n} \Xra{r}$ be the projections.

\begin{defn}
\begin{enumerate}
    \item For $k = k^0_{\alpha, I}$ define $\w{0} := \left(\nu_*\Sh{O}_{\V{\Sh{O}_L}(\Omega_{\Sh{A}}, s)}\right)[k]$. This by definition is the sheaf of sections $f \in \nu_*\Sh{O}_{\V{\Sh{O}_L}(\Omega_{\Sh{A}}, s)}$ that transform via $k$ for the action of $\mathfrak{T}^{\textrm{ext}}$. This is the interpolation sheaf of Hilbert modular forms for the universal weight $k$.
    
    \item For $k = k^0_{\alpha, I}$ define $\W{0} := \left(\rho_*\Sh{O}_{\V{\Sh{O}_L}(\Hs{A}, s, \Sh{Q})} \right)[k]$. This by definition is the sheaf of sections $f \in \rho_*\Sh{O}_{\V{\Sh{O}_L}(\Hs{A}, s, \Sh{Q})}$ that transform via $k$ for the action of $\mathfrak{T}^{\textrm{ext}}$. This is the interpolation sheaf of de Rham classes for the universal weight $k$.
\end{enumerate}
\end{defn}

\begin{lemma}\label{L301}
$\w{0} = (h_n)_*\mathfrak{w}'_{k,\alpha,I}[k]$ and $\W{0} = (h_n)_*\mathbb{W}'_{k,\alpha,I}[k]$ for the residual action of $(\Sh{O}_L \otimes \Z_p)^{\times}$
\end{lemma}

\begin{proof}
This is clear.
\end{proof}

\begin{rem}
Note that the universal weight $k = k^0_{\alpha, I}$ kills the torsion group $\Delta \subset (\Sh{O}_L \otimes \Z_p)^{\times}$. We will later take care of the torsion part of the character and define the interpolation sheaves $\w{}$ and $\W{}$ for the univeral weight $k_{\alpha,I} \colon (\Sh{O}_L \otimes \Z_p)^{\times} \to \Lambda_{\alpha,I}^{\times}$, as promised in the beginning of the section by tensoring $\w{0}$ and $\W{0}$ respectively with an appropriate coherent sheaf on $\Xra{r} \times_{\mathfrak{W}^0_{\alpha,I}}\mathfrak{W}_{\alpha,I}$. In particular, $\w{0}$ and $\W{0}$ will be the restriction of $\w{}$ and $\W{}$ to the connected component of the trivial character.
\end{rem}

\paragraph{\textbf{Filtration on $\mathbb{W}^0_k$.}}
An important result of this work is the following.

\begin{theorem*}
The sheaf $\W{0}$ comes equipped with a natural Hodge filtration induced by the Hodge filtration on $\mathrm{H}_{\Sh{A}}$ such that the Gauss--Manin connection $\nabla$ on $\W{0}$ satisfies Griffiths' transversality with respect to the Hodge filtration.
\end{theorem*}

We define the Hodge filtration on $\W{0}$ later in Lemma \ref{L2401} and prove Griffiths' transversality for $\nabla$ in Theorem \ref{T401}. The way we show that $\W{0}$ is equipped with a Hodge filtration is by producing a filtration locally on coordinates, and then proving that it glues. But before we prove these results, we will introduce a finer filtration on $\W{0}$ that will eventually help us to prove that the Hodge filtration is well-defined.

Choose a bijection $\Sigma \simeq \{1, \dots, g\}$. Consider the lexicographic order on $\N^g$: \\
$(a_1, \dots, a_g) > (b_1, \dots, b_g)$ if and only if\\
(1) $\sum a_i > \sum b_i$, or \\
(2) if $\sum a_i = \sum b_i$, then for the first index where $a_i \neq b_i$, $a_i > b_i$. \\
Since this defines a well ordering on $\N^g$ we get an order preserving bijection $\Xi \colon \N \simeq \N^g$. This allows us to define a natural filtration on $\mathbb{W}^0_k$.

\begin{lemma}\label{D308}
Let $f_0 \colon {\V{\Sh{O}_L}(\mathrm{H}^{\sharp}_{\Sh{A}}, s, \Sh{Q})} \to {\V{\Sh{O}_L}(\Omega_{\Sh{A}}, s)}$ be the projection. There is an increasing filtration $\{\Fil{i}\}_{i \geq 0}$ on ${f_0}_*\Sh{O}_{\V{\Sh{O}_L}(\mathrm{H}^{\sharp}_{\Sh{A}}, s, \Sh{Q})}$ with $\Fil{0}({f_0}_*\Sh{O}_{\V{\Sh{O}_L}(\Hs{A}, s, \Sh{Q})}) = \Sh{O}_{\V{\Sh{O}_L}(\Omega_{\Sh{A}}, s)}$. On local coordinates as in (\ref{eq:5}), 
\[
\Fil{i}\left({f_0}_*\Sh{O}_{\V{\Sh{O}_L}(\mathrm{H}^{\sharp}_{\Sh{A}}, s, \Sh{Q})}(\Spf{R})\right) = \sum_{j \leq i} R\langle\{Z_1, \dots, Z_g\}\rangle \cdot W^{\Xi(j)}.
\]
\end{lemma}

\begin{proof}
We need to show that the local description glues. For $\Fil{0}$ this is obvious by definition. For two different choice of $\Sh{O}_L \otimes R$-basis of ${\Hs{A}}_{|\Spf{R}}$, say $X, Y$ and $X', Y'$, with $X, X'$ being lifts of the marked section and $Y, Y'$ being lifts of some generator of $\Sh{Q}$, we get two different local coordinate description of $\V{\Sh{O}_L}(\Hs{A}, s, \Sh{Q})_{|\Spf{R}}$. Since all the filtered pieces contain $\Fil{0}$, we may assume $X = X'$. Then the components of $Y$ and $Y'$ are related by $Y'_i = u_iY_i + a_i \eta X_i$ for all $i$, with $u_i, a_i \in R$. This implies that the isomorphism $R\langle \{Z_i, W'_i\}_{i=1}^g\rangle \xrightarrow{\sim} R\langle \{Z_i, W_i\}_{i=1}^g \rangle$ is given by sending $W'_i \mapsto u_iW_i + a_i(1+\beta_n Z_i)$. Clearly, this isomorphism respects the filtration given by the lexicographic ordering. 
\end{proof}

Recall $\rho' \colon \V{\Sh{O}_L}(\Hs{A}, s, \Sh{Q}) \to \Ig{n}{r}$ was the projection.

\begin{lemma}
The filtration on $\rho'_*\Sh{O}_{\V{\Sh{O}_L}(\Hs{A},s,\Sh{Q})}$ induced by taking direct image of the filtration defined in Lemma \ref{D308} is stable for the action of $\mathfrak{T}$. Therefore, we can define a filtration on $\mathbb{W}'_{k,\alpha,I}$ as $\Fil{i}\mathbb{W}'_{k,\alpha,I} = \Fil{i}\left(\rho'_*\Sh{O}_{\V{\Sh{O}_L}(\Hs{A}, s, \Sh{Q})}\right)[k]$ with the property that $\Fil{0}\mathbb{W}'_{k,\alpha,I} = \mathfrak{w}'_{k,\alpha,I}$ and $\Gr{i}\mathbb{W}'_{k,\alpha,I} \simeq \mathfrak{w}'_{k,\alpha,I} \otimes {\eta}^{-\ell(i)}(HW\cdot\omega_{\Sh{A}}^{-2})^{\Xi(i)}$, where for $\Xi(i) = (a_1, \dots, a_g)$ we let $\ell(i) = \sum a_k$, and $HW^{\Xi(i)} = \prod HW(k)^{a_k}$.
\end{lemma}

\begin{proof}
This is clear from the description of the action on local coordinates.
\end{proof}

Here we collect a few results that will allow us to prove that $\w{0}$ is a line bundle and the filtration on $\mathbb{W}'_{k,\alpha, I}$ descends to a filtration on $\W{0}$ which then can be realized as the completion of the colimit of its filtered pieces, which are locally free sheaves of finite rank.

\begin{lemma}\label{L302}
Let $\Sh{O}_{\Ig{1}{r}}^{\circ \circ}$ be the ideal of topologically nilpotent elements in $\Sh{O}_{\Ig{1}{r}}$. With $I, r$ as fixed in the beginning of the section, for any $2 \leq l \leq a+r$,
\[
k(1 + p^{l-1}(\Sh{O}_L \otimes \Z_p)) - 1 \subset \Hdg{\frac{p^l-p}{p-1}}\Sh{O}_{\Ig{1}{r}}^{\circ\circ}.
\]
\end{lemma}

\begin{proof}
\cite[Lemma 4.4]{andreatta2016adic}.
\end{proof}

\begin{lemma}\label{L303}
The natural $\mathfrak{T}^{\textrm{ext}}$-equivariant map $\Sh{O}_{\Ig{n}{r}} \to \Sh{O}_{\V{\Sh{O}_L}(\Omega_{\Sh{A}}, s)}$ induces an isomorphism 
\[
\Sh{O}_{\Ig{n}{r}}/q\Sh{O}_{\Ig{n}{r}} \xrightarrow{\sim} \mathfrak{w}'_k/q\mathfrak{w}'_k.
\]
\end{lemma}

\begin{proof}
\cite[Lemme 4.5]{andreatta2016adic}.
\end{proof}

\begin{lemma}\label{L304}
Let $h \colon \Ig{n}{r} \to \Ig{n-1}{r}$ be the projection for any $n$. $h$ is finite and the trace map $\mathrm{Tr}_h \colon h_*\Sh{O}_{\Ig{n}{r}} \to \Sh{O}_{\Ig{n-1}{r}}$ induced by the trace on the adic generic fibre satisfies
\[
\Hdg{p^{n-1}}\Sh{O}_{\Ig{n-1}{r}} \subset \mathrm{Tr}_{h}(h_*\Sh{O}_{\Ig{n}{r}})
\]
for any $2 \leq n \leq a+r$.
\end{lemma}

\begin{proof}
\cite[Proposition 3.4]{andreatta2016adic}.
\end{proof}

\begin{lemma}\label{L305}
Let $\Spf{R} \subset \Xra{r}$ be an open where $\Hdg{}$ is trivialized. Letting $c_0 = 1$, for every $1 \leq n \leq a+r$, there exists $c_n \in \Hdg{-\frac{p^n-p}{p-1}}\Sh{O}_{\Ig{n}{r}}(\Spf{R})$ satisfying $\mathrm{Tr}_h(c_n) = c_{n-1}$.
\end{lemma}

\begin{proof}
Immediate from Lemma \ref{L304}.
\end{proof}

Recall $\rho \colon \V{\Sh{O}_L}(\Hs{A}, s, \Sh{Q}) \to \Xra{r}$ was the projection and $\W{0} = \rho_*\Sh{O}_{\V{\Sh{O}_L}(\Hs{A},s, \Sh{Q})}[k]$.

\begin{theorem}\label{T2301}
The action of $\mathfrak{T}^{\textrm{ext}}$ on $\rho_*\Sh{O}_{\V{\Sh{O}_L}(\Hs{A},s, \Sh{Q})}$ preserves the filtration $\{\Fil{i}\}_i$ induced by taking direct image of the filtration defined in Lemma \ref{D308}. Let $\Fil{i}\W{0} = \Fil{i}\left( \rho_*\Sh{O}_{\V{\Sh{O}_L}(\Hs{A},s, \Sh{Q})} \right)[k]$.
\begin{enumerate}
    \item $\Fil{i}\W{0}$ is a finite locally free $\Sh{O}_{\Xra{r}}$-module.
    \item $\W{0}$ is the $\alpha$-adic completion of $\varinjlim_i \Fil{i}\W{0}$.
    \item $\Fil{0}\W{0} = \w{0}$ and $\Gr{i}\W{0} \simeq \w{0} \otimes \eta^{-\ell(i)}(HW \cdot\omega_{\Sh{A}}^{-2})^{\Xi(i)}$.
\end{enumerate}
\end{theorem}

\begin{proof}
We already know the similar results for $\mathbb{W}'_{k,\alpha,I}$. Also it is clear that $\Fil{0}$ is preserved by the $\mathfrak{T}^{\textrm{ext}}$-action and $\Fil{0}\W{0} = \w{0}$. By Lemma \ref{L301}, $\W{0} = (h_n)_*\mathbb{W}'_{k,\alpha,I}[k]$ for the residual action of $(\Sh{O}_L \otimes \Z_p)^{\times}$. The idea is to pick generators of $\Gr{i}\mathbb{W}'_{k,\alpha,I}$ and modify them to produce generators of $\Gr{i}\W{0}$. To that intent, recall we have a $\mathfrak{T}^{\textrm{ext}}$ equivariant isomorphism $\Sh{O}_{\Ig{n}{r}}/q \simeq \mathfrak{w}'_{k,\alpha,I}/q$ by Lemma \ref{L303}. Let $\Spf{R} \subset \Xra{r}$ be a Zariski open that trivializes $\omega_{\Sh{A}}$ and let $\omega$ be a $\Sh{O}_L \otimes \Sh{O}_{\Xra{r}}$ generator. Let $V$ be the pullback of $\Spf{R}$ to $\Ig{n}{r}$. This gives a $\mathfrak{T}^{\textrm{ext}}$ equivariant isomorphism \[\Sh{O}_{\Ig{n}{r}}(V)/q \otimes \eta^{-\ell(i)}(HW\cdot \omega_{\Sh{A}}^{-2})^{\Xi(i)} \xrightarrow{\sim} \Gr{i}\mathbb{W}'_k(V)/q = \mathfrak{w}'_k(V)/q \otimes \eta^{-\ell(i)}(HW\cdot \omega_{\Sh{A}}^{-2})^{\Xi(i)}.\]

Let $\bar{s}_i$ be the image of the class of $\eta^{-\ell(i)}(HW\cdot \omega_{\Sh{A}}^{-2})^{\Xi(i)}$. In particular $t \ast \bar{s}_i = \bar{s}_i$ for all $t \in (\Sh{O}_L \otimes \Z_p)^{\times}$, since $\eta^{-\ell(i)}(HW\cdot \omega_{\Sh{A}}^{-2})^{\Xi(i)}$ is defined over $\Xra{r}$. Pick a lift $s_i$ of $\bar{s}_i$ to $\Fil{i}\mathbb{W}'_{k,\alpha,I}$. 

Choose lifts $\tilde{\tau}$ of $\tau \in (\Sh{O}_L/p^n\Sh{O}_L)^{\times}$ in $(\Sh{O}_L \otimes \Z_p)^{\times}$. With $c_n$ as in Lemma \ref{L305}, define 
\[
\tilde{s}_i := \sum_{\tau \in (\Sh{O}_L/p^n\Sh{O}_L)^{\times}} k(\tilde{\tau})^{-1}\tau(c_n s_i) \in \Hdg{-\frac{p^n-p}{p-1}}\mathbb{W}'_{k,\alpha,I}(V).
\]
We claim that $\tilde{s}_i \in \Fil{i}\mathbb{W}'_{k,\alpha,I}(V)$ and its image generates $\Gr{i}\mathbb{W}'_{k,\alpha,I}$. Moreover since $(\Sh{O}_L \otimes \Z_p)^{\times}$ acts on $\tilde{s}_i$ via $k$, it descends to $\mathbb{W}^0_k$. To prove the claim we note that
\begin{align*}
    \tilde{s}_i - s_i &= \left(\sum_{\tau \in (\Sh{O}_L/p^n\Sh{O}_L)^{\times}} k(\tilde{\tau})^{-1}\tau(c_ns_i)\right) - s_i \\
    &\in \sum_{\tau \in (\Sh{O}_L/p^n\Sh{O}_L)^{\times}} k(\tilde{\tau})^{-1}(\tau(c_ns_i) - \tau(c_n)s_i) + R^{\circ \circ} \Fil{i}\mathbb{W}'_{k,\alpha,I} \\
    &\subset \sum_{\tau \in (\Sh{O}_L/p^n\Sh{O}_L)^{\times}} k(\tilde{\tau})^{-1}\tau(c_n)(\tau(s_i) - s_i) + R^{\circ \circ} \Fil{i}\mathbb{W}'_{k,\alpha,I} \\
    &\subset R^{\circ \circ} \Fil{i}\mathbb{W}'_{k,\alpha,I} + \Fil{i-1}\mathbb{W}'_{k,\alpha,I}.
\end{align*}
Here we used the fact that $\sum k(\tilde{\tau})^{-1}\tau(c_n) \in 1 + R^{\circ \circ}\Sh{O}_{\Ig{n}{r}}$ whose proof we refer to \cite[Lemme 5.4]{Andreatta2018leHS}.
The fact that $\tilde{s}_i$ generates $\Gr{i}\mathbb{W}'_{k,\alpha,I}$ follows by noting that it does so modulo $R^{\circ \circ}$.

So we have produced a local basis of $\Fil{i}\mathbb{W}'_{k,\alpha,I}$ for each $i$ that descends to $\W{0}$. This proves that the filtration on $\rho_*\Sh{O}_{\V{\Sh{O}_L}(\Hs{A},s, \Sh{Q})}$ is preserved by the $\mathfrak{T}^{\textrm{ext}}$-action and so the filtration on $\mathbb{W}'_{k,\alpha,I}$ descend to a filtration on $\W{0}$. The rest of the claims in the theorem follow immediately. 
\end{proof}

Let $\chi \colon \Delta \subset (\Sh{O}_L \otimes \Z_p)^{\times} \xrightarrow{k^{\textrm{un}}} \Lambda_{\alpha,I}^{\times}$ be the finite part of the universal character. For $p \neq 2$, $\Delta = (\Sh{O}_L \otimes Z_p)^{\times}/1+p(\Sh{O}_L \otimes Z_p)$ and for $p = 2$, $\Delta$ is a quotient of $(\Sh{O}_L \otimes Z_p)^{\times}/1 + 4(\Sh{O}_L \otimes \Z_p)$. Using this we view $\chi$ as a character of $(\Sh{O}_L/q\Sh{O}_L)^{\times}$. Let $p \colon \bar{\mathfrak{M}}_{r,\alpha,I} := \Xra{r} \times_{\mathfrak{W}^0_{\alpha,I}} \mathfrak{W}_{\alpha,I} \to \Xra{r}$ be the projection induced by the finite flat base change $\mathfrak{W}_{\alpha,I} \to \mathfrak{W}^0_{\alpha,I}$. 

\begin{defn}
For $i = 1$ if $p \neq 2$ and $i = 2$ if $p = 2$, define a coherent sheaf $\w{\chi} := \left(p^*(f_i)_*\Sh{O}_{\Ig{i}{r}}\right)[\chi^{-1}]$ for the action of $(\Sh{O}_L/q\Sh{O}_L)^{\times}$ on $f_i \colon \Ig{i}{r} \to \Xra{r}$. 
\end{defn}

\begin{defn}
\begin{enumerate}
    \item Define the sheaf of overconvergent Hilbert modular forms of weight $k = k^{\textrm{un}}$ to be $\w{} = p^*\w{0} \otimes \w{\chi}$.
    
    \item Define the sheaf of overconvergent de Rham classes to be $\W{} = p^*\W{0} \otimes \w{\chi}$.
\end{enumerate}
\end{defn}

\begin{prop}
The sheaf $\W{}$ is equipped with a filtration by coherent $\Sh{O}_{\bar{\mathfrak{M}}_{r,\alpha,I}}$ modules $\Fil{i}\W{}$ and moreover $\W{}$ is the $\alpha$-adic completion of $\varinjlim_i \Fil{i}\W{}$. We have $\Fil{0}\W{} = \w{}$ and $\Gr{i}\W{} \simeq \w{} \otimes \eta^{-\ell(i)}(HW\cdot \omega_{\Sh{A}}^{-2})^{\Xi(i)}$.
\end{prop}

\begin{proof}
Define $\Fil{i}\W{} := p^*\Fil{i}\W{0} \otimes \w{\chi}$. The rest of the claims follow immediately from Theorem \ref{T2301}.
\end{proof}

We remarked in the introduction to the section that the construction of $\w{}$ appears in the previous work of Andreatta--Iovita--Pilloni \cite{andreatta2016adic}. Here we compare our construction to theirs and show why we get isomorphic sheaves. 

In \cite{andreatta2016adic} the authors consider a torsor $\mathfrak{F}_{n,r,I}$ over $\Ig{n}{r}$ for the group $\mathfrak{T}$. This torsor is defined on points $\Spf{R} \xrightarrow{\gamma} \Ig{n}{r}$ for any normal admissible $\Lambda^0_{\alpha,I}$-algebra $R$ as follows. 
\[
\mathfrak{F}_{n,r,I}(R) := \{ \omega \in \omega_{\Sh{A}} \, | \, \omega = \gamma^*(s) \in \gamma^*\Omega_{\Sh{A}}/\beta_n \}.
\]
The action of $\mathfrak{T}$ on $\mathfrak{F}_{n,r,I}$ is the obvious one, i.e. $\lambda \in \mathfrak{T}(R)$ acts via $\lambda \ast \omega = \lambda \omega$. Moreover there is an action of $\mathfrak{T}^{\textrm{ext}}$ on $\mathfrak{F}_{n,r,I}$ over $\Xra{r}$. This is given by first noting that any point of $\mathfrak{F}_{n,r,I}(R)$ can be seen as a pair $(P, \omega)$ where $P \in H^{\vee}_n(R)$ and $\omega \in \omega_{\Sh{A}}(R), \omega = \dlog(P) \textrm{ mod } \beta_n$. Then $\lambda \in (\Sh{O}_L \otimes \Z_p)^{\times}$ acts via $\lambda \ast (P, \omega) = (\bar{\lambda} P, \lambda \omega)$, where $\bar{\lambda}$ is the class of $\lambda$ in $(\Sh{O}_L/p^n\Sh{O}_L)^{\times}$. Then they define the sheaf of Hilbert modular forms for universal weight $k = k^0_{\alpha,I}$ as follows. Let $\nu' \colon \mathfrak{F}_{n,r,I} \to \Xra{r}$ be the projection.

\begin{defn}
The sheaf of Hilbert modular forms for weight $k = k^0_{\alpha,I}$ is defined as 
$\w{0,\textrm{old}} := \nu'_*\Sh{O}_{\mathfrak{F}_{n,r,I}}[k^{-1}]$
for the action of $\mathfrak{T}^{\textrm{ext}}$. 
\end{defn}

We now show why the sheaf $\w{0}$ is naturally isomorphic to $\w{0,\textrm{old}}$.

\begin{prop}\label{P239}
There is an isomorphism of formal schemes over $\Xra{r}$, $a \colon \mathfrak{F}_{n,r,I} \xrightarrow{\sim} \V{\Sh{O}_L}(\Omega_{\Sh{A}}, s)$. This isomorphism interacts with the $\mathfrak{T}^{\textrm{ext}}$ action in the following manner: for any point $x \in \mathfrak{F}_{n,r,I}(R)$, $\lambda \ast a(x) = a(\lambda^{-1} \ast x)$.
\end{prop}

\begin{proof}
Define $a$ by sending a point $(P, \omega) \mapsto (P, \omega^{\vee})$. Then it is easy to check the rest of the claims.
\end{proof}

\begin{prop}
There is a natural isomorphism $\w{0} \simeq \w{0, \textrm{old}}$.
\end{prop}

\begin{proof}
Let $\nu \colon \V{\Sh{O}_L}(\Omega_{\Sh{A}}, s) \to \Xra{r}$ be the projection. Then the isomorphism induced by $a$, \[\nu_*\Sh{O}_{\V{\Sh{O}_L}(\Omega_{\Sh{A}},s)} \simeq \nu'_*\Sh{O}_{\mathfrak{F}_{n,r,I}}\]
induces an isomorphism $\w{0} \simeq \w{0,\textrm{old}}$ due to Proposition \ref{P239}.
\end{proof}

Here we recall an important result about the surjectivity of the specialization map for cusp forms. 

Let $\Xra{r}^* \to \mathfrak{W}^0_{\alpha,I}$ be the blow-up spaces constructed exactly as $\Xra{r}$ but now starting from the minimal compactification $M^*(\mu_N, \mathfrak{c})$. These are formal models for overconvergent neighbourhoods of $M^*(\mu_N, \mathfrak{c})$. Let $\mathfrak{M}^*_{r, \alpha,I} := \Xra{r}^*\times_{\mathfrak{W}^0_{\alpha,I}} \mathfrak{W}_{\alpha,I}$. There is a natural map $f \colon \bar{\mathfrak{M}}_{r,\alpha,I} \to \mathfrak{M}^*_{r,\alpha,I}$ induced by the projection $\bar{M}(\mu_N, \mathfrak{c}) \to M^*(\mu_N, \mathfrak{c})$. Let $\bar{\Sh{M}}_{r,\alpha,I}, \Sh{M}^*_{r,\alpha,I}$ be their adic generic fibre. Recall that $D$ was the boundary divisor of $\bar{M}(\mu_N, \mathfrak{c})$. By an abuse of notation we denote by $D$ its inverse image in $\bar{\mathfrak{M}}_{r,\alpha,I}$.

\begin{theorem}\label{T2302}
We have $R^if_*\w{}(-D) = 0$ for all $i > 0$. Let $g \colon \bar{\Sh{M}}_{r,\alpha,I} \to \Sh{W}_{\alpha,I}$ be the projection to the weight space. Then for any weight $\kappa \in \Sh{W}_{\alpha,I}$, 
\[
\kappa^*g_*\w{}(-D)[1/\alpha] = H^0(\bar{\Sh{M}}_{r,\alpha,I}, \kappa^*\w{}(-D)[1/\alpha])
\]
is the space of $r$-overconvergent Hilbert cuspforms of tame level $\mu_N$, $\mathfrak{c}$-polarization and weight $\kappa$.
\end{theorem}

\begin{proof}
The first part follows from \cite[Corollary 3.20]{Andreatta2016Hilbert}. For the second part we remark that the projection $g$ factors through $\bar{\Sh{M}}_{r,\alpha,I} \xrightarrow{f} \Sh{M}^*_{r,\alpha,I} \to \Sh{W}_{\alpha,I}$, and $\Sh{M}^*_{r,\alpha,I}$ is affinoid. Then the claim follows from the first part.
\end{proof}

\subsection{Overconvergent arithmetic Hilbert modular forms}

Recall that we defined the notion of arithmetic Hilbert modular forms in \S\ref{2S121}. These were Hilbert modular forms associated to the group $G = \Res_{L/\Q}\mathbf{GL}_{2,L}$. With $\Gamma = \Sh{O}_L^{\times,+}/U_N^2$, we saw that the quotient $\bar{M}(\mu_N, \mathfrak{c}) \to \bar{M}(\mu_N, \mathfrak{c})/\Gamma$ is finite \'{e}tale. Given a classical weight $(v, w) \in \mathfrak{W}^G$ the sheaf of arithmetic Hilbert modular forms of tame level $\mu_N$, $\mathfrak{c}$-polarization and weight $(v, w)$ with coefficients in $R$ was defined to be $\underline{\omega}^{(v,w)}_R := (p_*\omega^k_{\Sh{A},R})^{\Gamma}$. The definition of overconvergent arithmetic Hilbert modular forms is given in a similar manner. We follow \cite{andreatta2016adic}.

Let $\Mbar{r}{G} := (\Mbar{r}{} \times_{\mathfrak{W}} \mathfrak{W}^G)/\Gamma$. By Proposition \ref{P211} the quotient map $p \colon \Mbar{r}{} \times_{\mathfrak{W}} \mathfrak{W}^G \to \Mbar{r}{G}$ is finite \'{e}tale. Consider the pullback of $\w{}$ along $f \colon \Mbar{r}{} \times_{\mathfrak{W}} \mathfrak{W}^G \to \Mbar{r}{}$. The action of $\Gamma$ on $\Mbar{r}{} \times_{\mathfrak{W}} \mathfrak{W}^G$ can be lifted to an action on $f^*\w{}$ as follows. Let $(v_{\alpha,I}, w_{\alpha,I}) \colon (\Sh{O}_L \otimes \Z_p)^{\times} \times \Z_p^{\times} \to (\Lambda_{\alpha,I} \hat{\otimes}_{\Lambda} \Lambda^G)^{\times} =: (\Lambda^G_{\alpha,I})^{\times}$ be the universal character. Viewing a section $g$ of $f^*\w{}$ by Koecher's principle as a rule that associates to any tuple $(A, \iota, \lambda, \psi, \omega)$ a value $g(A, \iota, \lambda, \psi, \omega) \in \Lambda^G_{\alpha,I}$, the action of $\epsilon$ is given by 
\[
(\epsilon \cdot g)(A, \iota, \lambda, \psi, \omega) = v_{\alpha,I}(\epsilon^{-1})g(A, \iota, \epsilon\lambda, \psi, \omega).
\]
\begin{defn}
The sheaf of $r$-overconvergent arithmetic Hilbert modular forms of tame level $\mu_N$, $\mathfrak{c}$-polarization and weight $k^G_{\alpha,I} := (v_{\alpha,I}, w_{\alpha,I})$ is defined to be 
\[
\w{G,\mathfrak{c}} := (p_*f^*\w{})^{\Gamma}.
\]
Similarly one defines the sheaf of $r$-overconvergent arithmetic Hilbert cuspforms of tame level $\mu_N$, $\mathfrak{c}$-polarization and weight $k^G_{\alpha,I}$ to be 
\[
\w{G,\mathfrak{c}}(-D) := (p_*f^*\w{}(-D))^{\Gamma}.
\]
\end{defn}

We have a surjectivity result about the specialization map of arithmetic Hilbert cuspforms analogous to Theorem \ref{T2302}.

Let $\bar{\Sh{M}}^G_{r,\alpha,I}$ be the adic generic fibre of $\Mbar{r}{G}$. Let $g \colon \bar{\Sh{M}}^G_{r,\alpha,I} \to \Sh{W}^G_{\alpha,I}$ be the projection to the weight space, where $\Sh{W}^G_{\alpha,I} = \Spa(\Lambda^G_{\alpha,I}[1/\alpha], \Lambda^G_{\alpha,I})$. 

\begin{theorem}
For any weight $k^G \in \Sh{W}^G_{\alpha,I}$, $(k^G)^*\w{G,\mathfrak{c}}(-D)[1/\alpha]$ is the space of $r$-overconvergent arithmetic Hilbert cuspforms of tame level $\mu_N$, $\mathfrak{c}$-polarization and weight $k^G$.
\end{theorem}

\begin{proof}
We need to show that $R^ig_*\w{G,\mathfrak{c}}(-D)[1/\alpha] = 0$ for all $i > 0$. This follows from Theorem \ref{T2302} by noting that over the generic fibre applying the invariant functor $( \cdot )^{\Gamma}$ to a \v{C}ech resolution of $p_*f^*\w{}(-D)$ is exact since $( \cdot )^\Gamma$ is obtained by the application of the projector $e = \frac{1}{\#\Gamma} \sum_{\epsilon \in \Gamma}\epsilon$.
\end{proof}

Let $g_{\mathfrak{c}}$ denote the projection to the weight space $\mathfrak{W}^G_{\alpha,I}$ from the formal model $\Mbar{r}{G}$ corresponding to the moduli of abelian schemes with $\mathfrak{c}$-polarization. For $x \in L^{\times,+}$ coprime to $p$, consider the isomorphism $L_{(x\mathfrak{c}, \mathfrak{c})} \colon {g_{x\mathfrak{c}}}_*\w{G,x\mathfrak{c}} \xrightarrow{} {g_{\mathfrak{c}}}_*\w{G,\mathfrak{c}}$ given by
\[
L_{(x\mathfrak{c}, \mathfrak{c})}(f)(A, \iota, \lambda, \psi, \omega) = v_{\alpha,I}(x) f(A, \iota, x\lambda, \psi, \omega).
\]
The isomorphism depends only on the principal ideal $(x)$ and preserves cuspidality.

\begin{defn}\label{D358}
Define the sheaf of $r$-overconvergent arithmetic Hilbert modular forms of tame level $\mu_N$ and weight $k^G_{\alpha,I}$ to be
\[
\w{G} := \left(\bigoplus_{\mathfrak{c} \in \textrm{Frac}(L)^{(p)}} {g_{\mathfrak{c}}}_*\w{G,\mathfrak{c}} \right)/\left(L_{(x\mathfrak{c},\mathfrak{c})}(f) - f \right)_{x \in \textrm{Princ}(L)^{+,(p)}}.
\]
One defines similarly the subsheaf of $r$-overconvergent arithmetic Hilbert cuspforms, which we denote by $\w{G}(-D)$. (Note the $D$ in the notation does not have anything to do any boundary divisor, but we choose this notation to stay consistent with our previous notation when the polarization module was fixed and we were working over a fixed toroidal compactification.)
\end{defn}

We can analogously define arithmetic overconvergent de Rham sheaves. With notation as before, let $\gamma_{\epsilon} \colon \Mbar{r}{} \times_{\mathfrak{W}} \mathfrak{W}^G \to \Mbar{r}{} \times_{\mathfrak{W}} \mathfrak{W}^G$ be given by sending $(A, \iota, \lambda, \psi) \mapsto (A, \iota, \epsilon\lambda, \psi)$ for any $\epsilon \in \Gamma$. The identity $[1] \colon A \to A$ induces an isomorphism $[1]^* \colon f^*\W{} \to {\gamma_{\epsilon}}_*f^*\W{}$. Define the action of $\epsilon$ on $p_*f^*\W{}$ to be $\epsilon \cdot s = v_{\alpha, I}(\epsilon)[1]^*s$. 

\begin{defn}
    The sheaf of $r$-overconvergent arithmetic de Rham classes of tame level $\mu_N$, $\mathfrak{c}$-polarization and weight $k^G_{\alpha,I} = (v_{\alpha,I}, w_{\alpha, I})$ is defined to be 
    \[
    \W{G,\mathfrak{c}} := (p_*f^*\W{})^{\Gamma}.
    \]
    Denote by $\W{G,\mathfrak{c}}(-D)$ the subsheaf of cuspidal classes.
\end{defn}

For $x \in L^{\times, +}$ coprime to $p$, let $\gamma_x \colon \Mbar{r}{G, \mathfrak{c}} \to \Mbar{r}{G, x\mathfrak{c}}$ be given by sending $(A, \iota, \lambda\Sh{O}_L^{\times, +}, \psi) \mapsto (A, \iota, x\lambda\Sh{O}_L^{\times, +}, \psi)$. The identity $[1]\colon A \to A$ induces an isomorphism $[1]^* \colon \W{G,x\mathfrak{c}} \to {\gamma_x}_*\W{G, \mathfrak{c}}$. Define the map $L_{(x\mathfrak{c}, \mathfrak{c})} \colon {g_{x\mathfrak{c}}}_*\W{G,x\mathfrak{c}} \to {g_{\mathfrak{c}}}_*\W{G, \mathfrak{c}}$ by $L_{(x\mathfrak{c},\mathfrak{c}}(s) = v_{\alpha, I}(x^{-1})[1]^*s$. 

\begin{defn}\label{D459}
    Define the sheaf of $r$-overconvergent arithmetic de Rham classes of tame level $\mu_N$ and weight $k^G_{\alpha,I} = (v_{\alpha,I}, w_{\alpha,I})$ to be 
    \[
    \W{G} := \left(\bigoplus_{\mathfrak{c} \in \textrm{Frac}(L)^{(p)}} {g_{\mathfrak{c}}}_*\W{G,\mathfrak{c}} \right)/\left(L_{(x\mathfrak{c},\mathfrak{c})}(s) - s \right)_{x \in \textrm{Princ}(L)^{+,(p)}}.
    \]
    Denote by $\W{G}(-D)$ the subsheaf of cuspidal classes.
\end{defn}

%% file: Hilbert4.tex
\section{\texorpdfstring{$p$}{p}-adic iteration of the Gauss--Manin connection}\label{2S4}

In this section we will define iteration of the Gauss--Manin connection for analytic weights. For simplicity of notation we will ignore the log poles at the cusps. The first main result of this section is the following.

\begin{theorem*}
There is a filtration $\{\Fil{i}\}_{i\geq 0}$ on $\W{0}$ such that the graded pieces over the generic fibre are $\Gr{n}\W{0}[1/\alpha] \simeq \w{0}\otimes \Sym^n\omega_{\Sh{A}}^{-\otimes 2}[1/\alpha]$. There is a connection 
\[\nabla_k \colon \W{0}[1/\alpha] \to \W{0} \hat{\otimes} \Omega^1_{\Xra{r}/\Lambda^0_{\alpha,I}}[1/\alpha]\]
induced by the Gauss--Manin connection on $\mathrm{H}_{\Sh{A}}$, which satisfies Griffiths' transversality with respect to the filtration $\Fil{i}$ above. Moreover, it also induces a connection 
\[
\W{}[1/\alpha] \to \W{}\hat{\otimes}\Omega^1_{\bar{\mathfrak{M}}_{r,\alpha,I}/\Lambda_{\alpha,I}}[1/\alpha]
\]
that satisfies Griffiths' transversality with respect to the filtration defined by tensoring $\Fil{i}$ with $\w{\chi}$.
\end{theorem*}

\subsection{The Gauss--Manin connection on \texorpdfstring{$\mathrm{H}^{\sharp}_{\Sh{A}}$}{HA}}

Let $\Sh{IG}'_{n, r, I} \to \Sh{IG}_{n, r, I}$ be the analytic adic space classifying trivializations $(\Sh{O}_L/p^n\Sh{O}_L)^2 \xrightarrow{\sim} \Sh{A}[p^n]^{\vee}$ compatible with the trivializations $\Sh{O}_L/p^n \Sh{O}_L \xrightarrow{\sim} H_n^{\vee}$. Let $\mathfrak{IG}'_{n,r,I} \to \Ig{n}{r}$ be the normalization. 

\begin{prop}\label{P401}
The Gauss--Manin connection on $\mathrm{H}_{\Sh{A}}$ over $\Ig{n}{r}'$ restricts to a connection 
\[
\nabla \colon \mathrm{H}^{\sharp}_{\Sh{A}} \to \mathrm{H}^{\sharp}_{\Sh{A}} \hat{\otimes} \Omega^1_{\Ig{n}{r}'/\Lambda^0_{\alpha,I}}
\]
such that $(\nabla \textrm{ mod }\beta_n)(s) = 0$ and $(\nabla \textrm{ mod } \eta)(\Sh{Q}) \subset \Sh{Q} \otimes \Omega^1_{\Ig{n}{r}'/\Lambda^0_{\alpha,I}}/(\eta)$.
\end{prop}

\begin{proof}
We have a commutative diagram as follows coming from the functoriality of the $\dlog$ map. 
\[
\begin{tikzcd}
(\Sh{O}_L/p^n\Sh{O}_L)^2 \arrow[d, "\simeq"] \arrow[r, "\dlog"] & \omega_{\mu_{p^n}^{2}} \otimes \mathfrak{d}^{-1} \arrow[d] \\
{\Sh{A}[p^n]^{\vee}(\Ig{n}{r}')} \arrow[r, "\dlog"]   & \omega_{\Sh{A}}/p^n     
\end{tikzcd}
\]
The bottom arrow $\dlog_{\Sh{A}[p^n]^{\vee}}$ composed with the projection on to $\omega_{\Sh{A}}/\beta_n \simeq \omega_{H_n}/\beta_n$ factors through $\Sh{A}[p^n]^{\vee} \to H_n^{\vee} \xrightarrow{\dlog} \omega_{H_n}$. The connection on $\mathrm{H}_{\Sh{A}}$ modulo $p^n$ is the connection $\bar{\nabla}$ on the invariant differentials of the universal vector extension of $\Sh{A}[p^n]^{\vee}$. Since $\mu_{p^n}$ is isotrivial, the functoriality of the Gauss--Manin connection and the commutativity of the diagram shows that $\bar{\nabla}(\dlog(P^{\textrm{univ}}_n)) = 0$. This shows that $\nabla(\Omega_{\Sh{A}}) = 0 \textrm{ mod } \beta_n\mathrm{H}_{\Sh{A}} \hat{\otimes} \Omega^1_{\Ig{n}{r}'/\Lambda^0_{\alpha,I}}$. In particular $\nabla(\Omega_{\Sh{A}}) \subset \mathrm{H}^{\sharp}_{\Sh{A}} \hat{\otimes} \Omega^1_{\Ig{n}{r}'/\Lambda^0_{\alpha,I}}$.

On the other hand since $\mathrm{H}^{\sharp}_{\Sh{A}} = \Omega_{\Sh{A}} + \underline{\xi}\widetilde{HW}\cdot \mathrm{H}_{\Sh{A}}$ we are left to show that $\nabla$ maps $\underline{\xi}\widetilde{HW}\cdot\mathrm{H}_{\Sh{A}}$ to $\mathrm{H}^{\sharp}_{\Sh{A}}$. Since the Gauss--Manin connection is functorial, it commutes with the splitting of $\mathrm{H}_{\Sh{A}}$ into the $g$ different components. Recalling the notation from Corollary \ref{C233} we need to show that 
\[
\nabla\left(\underline{\xi}(i_j)\widetilde{HW}(i_j)\cdot\mathrm{H}_{\Sh{A}}(i_j)\right) \subset \mathrm{H}^{\sharp}_{\Sh{A}}(i_j) \hat{\otimes} \Omega^1_{\Ig{n}{r}'/\Lambda^0_{\alpha,I}}.
\]
To show this note that by Lemma \ref{L403}, $\underline{\xi}(i_j)\widetilde{HW}(i_j) = \tilde{F}^*(\underline{\xi}(i_{j-1}))$ is a $p$-th power at all height 1 localizations. Thus $\nabla$ maps it to $\mathrm{H}^{\sharp}_{\Sh{A}} + p\mathrm{H}_{\Sh{A}}$. Since $p \in \underline{\xi}(i_j)\widetilde{HW}(i_j)$ for all $i_j$ we conclude. This proves the first two claims of the proposition. Now we show that $(\nabla \textrm{ mod } \eta)(\Sh{Q}) \subset \Sh{Q} \otimes \Omega^1_{\Ig{n}{r}'/\Lambda^0_{\alpha,I}}/(\eta)$.

Let $\Spf{R_0} \subset \Xra{r}$ be an open  such that $\mathrm{H}_{\Sh{A}}$ is trivialized over $\Spf{R}$ as $\Sh{O}_L \otimes R_0$-modules. Let $\Spf{R} \subset \Ig{n}{r}$ and $\Spf{R}' \subset \Ig{n}{r}'$ be its inverse image in $\Ig{n}{r}$ and $\Ig{n}{r}'$ respectively. Assume also that $\Hs{A}$ is trivialized over $\Spf{R}$. Pick $\Sh{O}_L \otimes R_0$-basis $\omega, \zeta$ of $\mathrm{H}_{\Sh{A}}$ such that $e := \xi\omega$ is a lift of $s$ for some local generator $\xi$ of $\underline{\xi}$ and such that $f := Ce + (\xi HW)\zeta$ is a lift of a generator of $\Sh{Q}$ for some local generator $HW$ of $\widetilde{HW}$ and some $C \in \Sh{O}_L \otimes R$ as in the proof of Proposition \ref{P232}. Assume also that the image of $\zeta$ in $\omega_{\Sh{A}}^{\vee}$ is the dual of $\omega$. 

Let $\omega_{\sigma}, \zeta_{\sigma}$ be the $\sigma$-components of $\omega$ and $\zeta$. Let $\Theta_{\sigma} = KS(\omega_{\sigma}, \zeta_{\sigma})$ be the Kodaira--Spencer class in $\Omega^1_{\Xra{r}/\Lambda^0_{\alpha, I}}$ corresponding to the image of $\omega_{\sigma}^{\otimes 2}$ under the Kodaira--Spencer isomorphism
\[
KS \colon \omega_{\Sh{A}}^{\otimes 2} \xrightarrow{\sim} \Omega^1_{\Xra{r}/\Lambda^0_{\alpha, I}}.
\]
The $\Theta_{\sigma}$ thus form a basis for $\Omega^1_{\Xra{r}/\Lambda^0_{\alpha, I}}$. Suppose that
\begin{align}
    \nabla(\omega_{\sigma}) &= \sum_{\tau} \omega_{\sigma} \otimes \alpha^{\sigma}_{\tau} \Theta_{\tau} + \zeta_{\sigma} \otimes \Theta_{\sigma} \\
    \nabla(\zeta_{\sigma}) &= \sum_{\tau} \omega_{\sigma} \otimes \beta^{\sigma}_{\tau} \Theta_{\tau} + \sum_{\tau} \zeta_{\sigma} \otimes \gamma^{\sigma}_{\tau} \Theta_{\tau}
\end{align}
for $\alpha^{\sigma}_{\tau}, \beta^{\sigma}_{\tau}, \gamma^{\sigma}_{\tau} \in \Sh{O}_{\Xra{r}}$.
Therefore we have
\begin{align}\label{eq:9}
\begin{split}
    \nabla(e_{\sigma}) &= \nabla(\xi({\sigma})\cdot \omega_{\sigma}) \\
    &= \xi(\sigma) \sum_{\tau} \omega_{\sigma} \otimes \alpha^{\sigma}_{\tau} \Theta_{\tau} + \xi(\sigma)\cdot \zeta_{\sigma}\otimes \Theta_{\sigma} + \xi(\sigma)\cdot \omega_{\sigma} \otimes \dlog\xi(\sigma) \\
    &= \omega_{\sigma} \otimes \left(\xi(\sigma)\sum_{\tau}\alpha^{\sigma}_{\tau}\Theta_{\tau} + \mathrm{d}\xi(\sigma) \right) + \zeta_{\sigma} \otimes \xi(\sigma)\Theta_{\sigma} \\
    &= e_{\sigma} \otimes \left(\sum_{\tau}  \alpha^{\sigma}_{\tau}\Theta_{\tau} + \dlog \xi(\sigma) - \frac{C_{\sigma}\Theta_{\sigma}}{HW(\sigma)}\right) + f_{\sigma} \otimes \frac{\Theta_{\sigma}}{HW(\sigma)}
\end{split}
\end{align}
\begin{align}\label{eq:10}
\begin{split}
    \nabla(f_{\sigma}) &= \nabla(C_{\sigma} e_{\sigma} + \xi(\sigma)HW(\sigma)\cdot\zeta_{\sigma}) \\
    &= \nabla(C_{\sigma}e_{\sigma}) +  \xi(\sigma)HW(\sigma) \sum_{\tau} \omega_{\sigma} \otimes \beta^{\sigma}_{\tau}\Theta^{\tau} + \xi(\sigma)HW(\sigma) \sum_{\tau} \zeta_{\sigma} \otimes \gamma^{\sigma}_{\tau} \Theta_{\tau} \\ &+ \xi(\sigma)HW(\sigma)\cdot \zeta_{\sigma} \otimes \dlog (\xi(\sigma)HW(\sigma)) \\
    &= \nabla(C_{\sigma}e_{\sigma}) + e_{\sigma} \otimes \left(\sum_{\tau} HW(\sigma) \beta^{\sigma}_{\tau} \Theta_{\tau} - C_{\sigma}\sum_{\tau}\gamma^{\sigma}_{\tau}\Theta_{\tau} - \dlog(\xi(\sigma)HW(\sigma))  \right) \\
    &+ f_{\sigma} \otimes \left( \sum_{\tau} \gamma^{\sigma}_{\tau} \Theta_{\tau} + \dlog(\xi(\sigma)HW(\sigma))\right)
\end{split}
\end{align}

Now the proof of the previous part of the proposition shows that over $\Ig{n}{r}'$, $\nabla(\Omega_{\Sh{A}}) \subset \beta_n\mathrm{H}_{\Sh{A}} \otimes \Omega^1_{\Ig{n}{r}'/\Lambda^0_{\alpha,I}}$. Then the third equality of (\ref{eq:9}) implies that $\Theta_{\sigma} \in \beta_n\Omega^1_{\Ig{n}{r}'/\Lambda^0_{\alpha,I}}$ for all $\sigma$. In other words the image of $\Omega^1_{\Xra{r}/\Lambda^0_{\alpha,I}} \to \Omega^1_{\Ig{n}{r}'/\Lambda^0_{\alpha,I}}$ is contained in $\beta_n\Omega^1_{\Ig{n}{r}'/\Lambda^0_{\alpha,I}}$. Moreover the same is true of $\mathrm{d}\xi(\sigma)$. This implies together with the explicit formula of (\ref{eq:10}) that $\nabla(f_{\sigma}) \equiv 0 \textrm{ mod } (\eta)$, proving the final part of the proposition.
\end{proof}

\subsection{The Gauss--Manin connection on \texorpdfstring{$\W{0}$}{W}}\label{2S42}

Recall that the universal character $k = k^0_{\alpha,I}$ is analytic on $1 + p^{n-1}(\Sh{O}_L \otimes \Z_p)$. In particular there are $u_{\sigma} \in p^{1-n}\Lambda^0_{\alpha,I}$ such that $k_{\sigma}(t) = \exp(u_{\sigma}\log(t))$ for all $t \in 1 + \beta_n \Ga$. In this section we will define the connection on $\W{0}$. We would like to have Griffith's transversality for some filtration on $\W{0}$. But the filtration given in Theorem \ref{T2301}, i.e. the filtration given by lexicographic ordering for a choice of numbering of the set $\Sigma$, doesn't satisfy this. Thus as promised before, here we define the Hodge filtration on $\W{0}$.

\begin{lemma}\label{L2401}
Let $\rho' \colon \V{\Sh{O}_L}(\Hs{A}, s, \Sh{Q}) \to \Ig{n}{r}$ and $h_n \colon \Ig{n}{r} \to \Xra{r}$ be the projections. Then the filtration on $\rho'_*\Sh{O}_{\V{\Sh{O}_L}(\Hs{A}, s, \Sh{Q})}$ defined on local coordinates $\Spf{R} \subset \Ig{n}{r}$ by 
\[
\Fil{i}\left(\rho'_*\Sh{O}_{\V{\Sh{O}_L}(\mathrm{H}^{\sharp}_{\Sh{A}}, s, \Sh{Q})}(\Spf{R})\right) = \Fil{i}R\langle \{Z_{\sigma}, W_{\sigma}\}_{\sigma \in \Sigma}\rangle := \bigoplus_{j=0}^i R\langle \{Z_{\sigma}\}_{\sigma \in \Sigma}\rangle \otimes_R \Sym^j{R[W_{\sigma}]_{\sigma \in \Sigma}}
\]
is well-defined. Moreover, $(h_n)_*\Fil{i}$ is stable under the action of $\mathfrak{T}^{\textrm{ext}}$ for all $i$. In particular it induces a filtration on $\W{0}$, by defining $\Fil{i}\W{0} := \Fil{i}\left(\rho_*\Sh{O}_{\V{\Sh{O}_L}(\mathrm{H}^{\sharp}_{\Sh{A}}, s, \Sh{Q})}\right)[k]$, where $\rho = h_n \circ \rho'$. This is defined to be the Hodge filtration. The graded pieces over the generic fibre are $\Gr{i}\W{0}[1/\alpha] \simeq \w{0}\otimes \Sym^i\omega_{\Sh{A}}^{-\otimes 2}[1/\alpha] \simeq \w{0}\otimes \Sym^i(\oplus_{\sigma} \omega_{\Sh{A}}(-2\sigma))[1/\alpha]$.
\end{lemma}

\begin{proof}
We note that $\oplus_{j=0}^i \Sym^j R[W_{\sigma}]_{\sigma \in \Sigma}$ contains all polynomials in $\{W_{\sigma}\}_{\sigma \in \Sigma}$ of degree $\leq i$. In particular, choosing an ordering $\Xi \colon \Sigma \simeq \{1, \dots, g\}$, there is a greatest element in $\Sym^i R[W_{\sigma}]_{\sigma \in \Sigma}$ corresponding to the multi-index $(i,0, \dots, 0)$. Let $i_0 := \Xi^{-1}(i, 0, \dots, 0)$. Denoting by $\Fil{i}'$ the lexicographic ordering of Theorem \ref{T2301}, we then have $\Fil{i} = \Fil{i_0}'$. The lemma follows.
\end{proof}

\noindent\textbf{Convention:} Henceforth, unless otherwise stated $\{\Fil{i}\}_{i \geq 0}$ will denote the Hodge filtration of Lemma \ref{L2401}, and \textbf{not} the filtration induced by a lexicographic ordering on $\Sigma$. Also we will denote by $\Fil{i}$ the filtration induced on $\W{}$ by tensoring the above filtration with $\w{\chi}$.

Let $\Sh{P}^{(1)}_{\Ig{n}{r}'/\Lambda^0_{\alpha, I}}$ be the first infinitesimal neighborhood of the closed subscheme of $\Ig{n}{r}' \times_{\mathfrak{W}^0_{\alpha,I}} \Ig{n}{r}'$ defined by the diagonal embedding $\Delta \colon \Ig{n}{r}' \xhookrightarrow{} \Ig{n}{r}' \times_{\Spf{\Lambda^0_{\alpha, I}}} \Ig{n}{r}'$. Let $p_1, p_2$ be the first and second projections $\Sh{P}^{(1)}_{\Ig{n}{r}'/\Lambda^0_{\alpha, I}} \to \Ig{n}{r}'$. Then using Grothendieck's formalism of connections \cite[\S 2]{BerthelotOgus+2015}, we get an $\Sh{O}_L \otimes {\Sh{O}}_{\Sh{P}^{(1)}_{\Ig{n}{r}'/\Lambda^0_{\alpha,I}}}$-linear isomorphism $\epsilon^{\sharp} \colon p^*_2\Hs{A} \xrightarrow{\sim} p_1^*\Hs{A}$ associated to the connection $\nabla \colon \Hs{A} \to \Hs{A} \otimes \Omega^1_{\Ig{n}{r}'/\Lambda^0_{\alpha,I}}$. This $\epsilon^{\sharp}$ is characterized by the properties that $\Delta^*\epsilon^{\sharp} = \textrm{id}$,  $\nabla(x) = \epsilon^{\sharp}(1 \otimes x) - x \otimes 1$ and it satisfies a suitable cocycle condition with respect to the three possible pullbacks of $\epsilon^{\sharp}$ to $\Ig{n}{r}' \times_{\mathfrak{W}^0_{\alpha,I}} \Ig{n}{r}' \times_{\mathfrak{W}^0_{\alpha,I}} \Ig{n}{r}'$.

Let $\tilde{\rho} \colon \vartheta^*\V{\Sh{O}_L}(\Hs{A}, s, \Sh{Q}) \to \Ig{n}{r}'$ be the pullback of $\rho' \colon \V{\Sh{O}_L}(\Hs{A},s, \Sh{Q}) \to \Ig{n}{r}$ to $\vartheta \colon \Ig{n}{r}' \to \Ig{n}{r}$. Sometimes we will drop the notation $\vartheta$ for simplicity. 

\begin{lemma}\label{L2402}
The connection $\nabla \colon \Hs{A} \to \Hs{A} \hat{\otimes} \Omega^1_{\Ig{n}{r}'/\Lambda^0_{\alpha,I}}$ induces an isomorphism associated to a connection (in the sense of Grothendieck) on $\tilde{\rho}_*\Sh{O}_{\V{\Sh{O}_L}(\Hs{A}, s, \Sh{Q})}$.
\[
\epsilon_0^{\sharp}\colon p_2^*\tilde{\rho}_*\Sh{O}_{\V{\Sh{O}_L}(\Hs{A}, s, \Sh{Q})} \xrightarrow{\sim} p_1^*\tilde{\rho}_*\Sh{O}_{\V{\Sh{O}_L}(\Hs{A}, s, \Sh{Q})}
\]
\end{lemma}

\begin{proof}
The isomorphism $\epsilon^{\sharp} \colon p_2^*\Hs{A} \xrightarrow{\sim} p_1^*\Hs{A}$ splits by $\Sh{O}_L$-linearity into  $\epsilon^{\sharp}_{\sigma} \colon p_2^*\Hs{A}(\sigma) \xrightarrow{\sim} p_1^*\Hs{A}(\sigma)$ for all $\sigma \in \Sigma$. Let $\tilde{\rho}_{\sigma} \colon \V{}(\Hs{A}(\sigma), s_{\sigma}, \Sh{Q}(\sigma)) \to \Ig{n}{r}'$ be the $\sigma$-component of $\tilde{\rho}$. Each $\epsilon^{\sharp}_{\sigma}$ induces a connection $\epsilon^{\sharp}_{\sigma,0} \colon p_2^*{(\tilde{\rho}_{\sigma})}_*\Sh{O}_{\V{}(\Hs{A}(\sigma), s_{\sigma}, \Sh{Q}(\sigma))} \xrightarrow{\sim} p_1^*{(\tilde{\rho}_{\sigma})}_*\Sh{O}_{\V{}(\Hs{A}(\sigma), s_{\sigma}, \Sh{Q}(\sigma))}$ by \cite[\S2.4]{andreatta2021triple}. Then $\epsilon^{\sharp}_0$ is defined by the tensor product $\otimes \epsilon^{\sharp}_{\sigma,0}$.
\end{proof}

\begin{lemma}\label{L2403}
The action of $(\Sh{O}_L \otimes \Z_p)^{\times}$ on $\V{\Sh{O}_L}(\Hs{A},s,\Sh{Q})$ over $\Xra{r}$ as defined in \S\ref{2S321} can be lifted to an action on $\vartheta^*\V{\Sh{O}_L}(\Hs{A},s,\Sh{Q})$ over $\Xra{r}$ such that the induced action commutes with $\epsilon^{\sharp}_0$.
\end{lemma}

\begin{proof}
The map $\Sh{IG}'_{n,r,I} \to \xra{r}$ is a torsor for the group 
\[
\begin{pmatrix}
(\Sh{O}_L/p^n\Sh{O}_L)^{\times} & \mu_{p^n} \otimes \mathfrak{d}^{-1} \\
0 & (\Sh{O}_L/p^n\Sh{O}_L)^{\times}
\end{pmatrix}.
\]
Then $(\Sh{O}_L \otimes \Z_p)^{\times}$ acts on $\Ig{n}{r}'$ through the quotient
\[
(\Sh{O}_L \otimes \Z_p)^{\times} \to \begin{pmatrix}
1 & 0 \\
0 & (\Sh{O}_L/p^n\Sh{O}_L)^{\times}
\end{pmatrix}
\]
and this action lifts the action on $\Ig{n}{r}$. For $\lambda \in (\Sh{O}_L \otimes \Z_p)^{\times}$ we get an isomorphism $[\lambda] \colon \Ig{n}{r}' \to \Ig{n}{r}'$ over $\Xra{r}$, that induces an isomorphism $\gamma_{\lambda} \colon \Hs{A} \to \Hs{A}$ sending the marked section $s \mapsto \bar{\lambda}^{-1}s$ and the marked subspace $\Sh{Q}$ to itself. Since the connection on $\Hs{A}$ is induced by the Gauss--Manin connection on $\mathrm{H}_{\Sh{A}}$, by functoriality of the Gauss--Manin connection, $\nabla$ commutes with $\gamma_{\lambda}$. The last claim follows by noticing that the action on $\vartheta^*\V{\Sh{O}_L}(\Hs{A},s,\Sh{Q})$ is induced by the isomorphism $\Hs{A} \xrightarrow{\gamma_{\lambda}} \Hs{A} \xrightarrow{\times \lambda} \Hs{A}$ which obviously commutes with $\nabla$.
\end{proof}

\begin{theorem}\label{T401}
There is an integrable connection on $\W{0}$,
\[
\nabla_k \colon \W{0} \to \W{0} \hat{\otimes} \Omega^1_{\Xra{r}/\Lambda^0_{\alpha, I}}[1/\alpha]
\]
for which the filtration on $\W{0}$ defined in Lemma \ref{L2401} satisfies Griffith's transversality.
\end{theorem}

\begin{proof}
We use the notation of the proof of Proposition \ref{P401}. Recall $\Spf{R}' \subset \Ig{n}{r}'$ was an open that was the inverse image of an open $\Spf{R}_0 \subset \Xra{r}$ that trivializes $\mathrm{H}_{\Sh{A}}$, and such that over $\Spf{R}'$, $\Hs{A}$ is trivial with $\Sh{O}_L \otimes R'$-basis $\{e, f\}$ adapted to the marked section $s$  and modified unit root subspace $\Sh{Q}$. Let $I(\Delta) = \ker{(R' \hat{\otimes}_{\Lambda^0_{\alpha,I}} R' \xrightarrow{\textrm{mult}} R')}$, and let $R^{(1)} = R' \hat{\otimes}_{\Lambda^0_{\alpha,I}} R'/I(\Delta)^2$. 

Then in terms of the basis $\{e, f\}$, $\epsilon^{\sharp}$ is given by a matrix 
\[
A = \begin{pmatrix}
a && b \\
c && d
\end{pmatrix} \in \mathrm{GL}_2\left(\Sh{O}_L \otimes R^{(1)}\right).
\]
Decomposing into components, we get matrices 
\[
A_{\sigma} = \begin{pmatrix}
a_{\sigma} && b_{\sigma} \\
c_{\sigma} && d_{\sigma}
\end{pmatrix} \in \mathrm{GL}_2\left(R^{(1)}\right)
\]
with respect to the basis $e_{\sigma}, f_{\sigma}$ of $\mathrm{H}^{\sharp}_{\Sh{A}}(\sigma)$ for each $\sigma$.

Since $\Delta^*(\epsilon^{\sharp}) = \textrm{id}$, we have that $a_{\sigma} = 1 + a^0_{\sigma}$ and $d_{\sigma} = 1 + d^0_{\sigma}$ with $a^0_{\sigma}, d^0_{\sigma}, b_{\sigma}, c_{\sigma} \in I(\Delta)$ for all $\sigma$. Moreover, the squares of $a^0_{\sigma}, d^0_{\sigma}, b_{\sigma}, c_{\sigma}$ are all $0$ in $R^{(1)}$.

Comparing the expression of $\nabla$ in terms of $a, b, c, d$ on the one hand and that in (\ref{eq:9}) and (\ref{eq:10}) on the other, we see that $c_{\sigma}HW(\sigma)$ is the Kodaira--Spencer class $\Theta_{\sigma}$.

By Lemma \ref{L2402} there is an isomorphism $\epsilon^{\sharp}_0 \colon p_2^*\tilde{\rho}_*\Sh{O}_{\V{\Sh{O}_L}(\Hs{A}, s, \Sh{Q})} \xrightarrow{\sim} p_1^*\tilde{\rho}_*\Sh{O}_{\V{\Sh{O}_L}(\Hs{A}, s, \Sh{Q})}$ induced by $\epsilon^{\sharp} \colon p_2^*\Hs{A} \simeq p_1^*\Hs{A}$. We show that $\epsilon^{\sharp}_0$ restricts to a connection on $\mathbb{W}'_{k,\alpha,I}$. We show this on local coordinates. So recalling the local description of $\mathbb{W}'_{k,\alpha,I}$ from Corollary \ref{C302}, we have ${\mathbb{W}'_{k,\alpha,I}}_{|\Spf{R'}} = R'\langle \{ V_{\sigma} \}\rangle k(1 + \beta_n Z)$, where we recall $V_{\sigma} = \frac{W_{\sigma}}{1+\beta_nZ_{\sigma}}$ and $k(1+\beta_nZ)$ was the notation for $\prod_{\sigma} k_{\sigma}(1+\beta_n Z_{\sigma})$. Thus $\epsilon^{\sharp}_0$ is described by its action on $V_{\sigma}$ and $1 + \beta_n Z_{\sigma}$. We have
\[
\epsilon^{\sharp}_0(V_{\sigma}) = \eta^{-1}(b_{\sigma} + \eta d_{\sigma}V_{\sigma})(a_{\sigma} + \eta c_{\sigma}V_{\sigma})^{-1} \quad \epsilon^{\sharp}_0(1 + \beta_n Z_{\sigma}) = (a_{\sigma} + \eta c_{\sigma} V_{\sigma})(1 +\beta_n Z_{\sigma}).
\]
From this one can deduce the following formula for $\nabla_k(x) = \epsilon^{\sharp}_0(1 \otimes x) - x \otimes 1$.
\begin{align}\label{eq:11}
\begin{split}
\nabla_k\left(\prod_{\sigma} V^{i_{\sigma}}_{\sigma} \cdot k(1 + \beta_n Z)\right) = &\prod_{\sigma} V_{\sigma}^{i_{\sigma}}\Bigg(\sum_{\sigma} i_{\sigma}V_{\sigma}^{-1} \otimes b_{\sigma}\eta^{-1} + \sum_{\sigma} (u_{\sigma} - i_{\sigma})\otimes a^0_{\sigma} \\ &+ \sum_{\sigma} i_{\sigma} \otimes d^0_{\sigma} + \sum_{\sigma} (u_{\sigma} - i_{\sigma})V_{\sigma} \otimes \eta c_{\sigma}\Bigg) \big(k(1 + \beta_nZ) \otimes 1\big)
\end{split}
\end{align}
This connection descends to $\Ig{n}{r}$ after inverting $\alpha$, by the formula above. Then we descend $\nabla_k$ to $\Xra{r}$ by taking $k$-invariants for the $(\Sh{O}_L \otimes \Z_p)^{\times}$ action using Lemma \ref{L2403} and noting $\Ig{n}{r} \to \Xra{r}$ is generically \'{e}tale.
\end{proof}

\begin{cor}
The connection $\nabla_k \colon \W{0} \to \W{0} \hat{\otimes} \Omega^1_{\Xra{r}/\Lambda^0_{\alpha,I}}[1/\alpha]$ induces a connection which we still denote by $\nabla_k$
\[
\nabla_k \colon \W{} \to \W{} \hat{\otimes} \Omega^1_{\bar{\mathfrak{M}}_{r,\alpha,I}/\Lambda_{\alpha,I}}[1/\alpha]
\]
that satisfies Griffiths transverality with respect to the filtration on $\W{}$ defined by tensoring $\Fil{i}$ of Lemma \ref{L2401} with $\w{\chi}$.
\end{cor}

\begin{proof}
Recall $\w{\chi}$ was defined as $\left((f_i)_*\Sh{O}_{\Ig{i}{r}}\otimes_{\Lambda^0_{\alpha,I}} \Lambda_{\alpha,I}\right)[\chi^{-1}]$ for $f_i \colon \Ig{i}{r} \to \Xra{r}$ the projection with $i = 1$ if $p \neq 2$ and $i = 2$ if $p = 2$. The universal derivation $(f_i)_*\Sh{O}_{\Ig{i}{r}} \otimes_{\Lambda^0_{\alpha,I}} \Lambda_{\alpha,I} \to (f_i)_*\Omega^1_{\Ig{i}{r}/\Lambda^0_{\alpha,I}} \otimes \Lambda_{\alpha,I}$ commutes with the action of $(\Sh{O}_L/q\Sh{O}_L)^{\times}$ and thus induces a connection $\w{\chi} \to \w{\chi} \hat{\otimes} \Omega^1_{\bar{\mathfrak{M}}_{r,\alpha,I}}[1/\alpha]$ by taking $\chi^{-1}$-invariants and upon inverting $\alpha$. This along with $\nabla_k$ defined on $\W{0}$ above induces the required connection on $\W{}$.
\end{proof}

The Kodaira--Spencer isomorphism $\omega_{\Sh{A}}^{\otimes 2} \xrightarrow{\sim} \Omega^1_{\Xra{r}/\Lambda^0_{\alpha,I}}$ induces a decomposition of $\Omega^1_{\Xra{r}/\Lambda^0_{\alpha,I}}$ corresponding to the decomposition $\omega_{\Sh{A}}^{\otimes 2} = \prod_{\sigma} \omega_{\Sh{A}}^{2\sigma}$. Here the tensor product is taken as $\Sh{O}_L \otimes \Sh{O}_{{\mathfrak{X}_{\alpha, I}}}$-modules. This induces an isomorphism $\Omega^1_{\Xra{r}}[1/\alpha] \simeq \prod_{\sigma}\omega_{\Sh{A}}^{2\sigma}[1/\alpha]$.

\begin{defn}\label{D402}
Define $\nabla_k(\sigma) \colon \W{0} \to \mathbb{W}^0_{k+2\sigma, \alpha,I}[1/\alpha]$ as the map obtained by composing $\nabla_k$ with the projection onto the $\sigma$ component of $\Omega^1_{\Xra{r}/\Lambda^0_{\alpha,I}}[1/\alpha] \simeq \prod_{\sigma} \omega_{\Sh{A}}^{2\sigma}$, followed by the natural map to $\mathbb{W}^0_{k + 2\sigma, \alpha,I}$.
\[
\nabla_k(\sigma) \colon \W{0} \xrightarrow{\nabla_k} \W{0} \hat{\otimes} \Omega^1_{\Xra{r}/\Lambda^0_{\alpha,I}}[1/\alpha] \xrightarrow{} \W{0} \hat{\otimes}\omega_{\Sh{A}}^{2\sigma}[1/\alpha] \xrightarrow{} \mathbb{W}^0_{k+2\sigma, \alpha,I}[1/\alpha]
\]
Similarly denote still by $\nabla_k(\sigma) \colon \W{} \to \mathbb{W}_{k+2\sigma, \alpha,I}[1/\alpha]$ the map obtained by twisting $\nabla_k(\sigma)$ as above with the connection on $\w{\chi}$ followed by the projection onto the $\sigma$-component under the Kodaira--Spencer isomorphism.
\end{defn}

\begin{cor}
The $\Sh{O}_{\xra{r}}$-linear map induced by the connection $\nabla_k(\sigma)$ on the graded piece \[\nabla_k(\sigma) \colon \Gr{n}\W{0}[1/\alpha] \to \Gr{n+1}\mathbb{W}^0_{k+2\sigma, \alpha, I}[1/\alpha]\]
sends an element $\omega^{k - 2\mathbf{i}} \in \w{0}\otimes \Sym^n\omega_{\Sh{A}}^{-\otimes 2}[1/\alpha]$ to $(u_{\sigma} - i_{\sigma})\omega^{k-2\mathbf{i}}$ for $\mathbf{i} = (i_{\tau})_{\tau \in \Sigma} \in \N^g$.
\end{cor}

\begin{proof}
Follows from (\ref{eq:11}) above.
\end{proof}

\subsection{\texorpdfstring{$\nabla_k$}{Connection} on \texorpdfstring{$q$}{q}-expansions}\label{S4.3}
For simplicity of notation, in this section we drop $\alpha,I$ from the notation $\W{0}$ and simply write $\mathbb{W}^0_k$. Also since we want to iterate the connection, and the connection $\nabla_k(\sigma)$ maps $\mathbb{W}^0_k[1/\alpha]$ to $\mathbb{W}^0_{k+2\sigma}[1/\alpha]$, in our notation we forget the dependency of the connection on the weight $k$, and simply write $\nabla(\sigma)$. Later we will need to compose $\nabla(\sigma)$ and $\nabla(\tau)$ for $\sigma \neq \tau \in \Sigma$. Lemma \ref{L404} below will show us that the order of composition does not matter.

In this section we will study the effect of $\nabla_k$ on $q$-expansions. We begin by reviewing the definition of Tate objects for Hilbert--Blumenthal abelian varieties following \cite{KatzNicholasM1978pLfC}.

Fix fractional ideals $\mathfrak{a}, \mathfrak{b}$ such that $\mathfrak{c} = \mathfrak{a}\mathfrak{b}^{-1}$. Let $S$ be a set of $g$ linearly independent $\Q$-linear forms $\ell_i \colon L \to \Q$, such that $\ell_i(x) > 0$ for all $x \gg 0$, where by $x \gg 0$ we mean $x$ is a totally positive element. We say an element is $S$-positive if $\ell_i(x) \geq 0$ for all $i$. Let $\mathfrak{ab}_{S} = \{x \in \mathfrak{ab} \, | \, x \textrm{ is } S\textrm{-positive}\}$ be the set of $S$-positive elements in $\mathfrak{ab}$. $\mathfrak{ab}_S$ is a finitely generated monoid. 

\begin{defn}
Define $\Z\llbracket\mathfrak{ab}, S\rrbracket$ to be the ring of all formal series $\sum_{\beta \in \mathfrak{ab}_S} a_{\beta} q^{\beta}$. Define $\Z(\!(\mathfrak{ab}, S)\!) = \Z\llbracket\mathfrak{ab}, S\rrbracket[1/q^{\beta}]$ for some $\beta \gg 0$.
\end{defn}

We remark that inverting $q^{\beta}$ for some $\beta \gg 0$ inverts $q^{\gamma}$ for all $\gamma \gg 0$. So $\Z(\!(\mathfrak{ab}, S)\!)$ is well-defined. In particular, $\Z(\!(\mathfrak{ab}, S)\!)$ is the collection of all formal series $\sum_{\beta \in \mathfrak{ab}} a_{\beta} q^{\beta}$ such that for some integer $n \gg 0$, we have $\ell_i(\beta) \geq -n$ whenever $a_{\beta} \neq 0$.

Over the ring $\Z(\!(\mathfrak{ab}, S)\!)$, we have the $g$-dimensional algebraic torus $\Gm \otimes \mathfrak{d}^{-1}\mathfrak{a}^{-1}$ together with an $\Sh{O}_L$-linear group homomorphism $\underline{q} \colon \mathfrak{b} \to \Gm \otimes \mathfrak{d}^{-1}\mathfrak{a}^{-1}$ defined as follows. To give such a group homomorphism is the same as giving an $\Sh{O}_L$-linear group homomorphism $\mathfrak{ab} \to \Gm \otimes \mathfrak{d}^{-1}$. This is equivalent to giving a group homomorphism $\mathfrak{ab} \to \Gm$ which we define to be $\beta \mapsto q^{\beta} \in \Gm(\Z(\!(\mathfrak{ab}, S)\!))$. The rigid analytic quotient $\Gm \otimes \mathfrak{d}^{-1}\mathfrak{a}^{-1}/\underline{q}(\mathfrak{b})$ is algebraizable to a Hilbert--Blumenthal abelian variety denoted $\Tate{a}{b}$ over $\Z(\!(\mathfrak{ab}, S)\!)$ which carries a canonical $\mathfrak{c} = \mathfrak{a}\mathfrak{b}^{-1}$ polarization
\[
\lambda_{\textrm{can}} \colon \Tate{a}{b}^{\vee} \xrightarrow{\sim} \Tate{a}{b} \otimes \mathfrak{a}\mathfrak{b}^{-1} \simeq \Tate{b}{a}.
\]

We quickly recall that there exists canonical isomorphisms as follows \cite[(1.1.17), (1.1.18)]{KatzNicholasM1978pLfC}.

\begin{enumerate}
    \item $\omega_{\Tate{a}{b}} \simeq \mathfrak{a} \otimes \Z(\!(\mathfrak{ab}, S)\!)$; \hspace{1pt} $\omega^{\vee}_{\Tate{a}{b}} \simeq \mathfrak{d}^{-1}\mathfrak{a}^{-1} \otimes \Z(\!(\mathfrak{ab}, S)\!)$.
    \item $\Omega^1_{\Z(\!(\mathfrak{ab}, S)\!)} \simeq \mathfrak{ab} \otimes \Z(\!(\mathfrak{ab}, S)\!)$; \hspace{1pt} $\Der\Big(\Z(\!(\mathfrak{ab}, S)\!), \Z(\!(\mathfrak{ab}, S)\!)\Big) \simeq \mathfrak{d^{-1}a^{-1}b^{-1}} \otimes \Z(\!(\mathfrak{ab}, S)\!)$.
\end{enumerate}

We now base change to $\Lambda^0_{\alpha, I}$, so that $\Tate{a}{b}$ is defined over $R := \Lambda^0_{\alpha, I}(\!(\mathfrak{ab}, S)\!)$.

For simplicity assume $\mathfrak{a}, \mathfrak{b}$ are coprime to $p$. Everything that follows holds true with appropriate modifications in the general case by choosing an isomorphism $\Sh{O}_L \otimes \Z_p \simeq \mathfrak{a}^{-1} \otimes \Z_p$ which amounts to choosing a $\Gamma_{00}(p^{\infty})$-structure on $\Tate{a}{b}$ \cite[(1.1.15)]{KatzNicholasM1978pLfC}. When $\mathfrak{a}$ is coprime to $p$, we have the natural equality $\Sh{O}_L \otimes \Z_p = \mathfrak{a}^{-1} \otimes \Z_p$ inside $L \otimes \Q_p$ which induces a canonical $\Gamma_{00}(p^{\infty})$-structure on $\Tate{a}{b}$.

For any $\sigma \in \Sigma$, let $e_{\sigma} \in \Sh{O}_L \otimes R$ be the corresponding idempotent. Let $\omega_{\textrm{can}}(\sigma)$ be the image of $e_{\sigma} \in \mathfrak{a} \otimes R = \Sh{O}_L \otimes R$ under the canonical identification $\omega_{\Tate{a}{b}} \simeq \mathfrak{a} \otimes R$. Let $\Theta_{\sigma} = KS(\omega_{\textrm{can}}^{\otimes 2}(\sigma))$ be the corresponding Kodaira--Spencer class. Then $\Theta_{\sigma}$ is the image of $e_{\sigma} \in \Sh{O}_L \otimes R$ under the identification $\Omega^1_{R/\Lambda^0_{\alpha, I}} \simeq \mathfrak{ab} \otimes R = \Sh{O}_L \otimes R$. The homomorphism $e_{\sigma}^{\vee} \colon \Omega^1_{R/\Lambda^0_{\alpha, I}} \to R$ that is dual to $e_{\sigma}$, induces the derivation $\theta_{\sigma} \in \Der(R, R)$ defined as 
\[
\theta_{\sigma}\Big(\sum a_{\beta}q^{\beta}\Big) = \sum \sigma(\beta)a_{\beta}q^{\beta}.
\]

Having recalled generalities about the Tate objects, we go back to computing the effect of $\nabla_k$ on $q$-expansions. 

Let $\nabla(\omega_{\textrm{can}}(\sigma)) = \zeta_{\textrm{can}}(\sigma) \otimes \Theta_{\sigma}$. Then $\nabla(\zeta_{\textrm{can}}(\sigma)) = 0$. Let $\omega_{\textrm{can}}$ and $\zeta_{\textrm{can}}$ be the $\Sh{O}_L \otimes R$-basis of $\mathrm{H}^{\sharp}_{\Tate{a}{b}} = {H}^1_{\textrm{dR}}(\Tate{a}{b}/R)$, whose $\sigma$-components are $\omega_{\textrm{can}}(\sigma)$ and $\zeta_{\textrm{can}}(\sigma)$ respectively. With respect to this basis the matrix of $\nabla = (\nabla_{\sigma})_{\sigma}$ is thus given as follows.
\[
\nabla_{\sigma} = \begin{pmatrix}
0 & 0 \\
\Theta_{\sigma} & 0
\end{pmatrix}
\]
(Note $\nabla_{\sigma} \colon \mathrm{H}^{\sharp}_{\Tate{a}{b}}(\sigma) \to \mathrm{H}^{\sharp}_{\Tate{a}{b}}(\sigma) \hat{\otimes} \Omega^1_{R/\Lambda^0_{\alpha, I}}$ is just the $\sigma$ component of $\nabla$. In particular, it should not be confused with $\nabla_k(\sigma)$ of Definition \ref{D402}).

Let $\mathbb{W}^0_k(q)$ be the pullback of $\mathbb{W}^0_k$ to $\Spf{R}$ along the structure morphism defining $\Tate{a}{b}$ together with the canonical $\Gamma_{00}(p^n)$-structure defined as above. Then we can write $\mathbb{W}^0_k(q) \simeq R\langle \{V_{\sigma}\}_{\sigma} \rangle \cdot k(1+p^nZ)$ as in Corollary \ref{C302}. Then formula (\ref{eq:11}) gives us
\begin{align}\label{eq:12}
\begin{split}
    \nabla_k(\sigma)\left(a\prod_{\tau} V_{\tau}^{i_{\tau}}\cdot k(1+p^nZ)\right) &= \theta_{\sigma}(a)\prod_{\tau} V_{\tau}^{i_{\tau}} \cdot (k+2\sigma)(1 + p^nZ) \\ &+ p(u_{\sigma} - i_{\sigma})V_{\sigma}\prod_{\tau} V_{\tau}^{i_{\tau}} \cdot (k+2\sigma)(1 + p^nZ)
\end{split}
\end{align}
for any $a \in R$.

\begin{lemma}\label{L404}
For any $\sigma, \tau \in \Sigma$, the maps $\nabla_k(\sigma)$ and $\nabla_k(\tau)$ commute, i.e. \[\nabla_{k+2\sigma}(\tau) \circ \nabla_k(\sigma) = \nabla_{k+2\tau}(\sigma) \circ \nabla_k(\tau)\]
as maps $\mathbb{W}^0_k \to \mathbb{W}^0_{k+2\sigma+2\tau}[1/\alpha]$.
\end{lemma}

\begin{proof}
It is enough to check this on the ordinary locus $\X^{\textrm{ord}}$, as the ordinary locus is dense in $\Xra{r}$. On the ordinary locus the result follows by verifying on $q$-expansions using ($\ref{eq:12}$) and the $q$-expansion principle. (See also \cite[(2.1.14)]{KatzNicholasM1978pLfC})
\end{proof}

\begin{lemma}\label{L405}
Let $g(q) \in R$ and $N \geq 1$. Then we can write
\[
\nabla(\sigma)^N\Big( g(q) \prod_{\tau} V_{\tau}^{i_{\tau}} \cdot k(1 + p^nZ) \Big) = \sum_{j=0}^N p^j a_{N,k,i_{\sigma},j} \theta_{\sigma}^{N-j}(g(q)) V_{\sigma}^j \prod_{\tau} V_{\tau}^{i_{\tau}} \cdot (k+ 2N\sigma)(1 + p^nZ).
\]
Here $a_{N,k,i_{\sigma}, 0} = 1$ and for $j \geq 1$, we have
\[
a_{N,k,i_{\sigma}, j} = \binom{N}{j} \prod_{i=1}^{j-1} (u_{\sigma} - i_{\sigma} + N - 1 - i).
\]
\end{lemma}

\begin{proof}
This is a similar computation as \cite[Lemma 3.38]{andreatta2021triple}.
\end{proof}

Let ${\mathbb{W}}^0_k(\sigma) := \sum_n \mathbb{W}^0_{k+2n\sigma}$ and let $\mathbb{W}^0 = \rho_*\Sh{O}_{\V{\Sh{O}_L}(\mathrm{H}^{\sharp}_{\Sh{A}}, s, \Sh{Q})}$ where $\rho \colon \V{\Sh{O}_L}(\mathrm{H}^{\sharp}_{\Sh{A}}, s, \Sh{Q}) \to \Xra{r}$ is the projection. Let $\mathbb{W}_k(\sigma) = \mathbb{W}^0_k(\sigma) \otimes \w{\chi}$ and let $\mathbb{W} = \mathbb{W}^0 \otimes \mathfrak{F}$. Here $\mathfrak{F} := (f_i)_*\Sh{O}_{\Ig{i}{r}} \otimes_{\Lambda^0_{\alpha,I}} \Lambda_{\alpha,I}$, where $f_i \colon \Ig{i}{r} \to \Xra{r}$ is the projection with $i = 1$ for $p \neq 2$ and $i = 2$ otherwise. For each prime $\mathfrak{P}|p$, we have defined the $U_{\mathfrak{P}}$ operator in the next section (Definition \ref{D3101}) which we will use now. 

\begin{cor}\label{C403}
Let $f_{\sigma}$ be the inertia degree for the embedding $\sigma$, and suppose that $\sigma$ induces the $\mathfrak{P}$-adic valuation on $\Sh{O}_L$ for a prime $\mathfrak{P}|p$. Let $k = \chi \cdot k^0$ with $\chi = k|_{\Delta}$ the torsion part of the character and $k^0 = k\chi^{-1}$. Assume $k^0_{\sigma}(t) = \exp(u_{\sigma}\log{t})$  for $t \in 1+\beta_n\Ga$ and $u_{\sigma} \in \Lambda^0_{\alpha,I}$. Let $\bar{\mathfrak{M}}_{\alpha,I}^{\textrm{ord}} = \X^{\textrm{ord}} \times_{\mathfrak{W}^0_{\alpha,I}} \mathfrak{W}_{\alpha,I}$. For any $g \in H^0(\bar{\mathfrak{M}}_{\alpha,I}^{\textrm{ord}}, \mathbb{W}_k)^{U_{\mathfrak{P}} = 0}$, \[(\nabla(\sigma)^{p^{f_{\sigma}} - 1}- \textrm{{id}})(g) \in pH^0\big(\bar{\mathfrak{M}}_{\alpha,I}^{\textrm{ord}},  {\mathbb{W}}\big) \cap H^0\big(\bar{\mathfrak{M}}_{\alpha,I}^{\textrm{ord}}, \mathbb{W}_k(\sigma)\big).\]
\end{cor}

\begin{proof}
We recall that $\mathbb{W}_k = \mathbb{W}^0_{k^0} \otimes \w{\chi}$ and the connection $\nabla$ on $\mathbb{W}_k$ is defined by the composite of the connection on $\mathbb{W}^0_{k^0}$ as defined in Theorem \ref{T401} and the connection on $\w{\chi}$ which is defined by the universal derivation on $\Sh{O}_{\Ig{i}{r}}$ for $i$ as above. Let $\Lambda_{\alpha,I}(\!(\mathfrak{ab}, S)\!) = \Lambda_{\alpha,I} \otimes_{\Lambda^0_{\alpha,I}} R$. The base change of $\Spf{\Lambda_{\alpha,I}}(\!(\mathfrak{ab}, S)\!)$ to $\Ig{i}{r} \to \Xra{r}$ is just copies of $\Spf{\Lambda_{\alpha,I}}(\!(\mathfrak{ab}, S)\!)$ indexed by $(\Sh{O}_L/q\Sh{O}_L)^{\times}$. Hence the universal derivation on ${\Lambda_{\alpha,I}}(\!(\mathfrak{ab}, S)\!) \otimes_{\Sh{O}_{\Xra{r}}} \Sh{O}_{\Ig{i}{r}}$ is determined by the universal derivation on ${\Lambda_{\alpha,I}}(\!(\mathfrak{ab}, S)\!)$. The $q$-expansion of any section $g \in H^0(\bar{\mathfrak{M}}_{r, \alpha,I}, \mathbb{W})^{U_{\mathfrak{P}}=0}$ at $\Tate{a}{b}$ corresponds to a tuple $(g_j)_{j \in (\Sh{O}_L/q\Sh{O}_L)^{\times}}$ with $g_j \in \Lambda_{\alpha,I}(\!(\mathfrak{ab}, S)\!)\langle \{V_{\sigma}\}_{\sigma}\rangle \cdot k^0(1+p^nZ)$. Moreover, for each $j$, $U_{\mathfrak{P}}(g_j) = 0$. By the $q$-expansion principle it will be enough to prove the corollary for $g = g(q)\prod_{\tau} V^{i_{\tau}}_{\tau}\cdot k^0(1+p^nZ)$ for $g(q) \in \Lambda_{\alpha,I}(\!(\mathfrak{ab}, S)\!)$, such that $g(q)$ is $\mathfrak{P}$-depleted. By Lemma \ref{L405}, it is enough to show $\theta_{\sigma}^{p^{f_{\sigma}}-1}(g(q)) \cdot (k+2(p^{f_{\sigma}}-1)\sigma)(1+p^nZ) \equiv g(q) k(1+p^nZ) \textrm{ mod }p$, which is clear.
\end{proof}

\subsection{Iteration of \texorpdfstring{$\nabla$}{the connection}}\label{2S44}
In this section we will finally define the $p$-adic iteration of the Gauss--Manin connection. We begin with a preparatory lemma.

\begin{lemma}\label{L406}
For $I = [p^a, p^b]$, 
\begin{enumerate}
    \item $\Lambda^0_{\alpha,I} = \Sh{O}_K[[T_1, \dots, T_g]]\langle \frac{p}{\alpha}, \frac{T_1}{\alpha}, \dots, \frac{T_g}{\alpha}, u, v\rangle/({\alpha}^{p^a}v - p, uv - {\alpha}^{p^{b-a}})$ if $b \neq \infty$.
    \item $\Lambda^0_{\alpha, I} = \Sh{O}_K[[T_1, \dots, T_g]]\langle \frac{p}{\alpha}, \frac{T_1}{\alpha}, \dots, \frac{T_g}{\alpha}, u\rangle/({\alpha}^{p^a}u - p)$ if $b = \infty$.
\end{enumerate}
Let $U := \Spf{A}$ be a Zariski open in $\mathfrak{X}$ where $\omega_{\Sh{A}}$ is trivial. Then 
\[
U \times_{\mathfrak{X}} \Xr{r} = \Spf{A \hat{\otimes} \Lambda^0_{\alpha,I}\langle w \rangle/(w\Hdg{p^{r+1}}} - \alpha).
\]
\end{lemma}

\begin{proof}
See \cite[\S 3.4.1]{andreatta2016adic}.
\end{proof}

\begin{lemma}\label{L407}
Let $C_1 = \#(\Sh{O}_L/p\Sh{O}_L)^{\times}$. Then the kernel of the restriction map $\Sh{O}_{\Ig{1}{r}}/(\alpha^j) \xrightarrow{\phi_1} \Sh{O}_{\mathfrak{IG}^{\textrm{\emph{ord}}}_{1,I}}/(\alpha^j)$
is killed by $\Hdg{j(p^{r+1}) + C}$. The kernel of the restriction map $\Sh{O}_{\Ig{n}{r}}/(\alpha^j) \xrightarrow{\phi_n} \Sh{O}_{\mathfrak{IG}_{n,I}^{\textrm{ord}}}/(\alpha^j)$ is killed by $\Hdg{j(p^{r+1})+C_n}$ where $C_n = C_1 + \frac{p^n-p}{p-1}$.
\end{lemma}

\begin{proof}
The formulas of Lemma \ref{L406} show that the kernel of $\Sh{O}_{\Xr{r}}/(\alpha^j) \xrightarrow{\phi_0} \Sh{O}_{\mathfrak{X}^{\textrm{ord}}_{\alpha,I}}/(\alpha^j)$ is killed by $\Hdg{j(p^{r+1})}$. The trace map $\mathrm{Tr} \colon \Sh{O}_{\Ig{1}{r}} \to \Sh{O}_{\Xr{r}}$ then gives a commutative diagram as follows.
\[
\begin{tikzcd}
0 \arrow[r] & \ker{\phi_1} \arrow[r] \arrow[d] & \Sh{O}_{\Ig{1}{r}}/(\alpha^j) \arrow[r] \arrow[d, "\mathrm{Tr}"] & {\Sh{O}_{\mathfrak{IG}^{\textrm{ord}}_{1,I}}/(\alpha^j)} \arrow[d, "\mathrm{Tr}"] \\
0 \arrow[r] & \ker{\phi_0} \arrow[r]           & \Sh{O}_{\Xr{r}}/(\alpha^j)  \arrow[r]                                     & {\Sh{O}_{\mathfrak{X}^{\textrm{ord}}_{\alpha, I}}/(\alpha^j)}                    
\end{tikzcd}
\]
Suppose $x \in \ker{\phi_1}$. Then $\mathrm{Tr}(x) \in \ker{\phi_0}$ and hence $\mathrm{Tr}(\Hdg{j(p^{r+1})}x) = 0$. In other words, for any lift $\tilde{x} \in \Sh{O}_{\Ig{1}{r}}$ of $x$, $\mathrm{Tr}(\Hdg{j(p^{r+1})}\tilde{x}) \in \alpha^j\Sh{O}_{\Xr{r}}$. Let $\mathfrak{D}^{-1} := \{y \in \textrm{Frac}(\Sh{O}_{\Ig{1}{r}}) \, | \, \mathrm{Tr}(yz) \in \Sh{O}_{\Xr{r}} \textrm{ for all } z \in \Sh{O}_{\Ig{1}{r}}\}$. Then $\Hdg{j(p^{r+1})}\tilde{x} \in \alpha^j\mathfrak{D}^{-1}$ as $\ker{\phi_1}$ is an ideal. By using normality of the rings involved, the first claim then follows by localizing at height 1 primes and noting that $\mathfrak{D}^{-1}$ is the usual inverse different in such tamely ramified extensions of DVR's. For the second part, we note that $\Ig{n}{r}$ is the normalization of $\mathfrak{Y} := \Ig{1}{r}\times_{H^{\vee}_1} H_n^{\vee}$ where $\Ig{1}{r} \to H_1^{\vee}$ is the universal generator of $H_1^{\vee}$. The faithfully flat extension $H_n^{\vee} \to H_1^{\vee}$ has different $\mathfrak{D}(H_n^{\vee}/H_1^{\vee})$ that contains $\Hdg{\frac{p^n-p}{p-1}}$ \cite[Proposition 3.5]{Andreatta2018leHS}. By flatness the kernel of $\Sh{O}_{\mathfrak{Y}}/(\alpha^j) \to \Sh{O}_{\mathfrak{Y}^{\textrm{ord}}}/(\alpha^j)$ is killed by $\Hdg{j(p^{r+1}) + C_1}$. Since $\Ig{n}{r}$ is the normalization of $\mathfrak{Y}$, it is finite and in particular $\Ig{n}{r} \subset \mathfrak{D}^{-1}(\mathfrak{Y}/\Ig{1}{r})$. Thus $\Hdg{\frac{p^n-p}{p-1}}\Sh{O}_{\Ig{n}{r}} \subset \Sh{O}_{\mathfrak{Y}}$, which proves the second claim.
\end{proof}

\begin{lemma}\label{L2408}
Let $g_n \colon \Ig{n}{r} \to \Xra{r}$ be the projection. The kernel and cokernel of $g_n^*\Omega^1_{\Xra{r}/\Lambda^0_{\alpha,I}} \to \Omega^1_{\Ig{n}{r}/\Lambda^0_{\alpha,I}}$ is killed by a power of $\Hdg{}$. Let $\vartheta \colon \Ig{n}{r}' \to \Ig{n}{r}$ be the projection. The kernel of $\vartheta^*\Omega^1_{\Ig{n}{r}/\Lambda^0_{\alpha,I}} \to \Omega^1_{\Ig{n}{r}'/\Lambda^0_{\alpha,I}}$ is killed by a power of $\Hdg{}$.
\end{lemma}

\begin{proof}
This is similar to \cite[Lemma 3.3]{andreatta2021triple}. For the second part one views ${\mathfrak{IG}'}^{,\textrm{ord}}_{n,I}$ as a torsor over $\X^{\textrm{ord}}$ for the group 
$\left(\begin{smallmatrix} (\Sh{O}_L/p^n\Sh{O}_L)^{\times} & \mu_{p^n} \otimes \mathfrak{d}^{-1} \\ 0 & (\Sh{O}_L/p^n\Sh{O}_L)^{\times} \end{smallmatrix}\right)$, and argues as before using smoothness of $\Igord{n}$.
\end{proof}

\begin{ass}\label{A01}
Let $k \colon \mathbb{T}(\Z_p) \to {(\Lambda_{\alpha, I})}^{\times}$ be a weight such that $k = \chi k^0$ where $\chi = k|_{\Delta}$ is the finite part of the  character and $k^0 = k\chi^{-1}$. Assume that for all $\sigma \in \Sigma$, there exists $u_{\sigma} \in \Lambda^0_{\alpha, I}$, such that $k^0$ factors as 
\begin{align*}
k^0 \colon (\Sh{O}_L \otimes \Z_p)^{\times} \to (\Sh{O}_L \otimes \Sh{O}_K)^{\times} \simeq \prod_{\sigma} \Sh{O}_K^{\times} \xrightarrow{(k^0_{\sigma})_{\sigma}} {(\Lambda^0_{\alpha,I})}^{\times}
\end{align*}
with $k^0_{\sigma}(t) = \exp(u_{\sigma}\log{t})$ for all $t \in \Sh{O}_K^{\times}$. 

Let $s \colon \mathbb{T}(\Z_p) \to {(\Lambda^0_{\alpha,I})}^{\times}$ be a weight such that for all $\sigma \in \Sigma$, there exists $v_{\sigma} \in \Lambda^0_{\alpha,I}$ such that $s$ factors as
\[
s \colon (\Sh{O}_L \otimes \Z_p)^{\times} \to (\Sh{O}_L \otimes \Sh{O}_K)^{\times} \simeq \prod_{\sigma} \Sh{O}_K^{\times} \xrightarrow{(s_{\sigma})_{\sigma}} {(\Lambda^0_{\alpha,I})}^{\times}.
\]
with $s_{\sigma}(t) = \exp(v_{\sigma}\log{t})$ for all $t \in \Sh{O}_K^{\times}$.

In particular, we can take $\alpha = p$ and $I = [0,1]$ for $p\neq 2$, and $I = [0,1/2]$ for $p = 2$.
\end{ass} 

Note that the explicit description of the Gauss--Manin connection in (\ref{eq:11}) together with Lemma \ref{L2408} implies that there exists an integer $D$ such that $\nabla_k(\sigma)(\mathbb{W}_k) \subset \frac{1}{p\Hdg{D}}\mathbb{W}_{k+2\sigma}$ for all $\sigma$. Let $\nabla(\sigma) \colon \mathbb{W} \to \frac{1}{p\Hdg{D}}\mathbb{W}$ be the map defined by $\nabla(\sigma)|_{\mathbb{W}_k} = \nabla_k(\sigma)$. In particular, for all $N \geq 1$, 
\[
(\nabla(\sigma)^{p^{f_{\sigma}}-1} - \textrm{id})^N(\mathbb{W}_{k}) \subset \frac{1}{(p\Hdg{D})^{(p^{f_{\sigma}}-1)N}} \mathbb{W}_k(\sigma).
\]
\begin{lemma}\label{L408}
Let $\sigma \in \Sigma$ induce the $\mathfrak{P}$-adic valuation on $L$. There exists an integer $\ell$ depending on $r,n$ and $p$, and an integer $C > 0$, such that for any $g \in H^0(\bar{\mathfrak{M}}_{r,p,[0,1]}, \mathbb{W}_{k})^{U_{\mathfrak{P}}=0}$, and every positive integer $N$, we have
\begin{align*}
    \Big(\nabla(\sigma)^{p^{f_{\sigma}}-1} - \textrm{\emph{id}}\Big)^N(g) \in \left(\frac{p}{\Hdg{C}}\right)^{N}H^0(\bar{\mathfrak{M}}_{\ell,p,[0,1]}, \mathbb{W}) \cap H^0(\bar{\mathfrak{M}}_{\ell,p,[0,1]}, \mathbb{W}_k(\sigma)).
\end{align*}
\end{lemma}

\begin{proof}
By Corollary \ref{C403} we see that 
\[
\Big(\nabla(\sigma)^{p^{f_{\sigma}}-1} - \textrm{id}\Big)^N(g)|_{\bar{\mathfrak{M}}^{\textrm{ord}}_{p,[0,1]}} \in p^{N}H^0(\bar{\mathfrak{M}}^{\textrm{ord}}_{p,[0,1]}, \mathbb{W}) \cap H^0(\bar{\mathfrak{M}}^{\textrm{ord}}_{p,[0,1]}, \mathbb{W}_{k}(\sigma)).
\]
Locally on $\bar{\mathfrak{M}}_{r,p,[0,1]}$, we then have that 
\[
(p\Hdg{D})^{(p^{f_{\sigma}}-1)N}\Big(\nabla(\sigma)^{p^{f_{\sigma}}-1} - \textrm{id}\Big)^N(g) \in \ker{\left(\mathbb{W}/(p^{p^{f_{\sigma}}N}) \to \mathbb{W}^{ \textrm{ord}}/(p^{p^{f_{\sigma}}N})\right)}.
\]
Here $\mathbb{W}^{\textrm{ord}} = \mathbb{W}|_{\bar{\mathfrak{M}}^{\textrm{ord}}_{p,[0,1]}}$. By Corollary \ref{C302} $\mathbb{W}'/(p^j)$ is a polynomial algebra over $\Sh{O}_{\Ig{n}{r}}/(p^j)$ for any $j$. Since $\mathbb{W} = \mathbb{W}^0 \otimes \mathfrak{F}$, we first deal with $\mathbb{W}^0$. Here we see by Lemma \ref{L407}
the kernel of $\mathbb{W}^0/(p^{p^{f_{\sigma}}N}) \to \mathbb{W}^{0, \textrm{ord}}/(p^{p^{f_{\sigma}}N})$ is killed by $\Hdg{p^{f_{\sigma}}N(p^{r+1})+C_n}$. By the same lemma  $\ker\left(\mathfrak{F}/(p^{p^{f_{\sigma}}N}) \to \mathfrak{F}^{\textrm{ord}}/(p^{p^{f_{\sigma}}N})\right)$ is killed by $\Hdg{p^{f_{\sigma}}N(p^{r+1})+C_2}$. Therefore
\[
p^{(p^{f_{\sigma}}-1)N}\Hdg{N(2p^{f_{\sigma}}(p^{r+1})+D(p^{f_{\sigma}}-1)+C_n + C_2}\Big(\nabla(\sigma)^{p^{f_{\sigma}}-1} - \textrm{id}\Big)^N(g) \in p^{p^{f_{\sigma}}N}H^0(\bar{\mathfrak{M}}_{r, \alpha,I}, \mathbb{W}).
\]
In particular, choosing $C \gg 0$, such that $CN \geq N(2p^{f_{\sigma}}(p^{r+1})+D(p^{f_{\sigma}}-1)+C_n + C_2$ for all $N > 0$, we see that 
\[
\Hdg{CN}\Big(\nabla(\sigma)^{p^{f_{\sigma}}-1} - \textrm{id}\Big)^N(g) \in p^N H^0(\bar{\mathfrak{M}}_{r, \alpha,I}, \mathbb{W}).
\]
Choosing $\ell \geq r$ such that $p/\Hdg{C} \in \Sh{O}_{\Xr{\ell}}$, we get that 
\[
\Big(\nabla(\sigma)^{p^{f_{\sigma}}-1} - \textrm{id}\Big)^N(g) \in \left(\frac{p}{\Hdg{C}}\right)^N H^0(\bar{\mathfrak{M}}_{\ell, \alpha,I}, \mathbb{W}).
\]
\end{proof}

\begin{prop}\label{P402}
Let $k, s$ be as in Assumption \ref{A01}, and suppose $\sigma \in \Sigma$ induce the $\mathfrak{P}$-adic valuation.
Then for any prime $p \geq 3$, there exists an integer $\ell$ depending on $r, n$ and $p$ such that for every $g \in H^0(\bar{\mathfrak{M}}_{r,\alpha,I}, \mathbb{W}_k)^{U_{\mathfrak{P}}=0}$ the sequences in $m$
\[
A(g, s_{\sigma})_m := \sum_{j = 1}^{m} (-1)^{j-1}\frac{(\nabla(\sigma)^{p^{f_{\sigma}}-1} - \textrm{id})^j(g)}{j}
\]
and if we write $H_{i,m}$ for the set of tuples $(j_1, \dots, j_i)$ of $i$ positive integers with $j_1 + \cdots + j_i \leq m$,
\[
B(g, s_{\sigma})_m := \sum_{i=0}^m \frac{v_{\sigma}^i}{i!(p^{f_{\sigma}}-1)^i}\left(\sum_{(j_1, \dots, j_i) \in H_{i,m}} \Big(\prod_{a=1}^i \frac{(-1)^{j_a - 1}}{j_a}\Big)(\nabla(\sigma)^{p^{f_{\sigma}}-1} - \textrm{id})^{j_1 + \cdots + j_i}\right)(g)
\]
converge in $H^0(\bar{\mathfrak{M}}_{\ell,\alpha,I}, \mathbb{W})$. Moreover, if we denote the limits
\[
\log{\Big(\nabla(\sigma)^{p-1}\Big)}(g) := \lim_{m \to \infty} A(g,s_{\sigma})_m
\]
and
\[
\nabla(\sigma)^{s_{\sigma}}(g) = \exp\left(\frac{v_{\sigma}}{p^{f_{\sigma}}-1}\log{(\nabla(\sigma)^{p^{f_{\sigma}}-1})}\right)(g) := \lim_{m \to \infty} B(g, s_{\sigma})_m
\]
then $\nabla(\sigma)^{s_{\sigma}}(g) \in H^0(\bar{\mathfrak{M}}_{\ell,\alpha,I}, \mathbb{W}_{k+2s_{\sigma}})$.
The same results hold for $p = 2$ if $v_{\sigma} \in 4\Lambda^0_{\alpha, I}$.
\end{prop}

\begin{proof}
The convergence of $A(g,s_{\sigma})_m$ is clear from Lemma \ref{L408}. We prove convergence for $B(g, s_{\sigma})_m$.

Let's first deal with the case $p \geq 3$. Let 
\[
X := \frac{(\nabla(\sigma)^{p^{f_{\sigma}}-1}-\textrm{id})^{j_1 + \cdots +j_i}(g)}{i!\prod j_a}.
\]
Then by Lemma \ref{L408}, $X \in (p/\Hdg{C})^{\sum j_a - v_p(i!) - \sum v_p(j_a)}H^0(\bar{\mathfrak{M}}_{\ell,\alpha,I}, \mathbb{W})$. Now $v_p(i!) \leq \frac{i-1}{p-1} \leq \frac{i}{p-1}$. Hence $v_p(j_a) \leq \frac{j_a-1}{p-1}$ too. Using these inequalities,
\begin{align*}
    \sum_{a=1}^i j_a - v_p(i!) - \sum_{a=1}^i v_p(j_a) \geq \sum_{a=1}^i \left(j_a - \frac{1}{p-1} - v_p(j_a)\right) \geq \sum_{a=1}^i j_a\left(1 - \frac{1}{p-1}\right).
\end{align*}
This proves convergence in this case. For the case $p = 2$ we note that the terms $\frac{(\nabla(\sigma)^{2^{f_{\sigma}}-1} - \textrm{id})^{j_1+\cdots+j_i}(g)}{\prod j_a}$ do not have poles and the term $v_{\sigma}^i/i!$ is divisible by $2^i$, which gives convergence in this case. Finally, $\nabla(\sigma)^{s_{\sigma}}(g) \in H^0(\bar{\mathfrak{M}}_{\ell,\alpha,I}, \mathbb{W}_{k+2s_{\sigma}})$ as can be seen from its expansion as a power series and the fact that $t \ast \nabla(\sigma)(g) = (k+2\sigma)(t)\nabla(\sigma)(g)$.
\end{proof}

Thus given any $g \in H^0(\bar{\mathfrak{M}}_{r,\alpha,I}, \mathbb{W}_k)^{\{U_{\mathfrak{P}_i}=0\}_{i=1}^h}$, i.e. $g$ lies in the kernel of $U_{\mathfrak{P}_i}$ for all $i$, there exists a large enough $\ell$ depending on $r,n$ and $p$ such that one can consider $\prod_{\sigma} \nabla(\sigma)^{s_{\sigma}}(g)$ as an element of $H^0(\bar{\mathfrak{M}}_{\ell,\alpha,I}, \mathbb{W}_{k+2s})$. Here $\prod_{\sigma} \nabla(\sigma)^{s_{\sigma}}$ means the composition of the different $\nabla(\sigma)^{s_{\sigma}}$'s in any order. Note the order of composition does not matter since they mutually commute by Lemma \ref{L404}. Thus we fix such an $\ell$ and define the following.

\begin{defn}
For $s \colon \mathbb{T}(\Z_p) \to (\Lambda^0_{\alpha, I})^{\times}$ as in Proposition \ref{P402} and $k$ as in Assumption \ref{A01}, define $\nabla^s(g)$ for $g \in H^0(\bar{\mathfrak{M}}_{r,p,[0,1]}, \mathbb{W}_{k})^{\{U_{\mathfrak{P}_i}=0\}_{i=1}^h}$ to be $\prod_{\sigma} \nabla(\sigma)^{s_{\sigma}}(g) \in H^0(\bar{\mathfrak{M}}_{\ell,p,[0,1]}, \mathbb{W}_{k+2s})$ for some $\ell$ for which the expression makes sense by Proposition \ref{P402}.
\end{defn}

%% file: Hecke.tex
\section{Hecke operators and overconvergent projection}
In this section we always assume that $\alpha = p$ and $I = [0,1]$.

Recall that $p = \mathfrak{P}_1\cdots\mathfrak{P}_h$. In this section we first define the $U_{\mathfrak{P}_i}$ operators on $\W{}$ and the $V_{\mathfrak{P}_i}$ operators on $\w{}$. Next we will define the prime to $p$ Hecke operators on $\w{}$.

\subsection{The \texorpdfstring{$U_{\mathfrak{P}}$}{Up} and \texorpdfstring{$V_{\mathfrak{P}}$}{Vp} operators}
For each $i = 1, \dots, h$, fix $x_i \in L^{\times, +}$ such that $\varpi_i/x_i \in (\Sh{O}_L \otimes \Z_p)^{\times}$, and such that $\prod_{i=1}^h x_i = p$. Here $\varpi_i \in \Sh{O}_L \otimes \Z_p = \prod_{j=1}^h \hat{\Sh{O}}_{\mathfrak{P}_j}$ is the element with $p$ in the $i$-th coordinate and $1$ elsewhere.

Fix a suitable multi-index $\mathbf{r} = (r_1, \dots, r_h)$. Let $\mathbf{r}' = (r_1, \dots, r_i-1, \dots, r_h)$. Let $\M{\mathbf{r}}{\mathfrak{c}} \subset \Mbar{\mathbf{r}}{\mathfrak{c}}$ be the inverse image of $M(\mu_N, \mathfrak{c}) \subset \bar{M}(\mu_N, \mathfrak{c})$ under the projection $\Mbar{\mathbf{r}}{\mathfrak{c}} \to \bar{M}(\mu_N, \mathfrak{c})$, and let $\Sh{M}^{\mathfrak{c}}_{\mathbf{r},\alpha,I}$ be its generic fibre. For each $i = 1,\dots, h$, consider the subspace $\mathfrak{Y}_i \subset \M{\mathbf{r}'}{\mathfrak{c}} \times \M{\mathbf{r}}{x_i^{-1}\mathfrak{P}_i\mathfrak{c}}$ obtained as the normalization of the Hecke correspondence classifying pairs $(A',\iota',\lambda',\psi')$ and $(A,\iota,\lambda, \psi)$, together with an isogeny $\pi_{i}\colon A' \to A$, such that $\lambda = x_i^{-1}{\pi_{i}}_*\lambda'$, $\ker \pi_i$ is \'{e}tale locally isomorphic to $\Sh{O}_L/\mathfrak{P}_i$, and $\ker \pi_i$ neither intersects the level $N$ subgroup nor the canonical subgroup $H_1(A')$. Choosing suitable smooth toroidal compactification, we can extend the Hecke correspondence to get maps $p_1 \colon \bar{\mathfrak{Y}}_i \to \Mbar{\mathbf{r}'}{\mathfrak{c}}$, and $p_2 \colon \bar{\mathfrak{Y}_i} \to \Mbar{\mathbf{r}}{x_i^{-1}\mathfrak{P}_i\mathfrak{c}}$. The map $p_1$ is generically finite flat of degree $p^{f_i}$, and induces a trace map $\mathrm{Tr}_{p_1}$ on the formal sheaves by normality. The Hecke correspondence lifts canonically to the partial Igusa tower. By Proposition \ref{P317}, the universal isogeny $\pi_i$ induces a morphism of $\Sh{O}_L \otimes \Sh{O}_{\Ig{n}{\mathbf{r}'}}$-modules $\pi_i^* \colon p_2^*\Hs{A} \to p_1^*\Hs{A}$ that preserves the Hodge filtration and the splitting, and induces an isomorphism $p_2^*\Omega_{\Sh{A}} \xrightarrow{\sim} p_1^*\Omega_{\Sh{A}}$. (We drop the polarization module from the notation when the context is clear)

\begin{lemma}
    There is a morphism $\Sh{U}_{\mathfrak{P}_i} \colon {p_1}_*p_2^*\W{0} \to {p_1}_*p_1^*\W{0}$ of $\Sh{O}_{\Xra{\mathbf{r}'}}$-modules induced by the universal isogeny $\pi_i$ which is an isomorphism on the modular sheaf $\w{0}$ and which preserves the filtration and commutes with the Gauss--Manin connection. We also have a morphism $\Sh{U}_{\mathfrak{P}_i} \colon {p_2}_*p_1^*\W{} \to {p_2}_*p_2^*\W{}$ of $\Sh{O}_{\Mbar{\mathbf{r}'}{}}$-modules satisfying the same properties as above.
\end{lemma}

\begin{proof}
    The first claim follows from the definition of $\W{0}$ and the morphism $\pi_i^* \colon p_2^*\Hs{A} \to p_1^*\Hs{A}$ defined above. Considering the $\Sh{U}$-correspondence on $\w{\chi}$ we get the required map for $\W{}$.
\end{proof}

We can similarly consider the normalization $\mathfrak{Y} \subset \M{r}{\mathfrak{c}} \times \M{r+1}{\mathfrak{c}}$ of the Hecke correspondence classifying pairs $(A,\iota, \lambda, \psi)$ and $(A',\iota',\lambda',\psi')$, together with an isogeny $\pi\colon A'\to A$, such that $\lambda = p^{-1}\pi_*\lambda'$, $\ker{\pi}$ is \'{e}tale locally isomorphic to $\Sh{O}_L/p\Sh{O}_L$, and $\ker{\pi}$ neither intersects the level $N$ subgroup nor the canonical subgroup. Then similarly as above we have the following lemma.

\begin{lemma}\label{L3101}
There is a morphism $\Sh{U} \colon {p_1}_*p_2^*\W{0} \to {p_1}_*p_1^*\W{0}$ of $\Sh{O}_{\Xra{r}}$-modules induced by the universal isogeny $\pi$ which is an isomorphism on the modular sheaf $\w{0}$ and which preserves the filtration and commutes with the Gauss--Manin connection. We also have a morphism $\Sh{U} \colon {p_2}_*p_1^*\W{} \to {p_2}_*p_2^*\W{}$ of $\Sh{O}_{\Mbar{r}{}}$-modules satisfying the same properties as above. Moreover, the induced map on the $m$-graded pieces is $0$ modulo $(p/\Hdg{p+1})^m$.
\end{lemma}

\begin{proof}
The proof of the first two claims is similar as above. For the last claim we observe that by Proposition \ref{P232}, the induced map $\Hs{A'}/\Omega_{\Sh{A}'} \to \Hs{A}/\Omega_{\Sh{A}}$ is multiplication by $p/HW(\sigma)^{p+1}$ on the $\sigma$-component. The claim then follows from the local description of $\W{0}$.
\end{proof}

\begin{defn}\label{D3101}
Define the $U_{\mathfrak{P}_i}$ operator as the composition
\[
U_{\mathfrak{P}_i} \colon H^0(\Mbar{\mathbf{r}}{x_i^{-1}\mathfrak{P}_i\mathfrak{c}}, \W{}) \xrightarrow{{p_1}_* \circ \Sh{U}_{\mathfrak{P}_i} \circ p_2^*} H^0(\Mbar{\mathbf{r}'}{\mathfrak{c}}, {p_1}_*p_1^*\W{}) \xrightarrow{\frac{1}{p^{f_i}}\mathrm{Tr}} H^0(\Mbar{\mathbf{r}'}{\mathfrak{c}}, \W{})[1/p].
\]
Similarly, define the $U$ operator as the composition
\[
U \colon H^0(\Mbar{r+1}{\mathfrak{c}}, \W{}) \xrightarrow{{p_1}_* \circ \Sh{U} \circ p_2^*} H^0(\Mbar{r}{\mathfrak{c}}, {p_1}_*p_1^*\W{}) \xrightarrow{\frac{1}{p^{g}}\mathrm{Tr}} H^0(\Mbar{r}{\mathfrak{c}}, \W{})[1/p].
\]
\end{defn}

\begin{rem}
    Since we chose the $x_i$'s such that $\prod_{i=1}^h x_i = p$, we have that $\prod_{i=1}^h U_{\mathfrak{P}_i} = U$.
\end{rem}

\begin{cor}
For any $g \in H^0(\Mbar{r+1}{}, \w{})$ with $q$-expansion $g = \sum_{\mathfrak{ab}} a_{\beta}q^{\beta}$ at a cusp $\Lambda^0_{\alpha,I}(\!(\mathfrak{ab}, S)\!)$, $U(g) = \sum_{\mathfrak{ab}} a_{p\beta}q^{\beta}$ on $q$-expansions.
\end{cor}

We now define the $V_{\mathfrak{P}_i}$ operator. Fix $\mathbf{r}, \mathbf{r}'$ as before. Consider the map $p_2 \colon \Mbar{\mathbf{r}}{\mathfrak{c}} \to \Mbar{\mathbf{r}'}{x_i\mathfrak{P}_i^{-1}\mathfrak{c}}$ defined on the generic fibre by sending a tuple $(A, \lambda, \iota, \psi)$ to $(A', \iota', x_i\lambda', \psi')$, where $A' = A/H_1[\mathfrak{P}_i] \otimes \mathfrak{P}_i$, and $\iota', \lambda', \psi'$ are induced by the quasi-isogeny $A \xrightarrow{\pi} A/H_1[\mathfrak{P}_i] \xrightarrow{} A/H_1[\mathfrak{P}_i] \otimes \mathfrak{P}_i$. The map $p_2$ lifts to the partial Igusa tower. The map $\pi^{\vee} \colon A/H_1[\mathfrak{P}_i] \to A \otimes \mathfrak{P}_i^{-1}$ such that $\pi^{\vee} \circ \pi$ is the quotient by the $\mathfrak{P}_i$-torsion, induces an isomorphism $(\pi^{\vee})^* \colon \Omega_{A} \otimes \mathfrak{P}_i^{-1} \to \Omega_{A/H_1[\mathfrak{P}_i]}$, which upon tensoring by $\mathfrak{P}_i$ gives an isomorphism $(\pi^{\vee})^* \colon \Omega_{A} \to \Omega_{A'}$. Let $p_1 \colon \Mbar{\mathbf{r}}{\mathfrak{c}} \to \Mbar{\mathbf{r}}{\mathfrak{c}}$ be the identity. Considering the isomorphism on the universal objects $(\pi^{\vee})^* \colon p_1^*\Omega_{\Sh{A}} \xrightarrow{\sim} p_2^*\Omega_{\Sh{A}}$, we get an isomorphism $\Sh{V}_{\mathfrak{P}_i} \colon p_1^*\w{} \simeq p_2^*\w{}$.

\begin{defn}
Define the $V_{\mathfrak{P}_i}$ operator as the map 
\[
V_{\mathfrak{P}_i} \colon H^0(\Mbar{\mathbf{r}'}{x_i\mathfrak{P}_i^{-1}\mathfrak{c}}, \w{}) \xrightarrow{\Sh{V}^{-1}_{\mathfrak{P}_i} \circ p_2^*} H^0(\Mbar{\mathbf{r}}{\mathfrak{c}}, \w{}).
\]
\end{defn}

We can similarly consider the map $p_2 \colon \Mbar{r+1}{\mathfrak{c}} \to \Mbar{r}{\mathfrak{c}}$ defined on the generic fibre by sending $(A, \iota, \lambda, \psi)$ to $(A', \iota', \lambda', \psi')$ with $A' = A/H_1$ and $\iota', \lambda', \psi'$ induced by the quotient $\pi \colon A \to A/H_1$. The map $p_2$ lifts to the partial Igusa tower. Let $p_1 \colon \Mbar{r+1}{\mathfrak{c}} \to \Mbar{r+1}{\mathfrak{c}}$ be the identity. The isogeny $\pi^{\vee} \colon A/H_1 \to A$ such that $\pi^{\vee} \circ \pi = p$ induces an isomorphism $(\pi^{\vee})^* \colon p_1^*\Omega_{\Sh{A}} \xrightarrow{\sim} p_2^*\Omega_{\Sh{A}}$ which in turn induces an isomorphism $\Sh{V} \colon p_1^*\w{} \simeq p_2^*\w{}$.

\begin{defn}
Define the $V$ operator as the map
\[
V \colon H^0(\Mbar{r}{\mathfrak{c}}, \w{}) \xrightarrow{\Sh{V}^{-1} \circ p_2^*} H^0(\Mbar{r+1}{\mathfrak{c}}, \w{}).
\]
\end{defn}

\begin{rem}
We have $\prod_{i} V_{\mathfrak{P}_i} = V$.
\end{rem}

\begin{cor}
For any $g \in H^0(\Mbar{r}{}, \w{})$ with $q$-expansion $g = \sum_{\mathfrak{ab}} a_{\beta}q^{\beta}$ at a cusp $\Lambda^0_{\alpha,I}(\!(\mathfrak{ab}, S)\!)$, $V(g) = \sum_{\mathfrak{ab}} a_{\beta}q^{p\beta}$ on $q$-expansions.
\end{cor}

\begin{proof}
    Left to the reader.
\end{proof}

\begin{cor}
\begin{enumerate}
    \item $U_{\mathfrak{P}_i} \circ V_{\mathfrak{P}_i} = \textrm{id}$ on $H^0(\Mbar{\mathbf{r}}{}, \w{})$ for any suitable multi-index $\mathbf{r}$.
    \item For $g \in H^0(\Mbar{\mathbf{r}}{}, \w{})$, if we denote by $g^{[\mathfrak{P}_i]} := (\textrm{id} - V_{\mathfrak{P}_i} \circ U_{\mathfrak{P}_i})(g)$ the $\mathfrak{P}_i$-depletion of $g$, then $U_{\mathfrak{P}_i}(g^{[\mathfrak{P}_i]}) = 0$. Moreover if $g = \sum_{\mathfrak{ab}} a_{\beta}q^{\beta}$, then $g^{[\mathfrak{P}_i]} = \sum_{\beta \notin \mathfrak{abP}_i} a_{\beta}q^\beta$.
\end{enumerate}
\end{cor}

\begin{proof}
    Left to the reader.
\end{proof}

\begin{prop}\label{P311}
For every $\delta \in \Q_{\geq 0}$, the $\Lambda_{\alpha,I}[1/\alpha]$-Banach module $H^0(\bar{\Sh{M}}_{r,\alpha,I}, \W{})$ admits a slope $\delta$ decomposition which restricts to a slope decomposition on $H^0(\bar{\Sh{M}}_{r, \alpha,I}, \Fil{n}\W{})$ for all $n \in \N$. Moreover, the inclusion $H^0(\bar{\Sh{M}}_{r, \alpha,I}, \Fil{n}\W{})^{\leq \delta} \subset H^0(\bar{\Sh{M}}_{r,\alpha,I}, \W{})^{\leq \delta}$ is an isomorphism for $n$ large enough (depending on $\delta$).
\end{prop}

\begin{proof}
The operator $U$ is compact on the coherent sheaf $\Fil{n}\W{}$, and so by the usual formalism of slope decomposition we have locally on the weight space a slope $\delta$ decomposition 
\[
H^0(\bar{\Sh{M}}_{r,\alpha,I}, \Fil{n}\W{}) = H^0(\bar{\Sh{M}}_{r,\alpha,I}, \Fil{n}\W{})^{\leq \delta} \oplus H^0(\bar{\Sh{M}}_{r,\alpha,I}, \Fil{n}\W{})^{>\delta}.
\]
By Lemma \ref{L3101}, the $U$ operator on $H^0(\bar{\Sh{M}}_{r,\alpha,I}, \W{}/\Fil{n}\W{})$ is divisible by $p^{\delta+1}$ for $n$ large enough. It follows that $H^0(\bar{\Sh{M}}_{r,\alpha,I}, \W{}/\Fil{n}\W{})$ also admits a slope $\delta$ decomposition and that  $H^0(\bar{\Sh{M}}_{r,\alpha,I}, \W{}/\Fil{n}\W{})^{\leq \delta} = 0$.
\end{proof}

\subsection{Hecke operators prime to \texorpdfstring{$p$}{p}}

We will not need to define the prime to $p$ Hecke operators on the de Rham interpolation sheaf $\W{}$, it'll be sufficient to define them for the modular sheaf $\w{}$. We refer the reader to \cite[\S8.5]{andreatta2016adic} for the construction and definition of $T_{\ell}$ for any ideal $\ell \nmid pN$. We remark that the construction in loc. cit. works integrally as well.

Let $f_{\mathfrak{c}} \colon \Mbar{r}{\mathfrak{c}} \to \mathfrak{W}_{p}$ be the structure map. 

\begin{prop}
    For every ideal $\ell \nmid pN$, there exists an operator $T_\ell \colon {f_{\ell\mathfrak{c}}}_*\w{\ell\mathfrak{c}} \to {f_{\mathfrak{c}}}_*\w{\mathfrak{c}}$ which is the Hecke operator at $\ell$.
\end{prop}

\begin{proof}
    See \cite[\S8.5]{andreatta2016adic}.
\end{proof}

For a prime $\ell|N$, we define the $U_{\ell}$ operator as follows. Consider the normalization $\mathfrak{Y}_{\ell} \subset \Mbar{r}{\mathfrak{c}} \times \Mbar{r}{\ell\mathfrak{c}}$ of the Hecke correspondence defined generically as classifying pairs $(A,\iota, \lambda, \psi)$ and $(A', \iota', \lambda', \psi')$ together with an isogeny $\pi_{\ell} \colon A \to A'$ such that $\ker \pi_{\ell} \simeq \Sh{O}_L/\ell$ \'{e}tale locally, and $\ker\pi_{\ell} \cap \psi(\mu_N \otimes \mathfrak{d}^{-1}) = 0$. Then a similar discussion as before yields an operator $U_{\ell} \colon {f_{\ell\mathfrak{c}}}_*\w{\ell\mathfrak{c}} \to {f_{\mathfrak{c}}}_*\w{\mathfrak{c}}$. 

\begin{prop}
    For every prime $\ell|N$, there exists an operator $U_{\ell} \colon {f_{\ell\mathfrak{c}}}_*\w{\ell\mathfrak{c}} \to {f_{\mathfrak{c}}}_*\w{\mathfrak{c}}$.
\end{prop}

For $\ell$ an ideal prime to $pN$, we define the central action. Define a map $S_{\ell} \colon \M{r}{\mathfrak{c}} \to \M{r}{\mathfrak{c}\ell^2}$ as the normalization of the map induced on generic fibres by sending $A \mapsto A \otimes \ell^{-1}$ together with the induced real multiplication, polarization and level structure. As before it is easy to see that there is a morphism $\pi_{\ell}^* \colon S_{\ell}^*\w{\mathfrak{c}\ell^2} \to \w{\mathfrak{c}}$. 

\begin{defn}
The Hecke operator $S_{\ell}$ is defined as
\[
S_{\ell} \colon {f_{\mathfrak{c}\ell}}_*\w{\mathfrak{c}\ell^2} \xrightarrow{S^*_{\ell}} {f_{\mathfrak{c}}}_*S^*_{\ell} \w{\mathfrak{c}\ell^2} \xrightarrow{\pi_{\ell}^*} {f_{\mathfrak{c}}}_*\w{\mathfrak{c}}.
\]
\end{defn}

\subsection{Hecke operators on arithmetic Hilbert modular forms}

The Hecke operators defined above on the sheaf of geometric Hilbert modular forms induce Hecke operators on the sheaf of arithmetic Hilbert modular forms, which agree with the classically defined Hecke operators on automorphic forms using double coset operators. Let us first recall that we have the following diagram of analytic adic spaces. 

\[\begin{tikzcd}
	{\Man{r}^{G, \mathfrak{c}}} & {\Man{r}^{\mathfrak{c}}\times_{\Sh{W}}\Sh{W}^G} & {\Man{r}^{\mathfrak{c}}} \\
	& {\Sh{W}^G_p} & {\Sh{W}_p}
	\arrow["{f_{\mathfrak{c}}}", from=1-3, to=2-3]
	\arrow["f"', from=2-2, to=2-3]
	\arrow["f", from=1-2, to=1-3]
	\arrow["{f_{\mathfrak{c}}}"', from=1-2, to=2-2]
	\arrow["{g_{\mathfrak{c}}}"', from=1-1, to=2-2]
	\arrow["p"', from=1-2, to=1-1]
\end{tikzcd}\]

We recall that $\w{G, \mathfrak{c}} = (p_*f^*\w{})^{\Gamma}[1/p]$, where $\Gamma = \Sh{O}_L^{\times, +}/U_N^2$. The sheaf of overconvergent arithmetic Hilbert modular forms was defined as 
\[
\w{G}[1/p] = \left(\bigoplus_{\mathfrak{c} \in \textrm{Frac}(L)^{(p)}} {g_{\mathfrak{c}}}_*\w{G,\mathfrak{c}}[1/p] \right)/\left(L_{(x\mathfrak{c},\mathfrak{c})}(f) - f \right)_{x \in \textrm{Princ}(L)^{+,(p)}}.
\]

\begin{prop}
    \begin{enumerate}
        \item For $\ell$ an ideal in $\Sh{O}_L$ that is prime to $pN$, the operators $T_{\ell}$ and $S_{\ell}$ commute with the action of $\Gamma$ and induce operators $T_{\ell} \colon {g_{\ell\mathfrak{c}}}_*\w{G, \ell\mathfrak{c}} \to {g_{\mathfrak{c}}}_*\w{G, \mathfrak{c}}$ and $S_{\ell} \colon {g_{\mathfrak{c}\ell^2}}_*\w{G, \mathfrak{c}\ell^2} \to {g_{\mathfrak{c}}}_*\w{G, \mathfrak{c}}$ respectively.
        \item For $\ell$ a prime ideal dividing $N$, the operator $U_{\ell}$ commutes with the action of $\Gamma$, and induces an operator $U_{\ell} \colon {g_{\ell\mathfrak{c}}}_*\w{G, \ell\mathfrak{c}} \to {g_{\mathfrak{c}}}_*\w{G, \mathfrak{c}}$.
    \end{enumerate}
    Moreover, $T_{\ell}$ and $S_{\ell}$ as defined above, commutes with $L(xc,c)$ for any $x \in L^{\times,+}$, with $(x,p) = 1$. Thus they induce endomorphisms $T_{\ell}$ and $S_{\ell}$ of $\w{G}$ which preserve the cuspforms.
\end{prop}

\begin{proof}
    Left to the reader.
\end{proof}

We now explain how to descend the $U_{\mathfrak{P}_i}$ and $V_{\mathfrak{P}_i}$ operators to $\w{G}$. The operator $U_{\mathfrak{P}_i}$ commutes with the $\Gamma$-action as well as the $\mathrm{Princ}(L)^{+,(p)}$-action, and descends to an operator $U_{\mathfrak{P}_i} \colon {g_{x_i^{-1}\mathfrak{P}_i\mathfrak{c}}}_*\w{G, x_i^{-1}\mathfrak{P}_i\mathfrak{c}} \to {g_{\mathfrak{c}}}_*\w{G, \mathfrak{c}}$. However, this operator is dependent on the choice of $x_i$, and hence is not canonical. Similarly, $V_{\mathfrak{P}_i}$ descends to an operator $V_{\mathfrak{P}_i} \colon {g_{x_i\mathfrak{P}_i^{-1}\mathfrak{c}}}_*\w{G, x_i\mathfrak{P}_i^{-1}\mathfrak{c}} \to {g_{\mathfrak{c}}}_*\w{G, \mathfrak{c}}$ which also depends on the choice of $x_i$. We make these canonical in the following way. Let $(v_{\text{un}}, w_{\text{un}})$ be the universal weight on $\Sh{W}^G_p$.

\begin{defn}
    \begin{enumerate}
        \item Define the operator $\mathbf{U}_{\mathfrak{P}_i} \colon {g_{x_i^{-1}\mathfrak{P}_i\mathfrak{c}}}_*\w{G, x_i^{-1}\mathfrak{P}_i\mathfrak{c}} \to {g_{\mathfrak{c}}}_*\w{G, \mathfrak{c}}$ as $\mathbf{U}_{\mathfrak{P}_i} := (\varpi_i/x_i)^{v_{\text{un}}}U_{\mathfrak{P}_i}$.
        \item Define the operator $\mathbf{V}_{\mathfrak{P}_i} \colon {g_{x_i\mathfrak{P}_i^{-1}\mathfrak{c}}}_*\w{G, x_i\mathfrak{P}_i^{-1}\mathfrak{c}} \to {g_{\mathfrak{c}}}_*\w{G, \mathfrak{c}}$ as $\mathbf{V}_{\mathfrak{P}_i} = (x_i/\varpi_i)^{v_{\text{un}}}V_{\mathfrak{P}_i}$.
    \end{enumerate}
\end{defn}

\begin{prop}
    The operators $\mathbf{U}_{\mathfrak{P}_i}$ and $\mathbf{V}_{\mathfrak{P}_i}$ are independent of the choice of $x_i$, and induce endomorphisms $\mathbf{U}_{\mathfrak{P}_i}$ and $\mathbf{V}_{\mathfrak{P}_i}$ of $\w{G}$ preserving cuspforms. Moreover, $\mathbf{U}_{\mathfrak{P}_i} \circ \mathbf{V}_{\mathfrak{P}_i} = \text{id}$, and $\prod_{i=1}^h \mathbf{U}_{\mathfrak{P}_i} = U$, and $\prod_{i=1}^h \mathbf{V}_{\mathfrak{P}_i} = V$.
\end{prop}

\begin{proof}
    Left to the reader.
\end{proof}

\subsection{Overconvergent projection}
In this section we will define the overconvergent projection in families upon studying the cohomology of the complex of $\Sh{O}_{\bar{\Sh{M}}_{r,\alpha,I}}$ sheaves obtained by the connection $\nabla$ which is described as follows. In particular in this section all sheaves are considered over the analytic adic spaces. Hence to simplify notation we will still write $\W{}$ but it is to be understood that this is a sheaf over $\bar{\Sh{M}}_{r,\alpha,I}$. Consider the complex
\begin{equation}
\W{} \xrightarrow{\nabla} \W{} \hat{\otimes} \Omega^1_{\Man{r}/\Lambda_{\alpha,I}} \cdots \xrightarrow{\nabla} \W{} \hat{\otimes} \Omega^g_{\Man{r}/\Lambda_{\alpha,I}}.
\end{equation}
Denote by $\W{\bullet}$ the complex obtained by tensoring the above complex by $\Sh{O}_{\Man{r}}(-D)$ (and the universal derivation on it) where we recall $D$ is the boundary divisor. By Griffiths' transversality we obtain a complex corresponding to the filtration on $\W{}$ as follows.
\begin{equation}
    \Fil{n}\W{} \xrightarrow{\nabla} \Fil{n+1}\W{} \hat{\otimes} \Omega^1_{\Man{r}/\Lambda_{\alpha,I}} \cdots \xrightarrow{\nabla} \Fil{n+g}\W{} \hat{\otimes} \Omega^{g}_{\Man{r}/\Lambda_{\alpha,I}}.
\end{equation}
Denote by $\Fil{n}^{\bullet}\W{}$ the complex obtained by tensoring the above complex with $\Sh{O}_{\Man{r}}(-D)$. By taking quotient of the first complex by the second we obtain a third complex $(\W{}/\Fil{n}\W{})^{\bullet}$ which sits in a short exact sequence of complexes on $\Man{r}$ that gives a long exact sequence of hypercohomology groups.
\begin{align}\label{E303}
\begin{split}
0 &\to H^0_{\textrm{dR}}(\Man{r}, \Fil{n}^{\bullet}\W{}) \to H^0_{\textrm{dR}}(\Man{r}, \W{\bullet}) \to H^0_{\textrm{dR}}(\Man{r}, (\W{}/\Fil{n}\W{})^{\bullet}) \\
&\to H^1_{\textrm{dR}}(\Man{r}, \Fil{n}^{\bullet}\W{}) \to H^1_{\textrm{dR}}(\Man{r}, \W{\bullet}) \to \cdots
\end{split}
\end{align}

\begin{lemma}
The cohomology of the de Rham complex of coherent sheaves $\Fil{n}^{\bullet}\W{}$ can be computed using global sections.
\end{lemma}
\begin{proof}
We recall that if $f \colon \Man{r} \to \Sh{M}^*_{r,\alpha,I}$ is the projection to the minimal compactification, which is an affinoid adic space, $R^if_*\w{}(-D) = 0$. We note that in order to prove the lemma, it will be enough to prove that each sheaf in the complex $\Fil{n}^{\bullet}\W{}$ is acyclic for the direct image functor $f_*$. Because then an injective resolution of $\Fil{n}^{\bullet}\W{}$ will give an acyclic resolution of $f_*\Fil{n}^{\bullet}\W{}$, and since $\Gamma(\Sh{M}^*_{r,\alpha,I}, \cdot)$ is exact, the lemma will follow. 

We first show that $\Fil{n}\W{}(-D)$ is acyclic for $f_*$. By Lemma \ref{L2401}, the sheaf $\Fil{n}\W{}(-D)$ is equipped with a finite filtration such that the graded pieces are finite direct sums of sheaves of cuspforms. Thus the graded pieces of the filtration are acyclic for $f_*$ by what we just recalled above. Then a simple spectral sequence argument proves that $\Fil{n}\W{}(-D)$ is $f_*$-acyclic. Moreover, using the Kodaira--Spencer isomorphism, this same proof shows that $\Fil{n+i}\W{}(-D)\hat{\otimes}{\Omega}^i_{\Man{r}/\Lambda_{\alpha,I}}$ is also $f_*$-acyclic for all $0 \leq i \leq g$.
\end{proof}

\begin{lemma}\label{L619}
Let $t_L = \sum \sigma$ be the generator of the parallel weights. There exists an exact sequence 
\[
0 \to H^0(\Man{r}, \wc{k+2t_L}(-D)) \xrightarrow{i} H^g_{\textrm{dR}}(\Man{r}, \Fil{n}^{\bullet}\W{}) \to \coker{i} \to 0
\]
where $i$ is $U$-equivariant and $\coker{i}$ is killed by $\prod_{\sigma}\prod_{i=0}^{n+g-1}(u_{\sigma} - i)$.
\end{lemma}

\begin{proof}
We first note that \[H^g_{\textrm{dR}}(\Man{r}, \Fil{n}^{\bullet}\Wc{k}) = \frac{H^0(\Man{r}, \Fil{n+g}\Wc{k+2t_L}(-D))}{\nabla H^0(\Man{r}, \Fil{n+g-1}\Wc{k}(-D)\hat{\otimes}\Omega^{g-1}_{\Man{r}/\Lambda_{\alpha,I}})}\]
The $U$-equivariant inclusion of $H^0(\Man{r}, \wc{k+2t_L}(-D))$ inside $H^0(\Man{r}, \Fil{n+g}\Wc{k+2t_L}(-D))$ induces a map  $H^0(\Man{r}, \wc{k+2t_L}(-D)) \to  H^g_{\textrm{dR}}(\Man{r}, \Fil{n}^{\bullet}\Wc{k})$ which is an inclusion as can be seen from the local description below. We are left to understand the cokernel of the inclusion, and in particular to show that the cokernel is killed by $\prod_{\sigma}\prod_{i=0}^{n+g-1}(u_{\sigma} - i)$. 

The proof relies on the local description of the connection (\ref{eq:11}). Choosing a numbering $\sigma \colon \{1, \dots, g\} \simeq \Sigma$, we first write the sheaf $\Fil{n+g-1}\Wc{k}\hat{\otimes}\Omega^{g-1}_{\Man{r}/\Lambda_{\alpha,I}}$ on local coordinates  as 
\[
\Fil{n+g-1}\Wc{k}\hat{\otimes}\Omega^{g-1}_{\Man{r}/\Lambda_{\alpha,I}} = \bigoplus_{i} \wc{k}^{\leq n+g-1}[V_1, \dots, V_g]\D\hat{X_i}.
\]
Here $\D\hat{X_i}$ corresponds via Kodaira--Spencer to a generator of $\omega_{\Sh{A}}^{2(t_L - \sigma_i)}$ and the superscript $\leq n+g-1$ denotes we take the polynomials in $V_j$'s of degree at most $n+g-1$. 

The map $\nabla \colon \Fil{n+g-1}\Wc{k}(-D)\hat{\otimes}\Omega^{g-1}_{\Man{r}/\Lambda_{\alpha,I}} \to \Fil{n+g}\Wc{k+2t_L}(-D)$ can be described as the twist by $\Sh{O}_{\Man{r}}(-D)$ of a map 
\begin{equation*}
    \bigoplus_i\wc{k}^{\leq n+g-1}[V_1, \dots, V_g]\D\hat{X_i} \xrightarrow{\nabla} \wc{k+2t_L}^{\leq n+g}[V_1, \dots, V_g] .
\end{equation*}
that can be described using formula (\ref{eq:11}). In particular, the image of $\nabla$ consists of polynomials in $V_i$ of positive total degree and hence the map \[\wc{k+2t_l}(-D) \to \frac{\Fil{n+g}\Wc{k+2t_L}(-D)}{\nabla\Fil{n+g-1}\Wc{k}(-D)\hat{\otimes}\Omega^{g-1}_{\Man{r}/\Lambda_{\alpha,I}}}\]
is injective. Thus taking global sections we get $H^0(\Man{r}, \wc{k+2t_L}(-D)) \xhookrightarrow{} H^g_{\textrm{dR}}(\Man{r}, \Fil{n}^{\bullet}\Wc{k})$.

We prove the claim about the annihilator of $\coker{i}$ by induction on $n$, the base case being $n = 1-g$. For $n = 1-g$ we have  $\Fil{1}\Wc{k+2t_L}/(\nabla\wc{k} + \wc{k+2t_L}) \simeq \oplus_i \wc{k+2t_L}V_i/u_i\wc{k+2t_L}V_i$. This proves the base case. We have a diagram as follows with exact rows.
\[
\begin{tikzcd}
0 \arrow[r] & \Fil{n+g-2}\Wc{k}\hat{\otimes}\Omega^{g-1} \arrow[d, "\nabla"] \arrow[r] & \Fil{n+g-1}\Wc{k}\hat{\otimes}\Omega^{g-1} \arrow[d, "\nabla"] \arrow[r] & \Gr{n+g-1}\Wc{k} \hat{\otimes} \Omega^{g-1} \arrow[d, "\nabla"] \arrow[r] & 0 \\
0 \arrow[r] & \Fil{n+g-1}\Wc{k+2t_L} \arrow[r] & \Fil{n+g}\Wc{k+2t_L} \arrow[r] & \Gr{n+g}\Wc{k+2t_L} \arrow[r] & 0
\end{tikzcd}
\]
To complete the induction, we need to understand the connection on the graded pieces. Letting $k(1+\beta_nZ)$ be a local generator of $\wc{k}$, the map $\nabla$ on the graded pieces can be described as follows.
\begin{align*}
    \oplus \wc{k}^{n+g-1}[V_1, \dots, V_g]\D\hat{X}_i &\xrightarrow{\nabla} \wc{k+2t_L}^{n+g}[V_1, \dots, V_g] \\
    k(1+\beta_nZ)\prod_i V_{i}^{n_i} \D\hat{X}_j &\mapsto (u_j - n_j)(k+2t_L)(1+\beta_nZ)V_j\prod_i V_i^{n_i}
\end{align*}
This shows that the cokernel of $\nabla \colon \Gr{n+g-1}\Wc{k}\hat{\otimes}\Omega^{g-1}_{\Man{r}/\Lambda_{\alpha,I}} \xrightarrow{} \Gr{n+g}\Wc{k+2t_L}$ is annihilated by $\prod_{\sigma}\prod_{i=0}^{n+g-1}(u_{\sigma} - i)$. The lemma then follows by applying Snake lemma to the diagram above and by the induction hypothesis.
\end{proof}

\begin{lemma}
For $\delta \in \Q_{\geq 0}$, there exists $n \geq 0$ such that the map 
\[H^i_{\textrm{dR}}(\Man{r}, \Fil{n}^{\bullet}\W{})^{\leq \delta} \to H^i_{\textrm{dR}}(\Man{r}, \W{\bullet})^{\leq \delta}
\]
is an isomorphism.
\end{lemma}

\begin{proof}
Arguing as in Proposition \ref{P311}, the sheaf $H^i_{\textrm{dR}}(\Man{r}, (\Wc{k}/\Fil{n}\Wc{k})^{\bullet})$ admits slope decomposition locally on the weight space. Moreover, by the same proposition \hfill \\$H^i_{\textrm{dR}}(\Man{r}, (\Wc{k}/\Fil{n}\Wc{k})^{\bullet})^{\leq \delta} = 0$ for large enough $n$. The lemma then follows from the long exact sequence (\ref{E303}).
\end{proof}

\begin{defn}\label{D3301}
For $\delta \geq 0$, let $n$ be as in the above lemma. Let $\lambda = \prod_{\sigma} \prod_{i=0}^{n+g-1} (u_{\sigma} - i)$. For the finite slope $\delta \geq 0$, define the overconvergent projection in families to be the map induced by the isomorphisms as follows.
\begin{align*}
H^{\dagger} \colon H^g_{\textrm{dR}}(\Man{r}, \Wc{k}^{\bullet})^{\leq \delta}\otimes \Lambda_{\alpha,I}[\lambda^{-1}] &\xrightarrow{\sim} H^g_{\textrm{dR}}(\Man{r}, \Fil{n}^{\bullet}\Wc{k})^{\leq \delta}\otimes \Lambda_{\alpha,I}[\lambda^{-1}] \\
&\xrightarrow{\sim} H^0(\Man{r}, \wc{k+2t_L}(-D))^{\leq \delta}\otimes \Lambda_{\alpha,I}[\lambda^{-1}]
\end{align*}
\end{defn}

%% file: L-function.tex
\section{Twisted triple product \texorpdfstring{$p$}{p}-adic \texorpdfstring{$L$}{L}-function}

In this section we finally construct the twisted triple product $p$-adic $L$-function associated to finite slope families of Hilbert modular forms. We note that there has recently been a construction of these $p$-adic $L$-functions associated to Hida families due to I. Blanco-Chac\'{o}n and M. Fornea \cite{michele}. We follow their article to describe the generalities of the setting.

Let $L/F$ be a quadratic extension of totally real number fields, with $[F:\Q] = d$. Assume $p$ is unramified in $L$. Let $\Sigma_L$ and $\Sigma_F$ be their respective set of embeddings into a common finite extension $K$ of $\Q_p$ where both $L$ and $F$ are split. Let $G_L = \Res_{L/\Q}\mathbf{GL}_{2,L}$ and $G_F = \Res_{F/\Q}\mathbf{GL}_{2,F}$. Let $(v, n) \in \Z[\Sigma_L] \times \Z$ and $(w, m) \in \Z[\Sigma_F] \times \Z$ be two classical weights.

Consider primitive eigenforms $g \in S^{G_L}({N_g}, (v,n), \bar{\Q})$ and $f \in S^{G_F}({N_f}, (w,m), \bar{\Q})$ generating irreducible cuspidal automorphic representations $\pi, \sigma$ of $G_L(\A{}{})$ and $G_F(\A{}{})$ respectively. Here $S^{G_L}({N_g}, (v,n), \bar{\Q})$ denotes the space of arithmetic Hilbert cuspforms of level $N_g$ and weight $(v,n)$ with $\bar{\Q}$ coefficients, and similarly for $S^{G_F}({N_f}, (w,m), \bar{\Q})$. Denote by $\pi^u, \sigma^u$ their unitarizations and define a representation of $\GL{2}(\A{L \times F}{})$ by $\Pi = \pi^u \otimes \sigma^u$. Let $\rho \colon \Gal{F} \to S_3$ be the homomorphism mapping the absolute Galois group of $F$ to the symmetric group over 3 elements associated to the \'{e}tale cubic algebra $L \times F/F$. Let $G_{L \times F} = \Res_{L\times F/F}\mathbf{GL}_{2,L\times F}$. The $L$-group ${}^L(G_{L\times F})$ is given by the semi-direct product of $\hat{G} \rtimes \Gal{F}$, where $\Gal{F}$ acts on $\hat{G} = \GL{2}(\C{})^{\times 3}$ through $\rho$. Assume the central character of $\Pi$ is trivial when restricted to $\A{F}{\times}$.

\begin{defn}
The twisted triple product $L$-function associated with the unitary automorphic representation $\Pi$ is given by the Euler product 
\[
L(s, \Pi, r) = \prod_{v} L_{v}(s, \Pi_v, r)^{-1}
\]
where $\Pi_v$ is the local representation at the place $v$ of $F$ appearing in the restricted tensor product decomposition $\Pi = \otimes'_v \Pi_v$, and representation $r$ gives the action of the $L$-group of $G_{L\times F}$ on $\C{2}\otimes \C{2} \otimes \C{2}$ which restricts to the natural $8$-dimensional representation of $\hat{G}$ and for which $\Gal{F}$ acts via $\rho$ permuting the vectors.
\end{defn}

Let $D_{/F}$ be a quarternion algebra. We denote by $\Pi^D$ the irreducible unitary cuspidal automorphic representation of $D^{\times}(\A{L\times F}{})$ associated with $\Pi$ by the Jacquet--Langlands correspondence when it exists. For a vector $\phi \in \Pi^D$ one defines its period integral as
\[
I^D(\phi) = \int_{[D^{\times}(\A{F}{})]} \phi(x) \D x
\]
where $[D^{\times}(\A{F}{})] = \A{F}{\times}D^{\times}(F)\backslash D^{\times}(\A{F}{})$. For simplicity of notation we write $I(\phi)$ to denote the period integral when $D = M_2(F)$.

\begin{theorem}\label{T3301}
Let $\eta \colon \A{F}{\times} \to \C{\times}$ be the quadratic character attached to $L/F$ by class field theory. Then the following are equivalent:
\begin{enumerate}
    \item The central $L$-value $L(\frac{1}{2}, \Pi, r)$ does not vanish, and for every place $v$ of $F$ the local $\epsilon$-factor satisfies $\epsilon_v(\frac{1}{2}, \Pi_v, r)\eta_v(-1) = 1$.
    \item There exists a vector $\phi \in \Pi$, called a test vector, whose period integral $I(\phi)$ does not vanish.
\end{enumerate}
\end{theorem}

\begin{proof}
This relies on Ichino's formula \cite[Theorem 1.1]{Ichino}. For a complete proof see \cite[Theorem 3.2]{michele}.
\end{proof}

\begin{defn}
Let $(v, n) \in \Z[\Sigma_L]\times \Z$ and $(w, m) \in \Z[\Sigma_F] \times \Z$ be two classical weights. We say the pair $(v, n)$ and $(w, m)$ are $F$-dominated if there exists $r \in \N[\Sigma_L]$ such that $w = (v + r)_{|F}$ and $mt_F = {nt_L}_{|F}$ (or equivalently, $m = 2n$).
\end{defn}

\begin{rem} 
In their article \cite{michele}, the authors give sufficient criterion for the equivalent conditions of Theorem \ref{T3301} to hold. If the weights $(v, n)$ and $(w,m)$ of $g$ and $f$ respectively are $F$-dominated, then the local $\epsilon$-factor at the archimedean places satisfy the hypothesis of the theorem. Moreover the same is true of the finite places if we assume that $\Nm_{L/F}(N_g)\cdot d_{L/F}$ and $N_f$ are coprime, and that every finite prime dividing $N_f$ splits in $L$.
\end{rem}

\begin{ass}
Assume that the equivalent conditions of Theorem \ref{T3301} are satisfied. Also assume that the weights $(v, n), (w, m)$ of $g, f$ respectively are $F$-dominated.
\end{ass}

Let $\mathfrak{A}$ be an ideal in $\Sh{O}_F$. Denote by $K_{11}(\mathfrak{A})$ the following open compact group.
\[
K_{11}(\mathfrak{A}) = \left\{ \begin{pmatrix}
    a & b \\
    c & d
\end{pmatrix} \in \GL{2}(\hat{\Sh{O}}_F)\,|\, a\equiv d\equiv 1, c\equiv 0 \text{ mod } \mathfrak{A} \right\}.
\]

Let $\mathfrak{J}$ be the element
\[
{\begin{pmatrix}
-1 & 0 \\
0 & 1
\end{pmatrix}}^{\Sigma_F} \in \GL{2}(\R)^{\Sigma_F}.
\]
For any $h \in \sigma^u$, we define $h^{\mathfrak{J}} \in \sigma^u$ to be the vector obtained by right translation $h^{\mathfrak{J}}(x) = h(x\mathfrak{J})$. 

Let $\delta$ be the Shimura--Maass differential operator \cite[1956]{michele} on nearly holomorphic cuspforms. 

The natural inclusion $\GL{2}(\A{F}{}) \to \GL{2}(\A{L}{})$ defines by composition a \emph{diagonal restriction} map \[\zeta^* \colon S^{G_L}(K_{11}(\mathfrak{A}\Sh{O}_L), (v,n), \C{}) \to S^{G_F}(K_{11}(\mathfrak{A}), (v_{|F}, 2n), \C{}).\]

\begin{defn}
Let $\zeta \colon M(\mu_{\mathfrak{A}}, \mathfrak{c}) \to M(\mu_{\mathfrak{A}\Sh{O}_L}, \mathfrak{c}\Sh{O}_L\mathfrak{d}^{-1}_{L/F})$ be the map given on points by sending an abelian variety $A \mapsto A \otimes_{\Sh{O}_F} \Sh{O}_L$ \cite[\S4.1.2]{michele}. Here $\mathfrak{d}_{L/F}$ is the different ideal of the extension $L/F$. 
\end{defn}

Then as shown in loc. cit. 
$\zeta$ induces the map on Shimura varieties given by the inclusion $\GL{2}(\A{F}{}) \to \GL{2}(\A{L}{})$ as above. 

\begin{lemma}\label{L77}
  For $A/\Spf{R}$ parametrized by a point of $\Ig{n}{r}$, the trace map $\text{Tr}_{L/F} \colon \mathrm{H}^{\sharp}_{A}\otimes_{\Sh{O}_F} \Sh{O}_L \to \mathrm{H}^{\sharp}_{A}$ preserves marked sections and the kernel of the marked splitting.
\end{lemma}

\begin{proof}
    The first statement follows from the functoriality of the $\dlog$ map since $H^{\vee}_n(A \otimes_{\Sh{O}_F} \Sh{O}_L) = H_n^{\vee}(A)\otimes_{\Sh{O}_F}\Sh{O}_L$. The second statement follows from the fact that the kernel of the splitting is the kernel of Verschiebung, and the trace commutes with Verschiebung.
\end{proof}

\begin{lemma}\label{L3401}
Let $r \in \N[\Sigma_L]$ be such that $w = (v+r)_{|F}$. Then there is an $\Sh{O}_F$ ideal $\mathfrak{A}$ supported on a subset of the prime factors of $N_f\cdot\Nm_{L/F}(N_g)\cdot d_{L/F}$ such that a test vector $\phi$ as in the statement of Theorem \ref{T3301} can be chosen to be of the form $\phi = (\delta^r(\breve{g}))^u \otimes (\breve{f}^{\mathfrak{J}})^u$ for $\breve{g} \in S^{G_L}(K_{11}(\mathfrak{A}\Sh{O}_L), (v,n), \bar{\Q})$ and $\breve{f} \in S^{G_F}(K_{11}(\mathfrak{A}), (w,m), \bar{\Q})$. The cuspforms $\breve{g}, \breve{f}$ are eigenforms for all Hecke operators outside $N_f\cdot\Nm_{L/F}(N_g)\cdot d_{L/F}$ with the same eigenvalues of $g$ and $f$ respectively. Moreover, in this case the period integral can be expressed as
\[
I(\phi) = \int_{[\GL{2}(\A{F}{})]}(\delta^r\breve{g})^u \otimes (\breve{f}^{\mathfrak{J}})^u\D x = \langle \zeta^*(\delta^r \breve{g}), \breve{f}^*\rangle
\]
where $\breve{f}^*$ is the cuspform in $S^{G_F}(K_{11}(\mathfrak{A}), (w,m), \bar{\Q})$ whose adelic $q$-expansion coefficients are complex conjugates of those of $\breve{f}$, i.e. $\breve{f}^* = c\cdot\breve{f}|w_N$, where $w_N$ is the Atkin-Lehner operator, and $c \neq 0$ is a constant.
\end{lemma}

\begin{proof}
\cite[Lemma 3.4]{michele}.
\end{proof}

\begin{prop}
    The value $L^{\text{alg}}(g,f) := \frac{\langle \zeta^*(\delta^r \breve{g}), \breve{f}^*\rangle}{\langle f^*, f^*\rangle}$ is algebraic. We have
    \[
    L\left(\frac{1}{2}, \Pi, r\right) = (\ast)\cdot|L^{\text{alg}}(g,f)|^2
    \]
    where $(\ast)$ is a constant which can be ensured to be non-zero by suitably choosing the test vectors $\breve{g}, \breve{f}$.
\end{prop}

\begin{proof}
    The first statement is \cite[Proposition 3.5]{michele}. The second statement follows from the proof of Theorem \ref{T3301} using Ichino's formula.
\end{proof}

\begin{ass}\label{A710}
Let $g, f$ be as above with weights $(v,n)$ and $(w,m)$ respectively that are $F$-dominated. Let $K$ be a finite extension of $\Q_p$ containing all the Hecke eigenvalues of $g$ and $f$, and such that $L$ is split in $K$. Assume $f$ has finite slope $\leq a$ at a rational prime $p$ not dividing $N_gN_f$, and unramified in $L$. Let $\breve{g}, \breve{f}$ be as in Lemma \ref{L3401}, which by \cite[Proposition 6]{miyake} are $K$-linear combinations of the oldforms coming from $g$ and $f$ respectively. Let $\omega_g, \omega_f$ be overconvergent families of weight $(v_g, n_g)$ and $(w_f, m_f)$ deforming $g$ and $f$. Let $\breve{\omega}_g, \breve{\omega}_f$ be the overconvergent families of level $K_{11}(\mathfrak{A}\Sh{O}_L)$ and $K_{11}(\mathfrak{A})$ respectively, obtained from $\omega_g$ and $\omega_f$ using the same linear combinations of oldforms as with $\breve{g}$ and $\breve{f}$. Therefore $\breve{\omega}_g$ and $\breve{\omega}_f$ interpolate $\breve{g}$ and $\breve{f}$.  Let $\Lambda_g$ and $\Lambda_f$ be the $p$-adically complete $K$-algebras such that $(v_g, n_g) \in \Sh{W}^{G_L}(\Lambda_g)$ and $(w_f, m_f) \in \Sh{W}^{G_F}(\Lambda_f)$. Let $k_g = 2v_g + n_gt_L$, and $k_f = 2w_f+m_ft_F$. We assume that there exists a continuous character $r \colon (\Sh{O}_L \otimes \Z_p)^{\times} \to \Lambda_g^{\times}$ such that the following diagram is Cartesian.
\[
\begin{tikzcd}
	{\mathrm{Spm}(\Lambda_g \hat{\otimes}_{K}\Lambda_f)} & {\Sh{W}^{G_L}} \\
	{\mathrm{Spm}(\Lambda_f)} & {\Sh{W}^{G_F}}
	\arrow[from=1-1, to=2-1]
	\arrow["{(w_f, m_f)}"', from=2-1, to=2-2]
	\arrow["{(v_g+r, n_g)}", from=1-1, to=1-2]
	\arrow["{(v,n) \mapsto (v_{|F}, 2n)}", from=1-2, to=2-2]
\end{tikzcd}
\]
Assume that $r = r_0\cdot r_{\textrm{an}}$ with $r_0 \in \N[\Sigma_L]$ and $r_{\textrm{an}}$ is an analytic weight satisfying the same conditions as satisfied by the character $s$ in Proposition \ref{P402}. Assume that $k_g$ satisfies the same conditions as satisfied by $k$ in Proposition \ref{P402}, i.e. the analytic part of $k_g$ is given by $k^{\text{an}}_g(x) = \prod_{\sigma \in \Sigma} \exp(u_{\sigma}\log(\sigma(x)))$, for all $x \in (\Sh{O}_L \otimes \Z_p)^{\times}$ and for $u_{\sigma} \in \Lambda_g$ for all $\sigma$. For $\tau \in \Sigma_F$, let $v_{\tau} = u_{\sigma} + u_{\bar{\sigma}}$, where $\sigma, \bar{\sigma}$ are the two extensions of $\tau$ to $L$. Assume that the central character of $\zeta^*\omega_g\cdot \omega_f$ is trivial.
\end{ass}

Let $\Sh{P}$ be the set of primes lying above $p$ in $L$. Denote by $\breve{\omega}_g^{[\Sh{P}]}$ the depletion of $\breve{\omega}_g$ with respect to all the primes in $\Sh{P}$. Under the above assumptions, $\nabla^{r}(\breve{\omega}_g^{[\Sh{P}]}) =  \nabla^{r_{\textrm{an}}}(\nabla^{r_0}(\breve{\omega}_g^{[\Sh{P}]}))$ makes sense, and is a section of $\mathbb{W}^{G_L}_{k_g+2r,p,I}(-D)$ in the sense of Definition \ref{D459}. 

Let $\Lambda_{g,f} = \Lambda_g \hat{\otimes}_{\Lambda^{G_F}} \Lambda_f$. Consider the sheaves $\mathbb{W}_{k_g+2r, p,I}^{G_L}(-D) \otimes_{\Lambda_g} \Lambda_{g,f}$, and $\mathbb{W}_{k_f, p,I}^{G_F}(-D) \otimes_{\Lambda_f} \Lambda_{g,f}$. By Lemma \ref{L77} and functoriality of VBMS, the trace map induces a map $\zeta^* \colon \mathbb{W}_{k_g+2r, p,I}^{G_L}(-D)\otimes_{\Lambda_g} \Lambda_{g,f} \to \mathbb{W}_{k_f, p,I}^{G_F}(-D) \otimes_{\Lambda_f} \Lambda_{g,f}$, which is nothing but the diagonal restriction.

For each choice of representative $\mathfrak{c}_i$ of $\mathrm{Cl}^+(\mathfrak{A})$, the corresponding component $\zeta^*\nabla^r(\breve{\omega}_g^{[\Sh{P}]})_i$ can be viewed as a class in $H^d_{\textrm{dR}}(\bar{\Sh{M}}_{\ell,p,I}^{\mathfrak{c}_i} \otimes \Lambda_{g,f}, \Wc{k_f-2t_F})$ for some suitable $\ell$ determined by Proposition \ref{P402}, and $I \subset [0,1]$. Here we recall that $d = [F:\Q]$. Let $\zeta^*\nabla^r(\breve{\omega}_g^{[\Sh{P}]})^{\leq a}_i$ be its projection onto the slope ``$\leq a$" part $H^d_{\textrm{dR}}(\bar{\Sh{M}}_{\ell,p,I}^{\mathfrak{c}_i} \otimes \Lambda_{g,f}, \Wc{k_f-2t_F})^{\leq a}$ of the de Rham cohomology group, possible by the spectral theory of the $U$ operator \cite[Appendice B]{Andreatta2018leHS}. Then inverting $\lambda = \prod_{\tau \in \Sigma_F}\prod_{i = 0}^N ({v}_{\tau} - i)$ for $N \gg 0$, it can be realized in the space $H^0(\bar{\Sh{M}}_{\ell, p, I}^{\mathfrak{c}_i} \otimes \Lambda_{g,f}[\lambda^{-1}], \wc{k_f})^{\leq a}$ of slope $\leq a$ using the overconvergent projection $H^{\dagger}$ of Definition \ref{D3301}. Since all of these operations commute with the action of $\Gamma$ and $\text{Princ}(L)^{+,(p)}$, the sections $H^{\dagger}(\zeta^*\nabla^r(\breve{\omega}_g^{[\Sh{P}]})^{\leq a}_i)$ together define an arithmetic overconvergent Hilbert modular form. Let us call this form $H^{\dagger, \leq a}(\zeta^*\nabla^r\breve{\omega}_g^{[\Sh{P}]})$, which is a section of the $\mathrm{Spm}(\Lambda_{g,f}[\lambda^{-1}])$ sheaf $\mathfrak{w}_{k_f, p, I}^{G_F}(-D)^{\leq a}$ in the sense of Definition \ref{D358}.


Let $\mathrm{Spm}(A)$ be a connected affinoid of $\mathrm{Spm}(\Lambda_{g,f}[\lambda^{-1}])$. Consider the $A$-Banach module 
\[
M = H^0(\mathrm{Spm}(A), \mathfrak{w}_{k_f, p,I}^{G_F}(-D)^{\leq a}).
\]

\begin{defn}
    Define the Hecke algebra $\Sh{H}$ acting on $M$ to be the subring of $\End_A{M}$ generated by the operators $T_{\ell}, S_{\ell}$ for $(\ell,p\mathfrak{A}) = 1$, $U_{\ell}$ for $\ell|\mathfrak{A}$ a prime, and $\mathbf{U}_{\mathfrak{P}}$ for every prime $\mathfrak{P}|p$ in $F$.
\end{defn}

We note that for any element $f \in M$, the adelic $q$-expansion coefficient $c_p(y, f)$ makes sense for any $y \in \hat{\Sh{O}}_FF_{\infty,+}^{\times}$. This is because these coefficients only depend on the $q$-expansions of $f$ at appropriate cusps as discussed in \S\ref{S123}. Therefore, one can define a Hecke equivariant map
\begin{align*}
    M &\to \Sh{H}^{\vee} \\
    f &\mapsto [T\mapsto c_p(1, f|T)].
\end{align*}

The open affinoid $\mathrm{Spm}(\Sh{H})$ of the eigenvariety is finite and generically \'{e}tale over the weight space $\mathrm{Spm}(A)$. Therefore the trace map $\mathrm{Tr} \colon \Sh{H} \to \Sh{H}^{\vee}$ is an isomorphism over the total ring of fractions $\mathbb{K}$ of $A$.
Using this isomorphism we can define the Petersson inner product on $M$ as follows.

\begin{defn} 
Define the Petersson inner product as the pairing
\[
\langle \cdot, \cdot \rangle \colon M \times M \to \Sh{H}^{\vee}\otimes_A \mathbb{K} \times \Sh{H}^{\vee} \otimes_A \mathbb{K} \xrightarrow{\mathrm{Tr}} \mathbb{K}.
\]
\end{defn}

\begin{prop}\label{P712}
    Let $x \in \mathrm{Spm}(A)$ be a classical weight such that $M_x$ consists entirely of classical forms of slope less than or equal to $a$, and the fibre $\Sh{H}_x$ if finite \'{e}tale over $k(x)$. Then the Hecke equivariant map $M_x \to \Sh{H}^{\vee}_x$ is an isomorphism, and the inner product $\langle \cdot, \cdot \rangle_x$ coincides upto a scalar with the classical Petersson inner product after base change $k(x) \to \C{}$.
\end{prop}

\begin{proof}
    Under the assumption of \'{e}taleness, the fibre $\Sh{H}_x$ equals the Hecke algebra of $M_x$, and the first claim then follows from \cite[Theorem 2.2]{Hida}. For the second claim, we notice that the trace pairing on $\Sh{H}^{\vee}_x$ is perfect. Hence $\langle \cdot, \cdot \rangle_x$ is a Hecke equivariant perfect pairing on $M_x$. The result follows since the classical Petersson inner product is also a Hecke equivariant perfect pairing on $M_x \otimes_{k(x)} \C{}$.
\end{proof}

\begin{defn}\label{D714}
Let $\omega_f^* = \omega_f|w_{\mathfrak{A}}$ where $w_{\mathfrak{A}}$ is the Atkin--Lehner operator. The $p$-adic twisted triple product $L$-function associated to $g$ and $f$ satisfying the assumptions above is defined to be \[\sh{L}^g_p(\breve{\omega}_g, \breve{\omega}_f) =  \frac{\langle H^{\dagger, \leq a}(\zeta^*\nabla^r\breve{\omega}_g^{[\Sh{P}]}), {\breve{\omega}_f^*} \rangle}{\langle {\omega_f^*}, {\omega_f^*} \rangle}.\]
\end{defn}

\begin{cor}
    For a pair of classical weights $\left((x, c), (y, 2c)\right) \in \mathrm{Spm}(\Lambda_{g,f}[\lambda^{-1}])$ factoring through $\mathrm{Spm}(A)$ that are $F$-dominated, i.e. there exists $t \in \N[\Sigma_L]$, such that $(x+t)_{|F} = y$, letting $g_x, \breve{g}_x$ and $f_y, \breve{f}_y$ be the specializations of ${\omega}_g, \breve{\omega}_g$ and $\omega_f, \breve{\omega}_f$ at $(x,c)$ and $(y,2c)$ respectively, we have
    \[
    \sh{L}_p^g(\breve{\omega}_g, \breve{\omega}_f)\left((x,c), (y,2c)\right) = \frac{\langle H^{\dagger, \leq a}(\zeta^*\nabla^t \breve{g}_x^{[\Sh{P}]}), {\breve{f}_y^*} \rangle}{\langle {f_y^*}, {f_y^*} \rangle}.
    \]
\end{cor}

\begin{proof}
    Follows from construction since the projection to slope ``$\leq a$" space, the overconvergent projection, and the diagonal restriction all commute with the specialization map.
\end{proof}

Next we compare our $p$-adic $L$-function with the definition given by Fornea and Blanco-Chacon \cite[Definition 3.8]{michele}.

So let us now assume that $g, f$ are ordinary, and the families $\omega_g, \omega_f$ are Hida families interpolating ordinary $p$-stabilizations $g^{(p)}$ and $f^{(p)}$ respectively. The families $\breve{\omega}_g, \breve{\omega}_f$ are $\Lambda_{g,f}$-adic families interpolating $\breve{g}^{(p)}$ and $\breve{f}^{(p)}$ as in loc. cit.

\begin{prop}
    Fix a pair of classical weights $((x,c), (y,2c)) \in \mathrm{Spm}(A)$ such that there exists $t \in \Z[\Sigma_L]$ such that $y = (x+t)_{|F}$, and such that $(x,c)$ and $(y,2c)$ satisfy the condition of Proposition \ref{P712}. Then letting $g_x, \breve{g}_x$ and $f_y, \breve{f}_y$ be the specializations as above, we have
    \[
    \sh{L}^g_p(\breve{\omega}_g, \breve{\omega}_f)((x,c), (y,2c)) = \frac{\langle e_{n.o}\zeta^*(\theta^t\breve{g}^{[\Sh{P}]}_x), \breve{f}_y^{*(p)}\rangle}{\langle f_y^{*(p)}, f_y^{*(p)}\rangle}.
    \]
    Here $e_{n.o}$ is the ordinary projector $e_{n.o} = \lim_{n\to \infty} U^{n!}$. In particular, it coincides with Definition 3.8 in \emph{\cite{michele}} at these points.
\end{prop}

\begin{proof}
    We first note that since the Petersson product is Hecke equivariant, the special $L$-value $C = \sh{L}^g_p(\breve{\omega}_g, \breve{\omega}_f)((x,c), (y,2c))$ only depends on the projection of $H^{\dagger, \leq a}(\zeta^*\nabla^t \breve{g}_x^{[\Sh{P}]})$ to the $\breve{f}_y^{*(p)}$ isotypic component. In particular, since $\breve{f}_y^{*(p)}$ is ordinary, 
    \[
    C = \frac{\langle e_{n.o}H^{\dagger, \leq 0}(\zeta^*\nabla^t \breve{g}_x^{[\Sh{P}]}), {\breve{f}_y^{*(p)}} \rangle}{\langle {f_y^{*(p)}}, {f_y^{*(p)}} \rangle}.
    \]
    We note that since $\nabla^t\breve{g}_x^{[\Sh{P}]}$ is already realized at a finite step of the filtration on $\Wc{k_f}$, $H^{\dagger}(\zeta^*\nabla^t\breve{g}_x^{[\Sh{P}]})$ makes sense, and $e_{n.o}H^{\dagger}(\zeta^*\nabla^t\breve{g}_x^{[\Sh{P}]}) = e_{n.o}H^{\dagger, \leq 0}(\zeta^*\nabla^t\breve{g}_x^{[\Sh{P}]})$. Now by the proof of Lemma \ref{L619}, $\zeta^*\nabla^t\breve{g}_x^{[\Sh{P}]} = H^{\dagger}(\zeta^*\nabla^t\breve{g}_x^{[\Sh{P}]}) + \nabla(h)$ for some $h \in \Fil{n+d-1}\Wc{k_f}(-D)\hat{\otimes}\Omega^{g-1}_{\Man{r}}$. Let $\psi_{\Frob}$ be the unit root splitting over the ordinary locus. Then, since by Lemma \ref{L77}, the trace commutes with Verschiebung, implying the diagonal restriction commutes with unit root splitting, we have
    \begin{align*}
    H^{\dagger}(\zeta^*\nabla^t\breve{g}_x^{[\Sh{P}]}) = \psi_{\Frob}(H^{\dagger}(\zeta^*\nabla^t\breve{g}_x^{[\Sh{P}]})) &= \psi_{\Frob}(\zeta^*\nabla^t\breve{g}_x^{[\Sh{P}]}) - \psi_{\Frob}(\nabla h) \\
    &= \zeta^*\theta^t\breve{g}_x^{[\Sh{P}]} - \psi_{\Frob}(\nabla h).
    \end{align*}
    Now $\psi_{\Frob}\nabla = \sum_{\tau \in \Sigma_F} \theta_{\tau}$. Since $e_{n.o}\theta_{\tau} = 0$ for all $\tau \in \Sigma_F$, we have 
    \[
    e_{n.o}H^{\dagger, \leq 0}(\zeta^*\nabla^t\breve{g}_x^{[\Sh{P}]}) = \zeta^*\theta^t\breve{g}_x^{[\Sh{P}]}.
    \]
    This proves the claim.
\end{proof}

\subsection{Interpolation formula}
We now prove an interpolation formula relating special values of this $p$-adic $L$-function to classical $L$-values. 

Assume $(x,c)$ and $(y,2c)$ are a pair of $F$-dominated weights as before such that $y = (x+t)_{|F}$ for some $t \in \N[\Sigma_L]$, and such that $(x,c)$ and $(y,2c)$ satisfy the condition of Proposition \ref{P712}. Assume further that the specializations $g_x := {(\omega_{g})}_x$ and $f_y := {(\omega_f)}_y$, which are classical by assumption, are eigenforms for all the Hecke operators $T_{\ell}$, where $\ell \nmid \mathfrak{A}$. Note in particular $\ell$ can be $p$, and we assume $f_y$ has slope $\leq a$. Moreover, assume that the Hecke polynomials for $T_{\mathfrak{p}}$ for any prime $\mathfrak{p}|p$ associated to both $g_x$ and $f_y$ are separable.

To relate the value $\frac{\langle H^{\dagger, \leq a}(\zeta^*\nabla^t\breve{g}_x^{[\Sh{P}]}), \breve{f}_y^*\rangle}{\langle f_y^*, f_y^*}$ to the algebraic part of the classical $L$-value $\frac{\langle \zeta^*(\delta^t \breve{g}_x), \breve{f}_y^*\rangle}{\langle \breve{f}_y^*, \breve{f}_y^*\rangle}$, we will split the computation into two cases, based on whether a prime $\mathfrak{p}|p$ in $F$ splits in $L$ or stays inert. We will use this distinction to simplify the computation of Euler factors while passing from the $[\Sh{P}]$-depletion of $\breve{g}$ to $\breve{g}$.

We begin with some preparatory lemmas that we will repeatedly use in the computation.

\begin{lemma}\label{L716}
    Let $\mathfrak{p}|p$ be a prime in $F$.
    \begin{enumerate}
        \item $\mathbf{U}_{\mathfrak{p}}\nabla^t = p^{t_{\mathfrak{p}}}\nabla^t\mathbf{U}_{\mathfrak{p}}$, where $t_{\mathfrak{p}} = \sum_{\sigma|\mathfrak{p}} t_{\sigma}$.
        \item $\mathbf{U}_{\mathfrak{p}}H^{\dagger} = H^{\dagger}\mathbf{U}_{\mathfrak{p}}$ over any filtered piece of $\Wc{k}$.
    \end{enumerate}
\end{lemma}

\begin{proof}
    \begin{enumerate}
        \item We note that upto multiplication by a constant normalizing factor $p^{x_{\mathfrak{p}}}$ which we can ignore,  $\mathbf{U}_{\mathfrak{p}}V_{\tau} = pV_{\tau}$ if $\tau|\mathfrak{p}$, and $V_{\tau}$ otherwise. By \cite[\S7G]{Hida}, on adelic $q$-expansions $c_p(y, f|\theta^r) = y_p^rc_p(y, f)$. The claim then follows from the formula $c_p(y, f|\U{p}) = c_p(y\varpi_{\mathfrak{p}}, f)$ and equation (\ref{eq:12}), where $\varpi_{\mathfrak{p}}$ is the idele with the uniformizer $p$ in the $\mathfrak{p}$ place and $1$ elsewhere. 
        \item This again follows from equation (\ref{eq:12}), noting that $H^{\dagger}\nabla = 0$. In fact one can work out a formula like \cite[Proposition 3.37]{andreatta2021triple} for the overconvergent projection, and check the identity.
    \end{enumerate}
\end{proof}

\emph{Inert case}: Let $\mathfrak{p}|p$ be a prime in $F$ that stays inert in $L$. Let $\alpha_g, \beta_g$ be the two roots of the Hecke polynomial of $T_{\mathfrak{p}}$ associated to $g_x$, which by assumption are not equal. Let $\breve{g}^{(\mathfrak{p})}_{\bullet}$ be the $\mathfrak{p}$-stabilization of $\breve{g}_x$ with eigenvalue $\bullet$ for $\bullet = \alpha_g, \beta_g$. Remark that $\breve{g}^{[\mathfrak{p}]}_x = \breve{g}^{(\mathfrak{p})[\mathfrak{p}]}_{\bullet}$ for each of the $\mathfrak{p}$-stabilizations. This simply follows from $(1-\Vp{p}\U{p})(1- (\bullet)\Vp{p}) = 1-\Vp{p}\U{p}$. Finally, note that $\breve{g}_x = \frac{\alpha_g\breve{g}^{(\mathfrak{p})}_{\alpha} - \beta_g\breve{g}^{(\mathfrak{p})}_{\beta}}{\alpha_g - \beta_g}$. 

\begin{lemma}
    \[
    H^{\dagger, \leq a}\zeta^*(\nabla^t\st{g}{[\mathfrak{p}]}_x) = H^{\dagger, \leq a}\zeta^*(\nabla^t\st{g}{(\mathfrak{p})[\mathfrak{p}]}_{\bullet}) = \left(1-(\bullet)p^{t_{\mathfrak{p}}}\U{p}^{-1}\right)H^{\dagger, \leq a}\zeta^*(\nabla^t\st{g}{(\mathfrak{p})}_{\bullet}).
    \]
\end{lemma}

\begin{proof}
    We have 
    \[
    H^{\dagger, \leq a}\zeta^*(\nabla^t\st{g}{(\mathfrak{p})[\mathfrak{p}]}_{\bullet}) = H^{\dagger, \leq a}\zeta^*(\nabla^t\st{g}{(\mathfrak{p})}_{\bullet}) - (\bullet)H^{\dagger, \leq a}\zeta^*(\nabla^t\Vp{p}\st{g}{(\mathfrak{p})}_{\bullet}).
    \]
    By Lemma \ref{L716},
    \begin{align*}
    H^{\dagger, \leq a}\zeta^*(\nabla^t\Vp{p}\st{g}{(\mathfrak{p})}_{\bullet}) &= H^{\dagger, \leq a}\zeta^*(p^{t_{\mathfrak{p}}}\Vp{p}\nabla^t\st{g}{(\mathfrak{p})}_{\bullet}) \\
    &= p^{t_{\mathfrak{p}}}\U{p}^{-1}H^{\dagger, \leq a}\zeta^*(\nabla^t\st{g}{(\mathfrak{p})}_{\bullet}).
    \end{align*}
    We note that to obtain the last equality we used the fact that $\Vp{p}$ commutes with diagonal restriction. This follows from the fact that on $\Sym^k H^1_{\text{dR}}$ for any classical weight $k$, $\Vp{p}$ is induced (upto multiplication by $p^{-1}$) by the lift of the Frobenius isogeny at $\mathfrak{p}$. The claim follows.
\end{proof}

\begin{lemma}
    Let $\sh{E}_{\mathfrak{p}}(g_x, T) = (1-p^{t_{\mathfrak{p}}}\alpha_g T^{-1})(1-p^{t_{\mathfrak{p}}}\beta_g T^{-1})$. Then,
    \[
    H^{\dagger, \leq a}\zeta^*(\nabla^t\st{g}{}_x) = \frac{1}{\sh{E}_{\mathfrak{p}}(g_x,\U{p})}H^{\dagger, \leq a}\zeta^*(\nabla^t\st{g}{[\mathfrak{p}]}_x).
    \]
\end{lemma}

\begin{proof}
    Follows from a simple computation using \[H^{\dagger, \leq a}\zeta^*(\nabla^t\st{g}{}_x) = \frac{\alpha_gH^{\dagger, \leq a}\zeta^*(\nabla^t\st{g}{(\mathfrak{p})}_{\alpha}) - \beta_gH^{\dagger, \leq a}\zeta^*(\nabla^t\st{g}{(\mathfrak{p})}_{\beta})}{\alpha_g - \beta_g}.\]
\end{proof}

Let $k = 2y + 2ct_F$. Let $a_{\mathfrak{p}}^*$ be the $T_{\mathfrak{p}}$ eigenvalue of $f_y^*$. Let $\gamma = e_{\breve{f}^*_y}H^{\dagger, \leq a}\zeta^*(\nabla^t\st{g}{}_x)$ be the projection of $H^{\dagger, \leq a}\zeta^*(\nabla^t\st{g}{}_x)$ onto the Hecke eigenspace of $\breve{f}^*_y$. We have the following lemma.

\begin{lemma}
    For $\gamma$ as above, $\langle \mathbf{V}_{\mathfrak{p}}(\gamma), \breve{f}^*_y\rangle = \frac{a_{\mathfrak{p}}^*\chi(\mathfrak{p})}{p^{k-t_F}(q_{\mathfrak{p}}+1)}\langle \gamma, \breve{f}^*_y\rangle$, where $\chi \colon \mathrm{Cl}^+(\mathfrak{A}) \to \bar{\Q}^{\times}$ is the finite part of the central character associated to $f_y$, and $q_{\mathfrak{p}} = \#\Sh{O}_F/\mathfrak{p}$.
\end{lemma}

\begin{proof}
    First note that $T_{\mathfrak{p}}\gamma = a_{\mathfrak{p}}^*\gamma$ by definition. Consider the left action of the double coset operator 
    \[
    \left[K_1(\mathfrak{A})\left(\begin{smallmatrix}
        \varpi_{\mathfrak{p}} & 0 \\ 0 & 1
    \end{smallmatrix}\right)K_1(\mathfrak{A})\right] = \bigsqcup_{j \in \Sh{O}_F/\mathfrak{p}} \begin{pmatrix}
        \varpi_{\mathfrak{p}} & j \\ 0 & 1
    \end{pmatrix}K_1(\mathfrak{A}) \bigsqcup \begin{pmatrix}
        1 & 0 \\ 0 & \varpi_{\mathfrak{p}}
    \end{pmatrix}K_1(\mathfrak{A})
    \]
    on $\gamma$ as 
    \[
    \left[K_1(\mathfrak{A})\left(\begin{smallmatrix}
        \varpi_{\mathfrak{p}} & 0 \\ 0 & 1
    \end{smallmatrix}\right)K_1(\mathfrak{A})\right]\gamma(u) = \sum_{j \in \Sh{O}_F/\mathfrak{p}} \gamma\left(u\left(\begin{smallmatrix}
        \varpi_{\mathfrak{p}} & j \\ 0 & 1
    \end{smallmatrix}\right)\right) + \gamma\left(u \left(\begin{smallmatrix}
        1 & 0 \\ 0 & \varpi_{\mathfrak{p}}
    \end{smallmatrix}\right)\right).
    \]
    Then $T_{\mathfrak{p}}\gamma = p^{y-t_F}\left[K_1(\mathfrak{A})\left(\begin{smallmatrix}
        \varpi_{\mathfrak{p}} & 0 \\ 0 & 1
    \end{smallmatrix}\right)K_1(\mathfrak{A})\right]\gamma$. For simplicity of notation, let $\alpha = \left(\begin{smallmatrix}
        \varpi_{\mathfrak{p}} & 0 \\ 0 & 1
    \end{smallmatrix}\right)$, $\beta = \left(\begin{smallmatrix}
        1 & 0 \\ 0 & \varpi_{\mathfrak{p}}
    \end{smallmatrix}\right)$, and $h_j = \left(\begin{smallmatrix}
        1 & j \\ 0 & 1
    \end{smallmatrix}\right)$. We then have \[\langle h_j\alpha \gamma, \breve{f}^*_y\rangle = \langle \alpha \gamma, h_{-j}\breve{f}^*_y \rangle = \langle \alpha \gamma, \breve{f}^*_y \rangle.\]
    Similarly, writing $\beta = k\alpha k'$ for $k, k' \in K_1(\mathfrak{A})$, and $k^{\iota}$ for the adjugate matrix of $k$, we have \[\langle \beta \gamma, \breve{f}^*_y\rangle = \langle \alpha k' \gamma, k^{\iota}\breve{f}^*_y\rangle = \langle \alpha \gamma, \breve{f}^*_y\rangle.\]
    Now noting that $\mathbf{V}_{\mathfrak{p}}(\gamma) = p^{-y}\left(\begin{smallmatrix}
        \varpi_{\mathfrak{p}}^{-1} & 0 \\ 0 & 1
    \end{smallmatrix}\right)\gamma$, the formula follows from a simple computation.
\end{proof}

Let $\alpha_f^*$ and $\beta_f^*$ be the two roots of the Hecke polynomial of $T_{\mathfrak{p}}$ associated to $f^*_y$. Letting $\gamma_{\alpha}^{(\mathfrak{p})} = \gamma - \beta_f^*\mathbf{V}_{\mathfrak{p}}(\gamma)$ (and similarly for $\gamma_{\beta}^{(\mathfrak{p})}$) be the two $\mathfrak{p}$-stabilizations, we have the following proposition.

\begin{prop}\label{P719}
    Let $E(x,y) = (1-xy^{-1})$. Let $\sh{E}_1(T) = 1 - \frac{a_{\mathfrak{p}}^*\chi(\mathfrak{p})T}{p^{k-t_F}(q_{\mathfrak{p}}+1)}$. Then 
    \[
    \langle H^{\dagger, \leq a}\zeta^*(\nabla^t\st{g}{[\mathfrak{p}]}_x), \breve{f}^*_y\rangle = \left[ \frac{\sh{E}_{\mathfrak{p}}(g_x, \alpha_{f}^*)\sh{E}_1(\beta_f^*)}{E(\beta_f,\alpha_f)} + \frac{\sh{E}_{\mathfrak{p}}(g_x,\beta_f^*)\sh{E}_1(\alpha_f^*)}{E(\alpha_f, \beta_f)}\right]\langle H^{\dagger, \leq a}\zeta^*(\nabla^t\st{g}{}_x), \breve{f}^*_y\rangle.
    \]
\end{prop}

\begin{proof}
    We have $\st{f}{*}_y = (\alpha_{f}^*-\beta_{f}^*)^{-1}(\alpha_{f}^*\st{f}{*(\mathfrak{p})}_{\alpha} - \beta_{f}^*\st{f}{*(\mathfrak{p})}_{\beta})$ for the two $\mathfrak{p}$-stabilizations $\st{f}{*(\mathfrak{p})}_{\alpha}$ and $\st{f}{*(\mathfrak{p})}_{\beta}$ of $\st{f}{*}_y$. The formula then follows by observing that on the $\st{f}{*(\mathfrak{p})}_{\star}$-component, $\U{p}$ acts via $\star$, for $\star = \alpha_{f}^*, \beta_{f}^*$.
\end{proof}

\emph{Split case}: Let $\mathfrak{p}|p$ be a prime in $F$ that splits as $\mathfrak{p}\Sh{O}_L = \mathfrak{p}_1\mathfrak{p}_2$ in $L$. For $i=1,2$, let $\st{g}{(\mathfrak{p}_i)}_{\bullet}$ be the $\mathfrak{p}_i$-stabilization of $\st{g}{}_x$ with eigenvalue $\bullet = \alpha_i, \beta_i$. 

\begin{lemma}\label{L720}
    For $\mathfrak{i} \neq \mathfrak{j}$, we have $U\zeta^*(\nabla^t \Vp{p_i}\st{g}{[\mathfrak{p_j}]}) = 0$ for any $\st{g}{}$. In particular, \[H^{\dagger, \leq a}\zeta^*(\nabla^t\Vp{p_i}\st{g}{}) = H^{\dagger, \leq a}\zeta^*(\nabla^t\Vp{p}\U{p_j}\st{g}{}).\]
\end{lemma}

\begin{proof}
    This is a computation on $q$-expansion done in \cite[Lemma 3.10]{michele}.
\end{proof}

\begin{lemma}\label{L721}
    For any $\mathfrak{p_i}$-stabilization $\st{g}{(\mathfrak{p_i})}_{\bullet}$, we have $H^{\dagger, \leq a}\zeta^*(\nabla^t\st{g}{(\mathfrak{p_i})[\mathfrak{p_j}]}_{\bullet}) = H^{\dagger, \leq a}\zeta^*(\nabla^t\st{g}{[\mathfrak{p_1},\mathfrak{p_2}]}_x)$ for $\mathfrak{i} \neq \mathfrak{j}$. Hence, by taking the appropriate linear combination, we have $H^{\dagger, \leq a}\zeta^*(\nabla^t\st{g}{[\mathfrak{p_j}]}_{x}) = H^{\dagger, \leq a}\zeta^*(\nabla^t\st{g}{[\mathfrak{p_1},\mathfrak{p_2}]}_{x})$ for any $\mathfrak{j} = 1, 2$.
\end{lemma}

\begin{proof}
    We have the following.
    \begin{align*}
        H^{\dagger, \leq a}\zeta^*(\nabla^t\st{g}{[\mathfrak{p_1}, \mathfrak{p_2}]}_x) &= H^{\dagger, \leq a}\zeta^*(\nabla^t(1-\Vp{p_i}\U{p_i})\st{g}{(\mathfrak{p_i})[\mathfrak{p_j}]}_{\bullet}) \\
        &= H^{\dagger, \leq a}\zeta^*(\nabla^t\st{g}{(\mathfrak{p_i})[\mathfrak{p_j}]}_{\bullet}) - (\bullet)H^{\dagger, \leq a}\zeta^*(\nabla^t\Vp{p_i}\st{g}{(\mathfrak{p_i})[\mathfrak{p_j}]}_{\bullet}) \\
        &= H^{\dagger, \leq a}\zeta^*(\nabla^t\st{g}{(\mathfrak{p_i})[\mathfrak{p_j}]}_{\bullet}).
    \end{align*}
    Here the last equality follows from Lemma \ref{L720}.
\end{proof}

\begin{lemma}\label{L723}
    We have 
    \[
    H^{\dagger, \leq a}\zeta^*(\nabla^t\st{g}{[\mathfrak{p_1}, \mathfrak{p_2}]}_{x}) = \frac{\prod\limits_{\bullet,\star \in \{\alpha,\beta\}}1-p^{t_{\mathfrak{p}}}\bullet_1\star_2\U{p}^{-1}}{1-p^{2t_{\mathfrak{p}}}\alpha_1\alpha_2\beta_1\beta_2\U{p}^{-2}} = H^{\dagger, \leq a}\zeta^*(\nabla^t\st{g}{}_{x}).
    \]
\end{lemma}

\begin{proof}
    We have the following chain of identities.
    \begin{align*}
        H^{\dagger, \leq a}\zeta^*(\nabla^t\st{g}{[\mathfrak{p_i}]}_{x}) &= H^{\dagger, \leq a}\zeta^*(\nabla^t\st{g}{(\mathfrak{p_i})[\mathfrak{p_i}]}_{\bullet}) \\
        &= H^{\dagger, \leq a}\zeta^*(\nabla^t\st{g}{(\mathfrak{p_i})}_{\bullet}) - (\bullet)H^{\dagger, \leq a}\zeta^*(\nabla^t\Vp{p_i}\st{g}{(\mathfrak{p_i})}_{\bullet}) \\
        &= H^{\dagger, \leq a}\zeta^*(\nabla^t\st{g}{(\mathfrak{p_i})}_{\bullet}) - (\bullet)H^{\dagger, \leq a}\zeta^*(\nabla^t\Vp{p}\U{p_j}\st{g}{(\mathfrak{p_i})}_{\bullet}) \quad (\text{by Lemma \ref{L720}}) \\
        &= H^{\dagger, \leq a}\zeta^*(\nabla^t\st{g}{(\mathfrak{p_i})}_{\bullet}) - (\bullet)p^{t_{\mathfrak{p}}}\U{p}^{-1}H^{\dagger, \leq a}\zeta^*(\nabla^t\U{p_j}\st{g}{(\mathfrak{p_i})}_{\bullet}) \\
        &= \left[ 1 - (\bullet)p^{t_{\mathfrak{p}}}(\alpha_j+\beta_j)\U{p}^{-1}\right]H^{\dagger, \leq a}\zeta^*(\nabla^t\st{g}{(\mathfrak{p_i})}_{\bullet}) + (\bullet)p^{t_{\mathfrak{p}}}\alpha_j\beta_j\U{p}^{-1}H^{\dagger, \leq a}\zeta^*(\nabla^t\Vp{p_j}\st{g}{(\mathfrak{p_i})}_{\bullet}) \\
        &= \left[1-(\bullet)p^{t_{\mathfrak{p}}}(\alpha_j+\beta_j)\U{p}^{-1}+(\bullet)^2p^{2t_{\mathfrak{p}}}\alpha_j\beta_j\U{p}^{-2}\right]H^{\dagger, \leq a}\zeta^*(\nabla^t\st{g}{(\mathfrak{p_i})}_{\bullet}) \\
        &= (1- (\bullet)p^{t_{\mathfrak{p}}}\alpha_j\U{p}^{-1})(1-(\bullet)p^{t_{\mathfrak{p}}}\beta_j\U{p}^{-1})H^{\dagger, \leq a}\zeta^*(\nabla^t\st{g}{(\mathfrak{p_i})}_{\bullet}).
    \end{align*}
    Now by Lemma \ref{L721}, since $H^{\dagger, \leq a}\zeta^*(\nabla^t\st{g}{[\mathfrak{p_i}]}_{\bullet}) = H^{\dagger, \leq a}\zeta^*(\nabla^t\st{g}{[\mathfrak{p_1},\mathfrak{p_2}]}_{x})$, we have
    \begin{align*}
    &H^{\dagger, \leq a}\zeta^*(\nabla^t\st{g}{}_{x}) \\
    &= \frac{1}{\alpha_i-\beta_i}\left[\frac{\alpha_i}{(1- p^{t_{\mathfrak{p}}}\alpha_i\alpha_j\U{p}^{-1})(1-p^{t_{\mathfrak{p}}}\alpha_i\beta_j\U{p}^{-1})} - \frac{\beta_i}{(1- p^{t_{\mathfrak{p}}}\beta_i\alpha_j\U{p}^{-1})(1-p^{t_{\mathfrak{p}}}\beta_i\beta_j\U{p}^{-1})}\right]\\
    &\quad \quad \quad \quad \quad \quad \quad \times H^{\dagger, \leq a}\zeta^*(\nabla^t\st{g}{[\mathfrak{p_1},\mathfrak{p_2}]}_{x})\\
    &= \frac{1-p^{2t_{\mathfrak{p}}}\alpha_1\alpha_2\beta_1\beta_2\U{p}^{-2}}{\prod\limits_{\bullet,\star \in \{\alpha,\beta\}}1-p^{t_{\mathfrak{p}}}\bullet_1\star_2\U{p}^{-1}}H^{\dagger, \leq a}\zeta^*(\nabla^t\st{g}{[\mathfrak{p_1},\mathfrak{p_2}]}_{x}).
    \end{align*}
\end{proof}

\begin{prop}
    Let $\sh{E}_{\mathfrak{p}}(g_x, T) = \prod\limits_{\bullet,\star \in \{\alpha,\beta\}}(1-p^{t_{\mathfrak{p}}}\bullet_1\star_2T^{-1})$ for a split prime $\mathfrak{p}$. Let $\sh{E}_{0,\mathfrak{p}}(g_x, T) = 1-p^{2t_{\mathfrak{p}}}\alpha_1\alpha_2\beta_1\beta_2T^{-2}$ With assumptions as in Proposition \ref{P719}, we have
    \[
    \langle H^{\dagger, \leq a}\zeta^*(\nabla^t\st{g}{[\mathfrak{p_1}, \mathfrak{p_2}]}_{x}) , \st{f}{*}_y\rangle = \left[\frac{\sh{E}_{\mathfrak{p}}(g_x, \alpha_f^*)\sh{E}_1(\beta_f^*)}{E(\beta_f,\alpha_f)\sh{E}_{0,\mathfrak{p}}(g_x,\alpha_f^*)} + \frac{\sh{E}_{\mathfrak{p}}(g_x,\beta_f^*)\sh{E}_1(\alpha_f^*)}{E(\alpha_f,\beta_f)\sh{E}_{0,\mathfrak{p}}(g_x, \beta_f^*)} \right]\langle H^{\dagger, \leq a}\zeta^*(\nabla^t\st{g}{}_{x}), \st{f}{*}_y\rangle.
    \]
\end{prop}

\begin{proof}
    This is similar to the proof of Proposition \ref{P719}. Write $\st{f}{*}_y = (\alpha_f^*-\beta_f^*)^{-1}(\alpha_f^*\st{f}{*(\mathfrak{p})}_{\alpha} - \beta_{f}^*\st{f}{*(\mathfrak{p})}_{\beta})$ and apply Lemma \ref{L723} to each of the $\mathfrak{p}$-stabilizations. 
\end{proof}

\begin{lemma}
    We have $\langle H^{\dagger, \leq a}\zeta^*(\nabla^t\st{g}{}_{x}), \st{f}{*}_y\rangle/\langle f^*_y, f^*_y\rangle = \langle \zeta^*(\delta^t\st{g}{}_{x}), \st{f}{*}_y\rangle/\langle f^*_y, f^*_y\rangle$ upon base change to $\C{}$.
\end{lemma}

\begin{proof}
    By \cite[Lemma 2.5]{DarmonRotger}, we have $\langle \zeta^*(\delta^t\st{g}{}_{x}), \st{f}{*}_y\rangle = \langle \Pi^{\text{hol}}\zeta^*(\delta^t\st{g}{}_{x}), \st{f}{*}_y\rangle$
    where $\Pi^{\text{hol}}$ is the holomorphic projection of the nearly holomorphic modular form $\zeta^*(\delta^t \breve{g}_x)$. But the holomorphic projection is by definition the same thing as the overconvergent projection of $\zeta^*(\nabla^t \breve{g}_x)$ \cite[Proposition 2.8]{DarmonRotger}. The Petersson inner product cuts out the slope $\leq a$ part since $\breve{f}^*_y$ has slope $\leq a$. The lemma then follows from Proposition \ref{P712}, which shows that the pairing defined algebraically is a scalar multiple of the Petersson inner product.
\end{proof}

We can now collect all the results above to state the following theorem.

\begin{theorem}
    With all the assumptions as above, the value of the $p$-adic $L$-function is related to the classical central $L$-value in the following manner.
    \begin{gather*}
    \begin{aligned}
    &\sh{L}^g_p(\breve{\omega}_g, \breve{\omega}_f)((x,c), (y,2c)) \\
    &= \left(\prod_{\mathfrak{p} \text{ inert}}\left[ \frac{\sh{E}_{\mathfrak{p}}(g_x, \alpha_{f}^*)\sh{E}_1(\beta_f^*)}{E(\beta_f,\alpha_f)} + \frac{\sh{E}_{\mathfrak{p}}(g_x,\beta_f^*)\sh{E}_1(\alpha_f^*)}{E(\alpha_f, \beta_f)}\right] \right) \\ 
    &\times \left(\prod_{\mathfrak{p} \text{ split}} \left[\frac{\sh{E}_{\mathfrak{p}}(g_x, \alpha_f^*)\sh{E}_1(\beta_f^*)}{E(\beta_f,\alpha_f)\sh{E}_{0,\mathfrak{p}}(g_x,\alpha_f^*)} + \frac{\sh{E}_{\mathfrak{p}}(g_x,\beta_f^*)\sh{E}_1(\alpha_f^*)}{E(\alpha_f,\beta_f)\sh{E}_{0,\mathfrak{p}}(g_x, \beta_f^*)} \right] \right) \times \frac{\langle \zeta^*(\delta^t\st{g}{}_{x}), \st{f}{*}_y\rangle}{\langle f^*_y, f^*_y\rangle}.
    \end{aligned}
    \end{gather*}
\end{theorem}